\documentclass[12pt,reqno]{amsart}
\usepackage[utf8]{inputenc}
\usepackage[T1]{fontenc}% Pour pouvoir taper les accent directement et non pas passer par
\usepackage[usenames, dvipsnames]{color}
\usepackage{ulem}

\usepackage{dsfont, amsfonts, amsmath, amssymb,amscd, stmaryrd, latexsym, amsthm, dsfont}
\usepackage[frenchb,english]{babel}
\usepackage{enumerate}
\usepackage{longtable}
\usepackage{geometry}
\usepackage{float}
\usepackage{tikz}
\usetikzlibrary{shapes,arrows}
\usepackage{pifont}
\usepackage{float}
\usepackage{tikz}
\usepackage{xypic}
\usetikzlibrary{shapes,arrows}
\usepackage{multirow}
\usepackage{multicol}
\newtheorem{theorem}{Theorem}[section]
\newtheorem{lemma}[theorem]{Lemma}

\theoremstyle{remark}
\newtheorem{remark}[theorem]{\bf Remark}

\usepackage{hyperref}
\hypersetup{
	colorlinks=true,
	urlcolor=blue,
	citecolor=blue}
\def\NN{\mathds{N}}
\def\RR{\mathbb{R}}
\def\QQ{\mathbb{Q}}
\def\CC{\mathbb{C}}
\def\ZZ{\mathbb{Z}}
\def\kk{\mathds{k}}
\usepackage{mathtools}
\usepackage{mleftright}

\begin{document}
	
	\def\NN{\mathbb{N}}
	\def\RR{\mathds{R}}
	\def\HH{I\!\! H}
	\def\QQ{\mathbb{Q}}
	\def\CC{\mathds{C}}
	\def\ZZ{\mathbb{Z}}
	\def\DD{\mathds{D}}
	\def\OO{\mathcal{O}}
	\def\kk{\mathds{k}}
	\def\KK{\mathbb{K}}
	\def\ho{\mathcal{H}_0^{\frac{h(d)}{2}}}
	\def\LL{\mathbb{L}}
	\def\L{\mathds{k}_2^{(2)}}
	\def\M{\mathds{k}_2^{(1)}}
	\def\k{\mathds{k}^{(*)}}
	\def\l{\mathds{L}}

	\selectlanguage{english}
	%\title[ The unit groups of real triquadratic   fields ]{ The unit groups of real triquadratic number fields and applications}
	\title[ Arithmetic  of some real triquadratic   fields... ]{ Arithmetic  of some real triquadratic   fields; Units and 2-class groups}
	%premier auteur
	
	\author[M. M. Chems-Eddin]{Mohamed Mahmoud Chems-Eddin}
	\address{Mohamed Mahmoud Chems-Eddin: Sidi Mohamed Ben Abdellah University,   
		Faculty of Sciences Dhar El Mahraz, Departement of Mathematics, Fez, Morocco }
	\email{2m.chemseddin@gmail.com}

	\subjclass[2010]{11R04, 11R27, 11R29, 11R37.}
	\keywords{Real multiquadratic   fields,  unit group, $2$-class group, $2$-class  number.}

	\begin{abstract}
		In this paper, we   compute the unit groups and the $2$-class numbers of the Fr\"ohlich's    triquadratic fields $\KK=\mathbb{Q}(\sqrt{2},\sqrt{p},\sqrt{q})$,
		where  $p$  and $q$  are   two prime  numbers such that ($p\equiv 1 \pmod8$ and $q\equiv 3 \pmod4$) or  ($p\equiv 5$ or $3 \pmod8$  and $q\equiv 3 \pmod4$). Furthermore,   we determine some families of  the fields $\KK$ whose $2$-class groups are trivial or cyclic non trivial, and some other families with $2$-class groups   isomorphic to the Klein group.
	\end{abstract}
	
	\selectlanguage{english}
	
	\maketitle
	%	\tableofcontents

	\section{\bf Introduction}
	A Fr\"ohlich multiquadratic    field of degree $2^n$ is a real multiquadratic field of the form $F_n=\mathbb{Q}(\sqrt{p_1},\sqrt{p_2},...,\sqrt{p_n})$, where the $p_i$'s are prime numbers.
	These fields are of major interest in   class field theory  and genus theory of quadratic and biquadratic fields. Their study has a long history, and here we shall quote   some 
	works which are related to the subject of this paper.  
	In the best of our knowledge, when $n\geq 3$, all the facts that we have about the class groups of these  fields concern the cyclicity of their $2$-class groups and the parity of their class numbers. For example, in \cite{FrohlichCentral},
	Fr\"ohlich  showed that if more than four finite primes are ramified in a finite extension $K/\mathbb{Q}$, then the class number of $K$ is even and therefore $F_n$, with $n\geq 5$ 
	has an even class number. The parity of the class number of the quadratic field (i.e. $F_1$) can be determined using genus theory. The biquadratic field (i.e. $F_2$) was studied by 
	Fr\"ohlich \cite{FrohlichCentral}, Conner and  Hurrelbrink \cite{connor88} and  Ku{\v{c}}era \cite{kuvcera1995parity}. The parity of the class numbers of Fr\"ohlich  fields of degree $8$  (i.e. $F_3$) was studied by Bulant \cite{Bulant} who used the method of Ku{\v{c}}era which is based on circular units. Furthermore,  the authors of \cite{MouhibMouvahhedi} determined a list of the
	fields $F_3$ with $p_i\equiv 3\pmod 4$  whose $2$-class groups are cyclic non trivial. Finally, the parity of the class number of $F_4$, was   investigated  in   \cite{MouhibParity}. 
	We believe that after   this  list of interesting works, it  is  time to go further and discover more  and different arithmetical properties of these fields. 
	
	In the present paper, we provide the unit groups and the  $2$-class numbers  of the Fr\"ohlich field $F_3:=\KK=\mathbb{Q}(\sqrt{2},\sqrt{p},\sqrt{q})$, where $p$ and $q$ are two prime numbers such that ($p\equiv 1 \pmod8$ and $q\equiv 3 \pmod4$) or  ($p\equiv 5$ or $3 \pmod8$  and $q\equiv 3  \pmod4$). Furthermore, we shall give some   families  of   $\KK$ with $2$-class groups   of type $(2,2)$. Note that the reason behind choosing this form comes from our expertise and previous studies which showed the importance of these fields in the study of many problems of class field theory and genus theory related to biquadratic and triquadratic fields \cite{chemszekhniniaziziUnits1,ChemsUnits9}. Note also that the fields $\KK$ represent the first step of the cyclotomic $\mathbb{Z}_2$-extension of the fields $\mathbb{Q}( \sqrt{p},\sqrt{q})$ and our results may also be 
	very useful for studying some problems related to Iwasawa theory on biquadratic and triquadratic fields 
	%the fields $\mathbb{Q}( \sqrt{p},\sqrt{q})$ and $\mathbb{Q}( \sqrt{p},\sqrt{q}, i)$ which we shall expose in a forthcoming paper
	(see \cite[Theorem 3.6]{ChemsUnits9} for a  direct example of such applications).
	
	The plan of this paper is as follows; In Sec. 2, we   collect some preliminary results   which we shall use  later. In Sec. 3, we provide unit groups and $2$-class numbers of the Fr\"ohlich fields $ \KK=\mathbb{Q}(\sqrt{2},\sqrt{p},\sqrt{q})$. Therein we give some families whose $2$-class groups are trivial or cyclic non trivial. In the last section, we provide some families of Fr\"ohlich fields whose  $2$-class groups are of type $(2,2)$.
	\section*{Notations}
	Let $k$ be a number field. We shall use  the following notations for the rest of this paper:
	
	\begin{enumerate}[$\star$]

		\item $E_k$: The unit group of $k$,
		
		\item $q(k)=(E_{k}: \prod_{i}E_{k_i})$ is the unit index of $k$, if $k$ is multiquadratic, where   $k_i$ are  the  quadratic subfields	of $k$,
		\item $h(k)$: The  class number of  $k$,
		\item $h_2(k)$: The $2$-class number of $k$,
		\item $h_2(d)$: The $2$-class number of a quadratic field $\mathbb{Q}(\sqrt{d})$,
		\item $\varepsilon_d$: The fundamental unit of a real quadratic field $\mathbb{Q}(\sqrt{d})$,
		%	\item $x$, $y$: Two integers such that 	$ \varepsilon_{2pq}=x+y\sqrt{2pq}$,
		%\item $v$, $w$: Two integers such that 	$ \varepsilon_{pq}=v+w\sqrt{pq}$,
		\item $N(\varepsilon_d)$: The norm of   $\varepsilon_d$ in the extension $\mathbb{Q}(\sqrt{d})/\mathbb{Q}$,
		\item $\tau_i$: Defined in Page \pageref{fsu preparations},
		\item$k_i$: Defined in Page \pageref{fsu preparations},
		\item $u$: Defined in Lemma \ref{lm noms esp_2p},
		\item $\left(\frac{\cdot}{\cdot}\right)$: The Legendre symbol.
		%	\item $N_{k'/k}$: The norm map of an extension $k'/k$.
	\end{enumerate}
	\section{\bf Preparations} \label{sec2prep}	
	Let us start this section by  recalling the method given in    \cite{wada}, that describes a fundamental system  of units of a real  multiquadratic field $K_0$. Let  $\sigma_1$ and 
	$\sigma_2$ be two distinct elements of order $2$ of the Galois group of $K_0/\mathbb{Q}$. Let $K_1$, $K_2$ and $K_3$ be the three subextensions of $K_0$ invariant by  $\sigma_1$,
	$\sigma_2$ and $\sigma_3= \sigma_1\sigma_2$, respectively. Let $\varepsilon$ denote a unit of $K_0$. Then \label{algo wada}
	$$\varepsilon^2=\varepsilon\varepsilon^{\sigma_1}  \varepsilon\varepsilon^{\sigma_2}(\varepsilon^{\sigma_1}\varepsilon^{\sigma_2})^{-1},$$
	and we have, $\varepsilon\varepsilon^{\sigma_1}\in E_{K_1}$, $\varepsilon\varepsilon^{\sigma_2}\in E_{K_2}$  and $\varepsilon^{\sigma_1}\varepsilon^{\sigma_2}\in E_{K_3}$.
	It follows that the unit group of $K_0$  
	is generated by the elements of  $E_{K_1}$, $E_{K_2}$ and $E_{K_3}$, and the square roots of elements of   $E_{K_1}E_{K_2}E_{K_3}$ which are perfect squares in $K_0$.
	
	This method is very useful for computing a fundamental system of units of a real biquadratic number field, however, in the case of a real triquadratic 
	number field the problem of the determination of the unit group becomes very difficult and demands some specific computations and eliminations, as we will see in the next section.
	We shall consider the field $\KK=\mathbb{Q}(\sqrt{2},\sqrt{p},\sqrt{q})$, where $p$ and $q$ are two distinct prime numbers. Thus, we have the following diagram:

	\begin{figure}[H]
		$$\xymatrix@R=0.8cm@C=0.3cm{
			&\KK=\QQ( \sqrt 2, \sqrt{p}, \sqrt{q})\ar@{<-}[d] \ar@{<-}[dr] \ar@{<-}[ld] \\
			k_1=\QQ(\sqrt 2,\sqrt{p})\ar@{<-}[dr]& k_2=\QQ(\sqrt 2, \sqrt{q}) \ar@{<-}[d]& k_3=\QQ(\sqrt 2, \sqrt{pq})\ar@{<-}[ld]\\
			&\QQ(\sqrt 2)}$$
		\caption{Intermediate fields of $\KK/\QQ(\sqrt 2)$}\label{fig:I}
	\end{figure}
	Let $\tau_1$, $\tau_2$ and $\tau_3$ be the elements of  $ \mathrm{Gal}(\KK/\QQ)$ defined by
	\begin{center}	\begin{tabular}{l l l }
			$\tau_1(\sqrt{2})=-\sqrt{2}$, \qquad & $\tau_1(\sqrt{p})=\sqrt{p}$, \qquad & $\tau_1(\sqrt{q})=\sqrt{q},$\\
			$\tau_2(\sqrt{2})=\sqrt{2}$, \qquad & $\tau_2(\sqrt{p})=-\sqrt{p}$, \qquad &  $\tau_2(\sqrt{q})=\sqrt{q},$\\
			$\tau_3(\sqrt{2})=\sqrt{2}$, \qquad &$\tau_3(\sqrt{p})=\sqrt{p}$, \qquad & $\tau_3(\sqrt{q})=-\sqrt{q}.$
		\end{tabular}
	\end{center}
	Note that  $\mathrm{Gal}(\KK/\QQ)=\langle \tau_1, \tau_2, \tau_3\rangle$
	and the subfields  $k_1$, $k_2$ and $k_3$ are
	fixed by  $\langle \tau_3\rangle$, $\langle\tau_2\rangle$ and $\langle\tau_2\tau_3\rangle$ respectively. Therefore,\label{fsu preparations} a fundamental system of units  of $\KK$ consists  of seven  units chosen from those of $k_1$, $k_2$ and $k_3$, and  from the square roots of the elements of $E_{k_1}E_{k_2}E_{k_3}$ which are squares in $\KK$.   
	With these   notations, we have:

	\begin{lemma}\label{lm noms esp_2p}
		Let $p$ be a prime number such that $N(\varepsilon_{2p})=1$. Put $\varepsilon_{2p}=\beta+\alpha\sqrt{2p}$ with $\beta , \alpha\in\ZZ$.  Then  $\sqrt{\varepsilon_{2p}}=\frac{1}{\sqrt2}(\alpha_1+\alpha_2\sqrt{ 2p})$, for some integers $\alpha_1, \alpha_2$ such that $\alpha=\alpha_1 \alpha_2$. It follows that: 
		\begin{eqnarray}\label{T_3_-1_eqi2p_N=1} 
			\begin{tabular}{ |c|c|c|c|c|c|c|c|c|}
				\hline
				$  \sigma $               &$1+\tau_2$         &$1+\tau_1\tau_2$    &$1+\tau_1\tau_3$         &$1+\tau_2\tau_3$&    $1+\tau_1$      \\
				\hline
				$\sqrt{\varepsilon_{2p}}^\sigma$&$(-1)^{u}$           &$-\varepsilon_{2p}$                &	$(-1)^{u+1}$                 &$(-1)^u$  &   $(-1)^{u+1}$                     \\
				\hline
			\end{tabular} 
		\end{eqnarray}
		for some $u$   in $\{0, 1\}$ such that $\frac 12(\alpha_1^2-2p\alpha_2^2)=(-1)^u$\label{the int u}.	With   ${\varepsilon}^{\sigma+\tau}:=\sigma({\varepsilon})\tau({\varepsilon})$, for any $\sigma$, $\tau\in \mathrm{Gal}(\KK/\QQ)$ and for any $\varepsilon\in \KK$.
		
		%	********** to add the case N(epq)=-1.****
	\end{lemma}	
	
	To prove this lemma we need to recall the following :
	\begin{lemma}[\cite{Az-00},  {Lemma 5}]\label{1:046}
		Let $d>1$ be a square-free integer and $\varepsilon_d=x+y\sqrt d$,
		where $x$,  $y$ are  integers or semi-integers. If $N(\varepsilon_d)=1$,  then $2(x+1)$,  $2(x-1)$,  $2d(x+1)$ and
		$2d(x-1)$ are not squares in  $\QQ$.
	\end{lemma}

	\begin{proof}[Proof of Lemma \ref{lm noms esp_2p}]
		As $N(\varepsilon_{2p})=1$, then $\beta^2-1=\alpha^22p$. So by   Lemma \ref{1:046}, we have % ($\beta\pm1= \alpha_1^22$ and  $\beta\mp1= \alpha_2^2p$) or
		($\beta\mp1= \alpha_1^2 $  and  $\beta\pm1=  \alpha_2^22p$) for some integers $\alpha_1$ and $\alpha_2$. Thus, $2\beta=  \alpha_1^2+\alpha_2^22p$ and     $\frac 12(\alpha_1^2-2p\alpha_2^2)=(-1)^u$, for some $u$   in $\{0, 1\}$. Therefore, $2\varepsilon_{2p}=2\beta+2\alpha\sqrt{2p}=\alpha_1^22p+\alpha_2^2+2\alpha_1^2+\alpha_2^22p=(\alpha_1+\alpha_2\sqrt{2p})^2$. The reader can deduce the rest easily.
	\end{proof}

	\begin{lemma}[\cite{azizitalbi}, Theorem  6] \label{units of k1}Let $p\equiv 1\pmod 4$ be a prime number.  
		\begin{enumerate}[\rm 1.]
			\item  If $N(\varepsilon_{2p})=-1$, then $\{\varepsilon_{2}, \varepsilon_{p},	\sqrt{\varepsilon_{2}\varepsilon_{p}\varepsilon_{2p}}\}$ is a fundamental system of units
			of $k_1=\QQ(\sqrt 2,\sqrt{p})$.
			\item  If $N(\varepsilon_{2p})=1$, then $\{\varepsilon_{2}, \varepsilon_{p},	\sqrt{\varepsilon_{2p}}\}$ is a fundamental system of units
			of $k_1=\QQ(\sqrt 2,\sqrt{p})$.
		\end{enumerate}
	\end{lemma}

	\begin{lemma}\label{lm expressions of units under cond 1}
		Let   $p\equiv1\pmod 8$  and  $ q \equiv3\pmod8$  be two primes such that
		$\genfrac(){}{0}{p}{q} =-1$.
		\begin{enumerate}[\rm 1.]
			\item Let  $x$ and $y$   be two integers such that
			$ \varepsilon_{2pq}=x+y\sqrt{2pq}$. Then   
			\begin{enumerate}[\rm i.]
				\item $(x-1)$ is a square in $\NN$, 
				\item $\sqrt{2\varepsilon_{2pq}}=y_1+y_2\sqrt{2pq}$ and 	$2= -y_1^2+2pqy_2^2$, for some integers $y_1$ and $y_2$.
			\end{enumerate}
			% $p(x-1)$ is a square in $\NN$, $\sqrt{2\varepsilon_{2pq}}=y_1\sqrt{p}+y_2\sqrt{2q}$ and 	$2=2qy_2^2-py_1^2$, for some integers $y_1$ and $y_2$   such that $y=y_1y_2$.
			
			\item  Let    $v$ and $w$ be two integers such that
			$ \varepsilon_{pq}=v+w\sqrt{pq}$. Then    
			\begin{enumerate}[\rm i.]
				\item $(v-1)$ is a square in $\NN$, 
				\item   $\sqrt{ 2\varepsilon_{ pq}}=w_1+w_2\sqrt{pq}$ and 	$2=-w_1^2+pqw_2^2$, for some integers $w_1$ and $w_2$.
			\end{enumerate}
			
		\end{enumerate}
	\end{lemma}
	\begin{proof}
		It is known that $N(\varepsilon_{2pq})=1$. Then, by the unique factorization in $\mathbb{Z}$ and Lemma \ref{1:046}  there exist some integers $y_1$ and $y_2$  $(y=y_1y_2)$ such that
		$$(1):\ \left\{ \begin{array}{ll}
			x\pm1=y_1^2\\
			x\mp1=2pqy_2^2,
		\end{array}\right. \quad
		(2):\ \left\{ \begin{array}{ll}
			x\pm1=py_1^2\\
			x\mp1=2qy_2^2,
		\end{array}\right.\quad
		\text{ or }\quad
		(3):\ \left\{ \begin{array}{ll}
			x\pm1=2py_1^2\\
			x\mp1=qy_2^2,
		\end{array}\right.
		$$		
		\begin{enumerate}[\rm$*$]
			\item System $(2)$ can not occur since it implies  
			$-1=\left(\frac{2qy_2^2}{p}\right)=\left(\frac{x\mp 1}{p}\right)=\left(\frac{x\pm1\mp 2}{p}\right)=\left(\frac{\mp2}{p}\right)=\left(\frac{2}{p}\right)=1$, which is absurd. 
			\item We similarly show that  System $(3)$ and $\left\{ \begin{array}{ll}
				x+1=y_1^2\\
				x-1=2pqy_2^2
			\end{array}\right.$ 
			can not occur.\\
		\end{enumerate}	
		Therefore $\left\{ \begin{array}{ll}
			x-1=y_1^2\\
			x+1=2pqy_2^2
		\end{array}\right.$ which gives the first item. The proof of the second item is analogous.
	\end{proof}
	
	\begin{lemma}\label{lm expressions of units under cond 1 eps_q, 2q}
		Let    $ q \equiv3\pmod8$ be a prime number.  
		\begin{enumerate}[\rm 1.]
			\item  Let    $c $ and $d$ be two integers such that
			$ \varepsilon_{2q}=c +d\sqrt{2q}$. Then    
			\begin{enumerate}[\rm i.]
				\item   $c-1$ is a square in $\NN$,
				\item  $\sqrt{ 2\varepsilon_{  2q}}=d_1 +d_2\sqrt{2q}$ and 	$2=-d_1^2+2qd_2^2$, for some integers $d_1$ and $d_2$.
			\end{enumerate}

			\item  Let    $\alpha $ and $\beta$ be two integers such that
			$ \varepsilon_{q}=\alpha +\beta\sqrt{q}$. Then   
			\begin{enumerate}[\rm i.]
				\item   $\alpha-1$ is a square in $\NN$,
				\item  $\sqrt{ 2\varepsilon_{  q}}=\beta_1 +\beta_2\sqrt{q}$ and 	$2=-\beta_1^2+q\beta_2^2$, for some integers $\beta_1$ and $\beta_2$.
			\end{enumerate}
		\end{enumerate}
		Furthermore, for any prime number $p\equiv 1\pmod 4$ we have: 
		%	\begin{eqnarray}
		\begin{table}[H] 
			\begin{tabular}{ |c|c|c|c|c|c|c|c|c|}
				\hline
				$\varepsilon$&$\varepsilon_{2}$ &$ {\varepsilon_{p}}$&$\sqrt{\varepsilon_{q}}$&$\sqrt{\varepsilon_{2q}}$          \\
				\hline 
				$\varepsilon^{1+\tau_1}$&$-1$ &$\varepsilon_{p}^2$ 	&$-\varepsilon_{q}$&$1$\\
				\hline
				$\varepsilon^{1+\tau_2}$&$\varepsilon_{2}^2$ &$-1$&	$\varepsilon_{q}$&$\varepsilon_{2q}$\\
				\hline
				$\varepsilon^{1+\tau_3}$&$\varepsilon_{2}^2$ &$\varepsilon_{p}^2$& $-1$&$-1$\\
				\hline
				$\varepsilon^{1+\tau_1\tau_2}$&$-1$ &$-1$&	$-\varepsilon_{q}$&$1$\\
				\hline
				$\varepsilon^{1+\tau_1\tau_3}$&$-1$ &$\varepsilon_{p}^2$&	$1$&$-\varepsilon_{2q}$\\
				\hline
				$\varepsilon^{1+\tau_2\tau_3}$&$\varepsilon_{2}^2$ &$-1$&	$-1$&$-1$\\
				\hline
			\end{tabular} 
			\vspace*{0.2cm}
			\caption{Norm maps on units  }\label{norm q=3 p=1mod4q}
		\end{table}
		%	\end{eqnarray}
	\end{lemma}	
	\begin{proof}
		Similar to that of	Lemma  \ref{lm expressions of units under cond 1}.
	\end{proof}

	\begin{lemma}\label{lm expressions of units under cond 3}
		Let        $p\equiv1\pmod 8$ and $ q \equiv3\pmod8$ be two primes such that
		$\genfrac(){}{0}{p}{q} =1$.
		\begin{enumerate}[\rm 1.]
			\item Let  $x$ and $y$   be two integers such that
			$ \varepsilon_{2pq}=x+y\sqrt{2pq}$. Then   
			\begin{enumerate}[\rm i.]
				\item $(x-1)$, $p(x-1)$ or $2p(x+1)$ is a square in $\NN$, 
				\item Furthermore, 
				\begin{enumerate}[\rm a)]
					\item If $(x-1)$ is a square in $\NN$, then $\sqrt{2\varepsilon_{2pq}}=y_1+y_2\sqrt{2pq}$ and 	$2= -y_1^2+2pqy_2^2$.
					\item If $p(x-1)$ is a square in $\NN$, then $\sqrt{2\varepsilon_{2pq}}=y_1\sqrt{p}+y_2\sqrt{2q}$ and $2= -py_1^2+2qy_2^2$.
					\item If $2p(x+1)$ is a square in $\NN$, then $\sqrt{2\varepsilon_{2pq}}=y_1\sqrt{2p}+y_2\sqrt{q}$ and $2=  2py_1^2-qy_2^2$. 
				\end{enumerate}
				
			\end{enumerate}
			\indent Where $y_1$ and $y_2$ are two integers  such that $y=y_1y_2$.
			\item  Let    $v$ and $w$ be two integers such that
			$ \varepsilon_{pq}=v+w\sqrt{pq}$. Then   
			\begin{enumerate}[\rm i.]
				\item $(v-1)$, $p(v-1)$ or $2p(v+1)$ is a square in $\NN$, 
				\item  Furthermore, 
				\begin{enumerate}[\rm a)]
					\item If $(v-1)$ is a square in $\NN$, then $\sqrt{2\varepsilon_{pq}}=w_1+w_2\sqrt{pq}$ and 	$2= -w_1^2+pqw_2^2$.
					\item If $p(v-1)$ is a square in $\NN$, then $\sqrt{2\varepsilon_{pq}}=w_1\sqrt{p}+w_2\sqrt{q}$ and $2= -pw_1^2+qw_2^2$.
					\item If $2p(v+1)$ is a square in $\NN$, then $\sqrt{\varepsilon_{pq}}=w_1\sqrt{p}+w_2\sqrt{q}$ and $1= pw_1^2-qw_2^2$. 
				\end{enumerate}
				
			\end{enumerate}	
			Where $w_1$ and $w_2$ are two integers  such that $w=w_1w_2$ in $a)$ and $b)$, and  $w=2w_1w_2$ in $c)$.
		\end{enumerate}
	\end{lemma}	
	\begin{proof}
		We proceed as in the proof of	Lemma \ref{lm expressions of units under cond 1}.
	\end{proof}

	Now we recall the following lemmas:

	% 	\begin{lemma}[\cite{ZAT-15}]\label{3:105}
	% 		Let $d\equiv1\pmod4$ be a positive square free integer and   $\varepsilon_d=x+y\sqrt d$ be the fundamental unit of  $\QQ(\sqrt d)$. Assume   $N(\varepsilon_d)=1$,  then
	%		\begin{enumerate}[\rm\indent1.]
	%			\item $x+1$ and $x-1$ are not squares in  $\NN$,  i.e.,  $2\varepsilon_{d}$ is not a square in  $\QQ(\sqrt{d})$.
	%			\item For all prime  $p$ dividing   $d$,  $p(x+1)$ and $p(x-1)$ are not squares in  $\NN$.
	%		\end{enumerate}
	%	\end{lemma}

	% Let us close this section by recalling the   following class number formula for a multiquadratic number field  which is usually attributed to Kuroda \cite{Ku-50}, but it goes back to Herglotz \cite{He-22}  (cf. \cite[p. 27]{BFL}).
	\begin{lemma}[\cite{Ku-50}]\label{wada's f.}
		Let $K$ be a multiquadratic number field of degree $2^n$, $n\in\mathds{N}$,  and $k_i$ the $s=2^n-1$ quadratic subfields of $K$. Then
		$$h(K)=\frac{1}{2^v}( E_K: \prod_{i=1}^{s}E_{k_i})\prod_{i=1}^{s}h(k_i),$$
		with $$v=\left\{ \begin{array}{cl}
			n(2^{n-1}-1); &\text{ if } K \text{ is real, }\\
			(n-1)(2^{n-2}-1)+2^{n-1}-1 & \text{ if } K \text{ is imaginary.}
		\end{array}\right.$$
	\end{lemma}

	\begin{lemma}\label{class numbers of quadratic field}
		Let $q\equiv 3\pmod 4$ and  $p\equiv 1\pmod 4$ be two    primes. 
		\begin{enumerate}[\rm 1.]
			\item By \cite[Corollary 18.4]{connor88}, we have  $h_2(p)=h_2(q)=h_2(2q)= h_2(2)=h_2(-2)=h_2(-q)=h_2(-1)=1$.
			\item If $\genfrac(){}{0}{p}{q} =-1$,   then   $h_2(pq)=h_2(2pq)=h_2(-pq)=2$, else $h_2(pq)$, $h_2(2pq)$ and $h_2(-pq)$ are divisible by $4$ $($cf. \cite[Corollaries 19.6 and 19.7]{connor88}$)$.
			\item If $ q\equiv 3\pmod 8$, then $h_2(-2q)=2$ $($cf. \cite[Corollary 19.6]{connor88}$)$.
			%	\item If $p\equiv  -q\equiv 5\pmod 8$, then $h_2(-p)=h_2(-2p)=h_2(-2q)=2$.
		\end{enumerate}
	\end{lemma}

	\section{\bf Unit groups of real triquadratic number fields and their $2$-class numbers} 		
	
	Keep the notations in the above section. In this section we shall compute the unit groups and the $2$-class numbers of the Fr\"ohlich fields $\KK$.

	\subsection{\bf The case: $p\equiv 1 \pmod8$,  $q\equiv 3 \pmod4$ and $\genfrac(){}{0}{p}{q} =-1$}$\,$ 			
	
	We shall now state and prove our first main theorem.

	\begin{theorem}\label{T_3_-1} Let $p\equiv 1\pmod{8}$ and $q\equiv3\pmod 8$ be two primes such that     $\genfrac(){}{0}{p}{q} =-1$. 
		Put     $\KK=\QQ(\sqrt 2, \sqrt{p}, \sqrt{q} )$. Then
		\begin{enumerate}[\rm 1.]
			\item If $N(\varepsilon_{2p})=-1$,   we have
			\begin{enumerate}[\rm $\bullet$]
				
				\item The unit group of $\KK$ is :
				$$E_{\KK}=\langle -1,  \varepsilon_{2}, \varepsilon_{p},   \sqrt{\varepsilon_{q}}, \sqrt{\varepsilon_{2q}},  \sqrt{\varepsilon_{pq}} , \sqrt{\varepsilon_{2}\varepsilon_{p}\varepsilon_{2p}}, 
				\sqrt{\sqrt{\varepsilon_{q}} \sqrt{\varepsilon_{2q}} \sqrt{\varepsilon_{pq}} \sqrt{\varepsilon_{2pq}}}\rangle$$
				\item The $2$-class group of  $\KK$ is cyclic of order  $\frac 12h_2(2p).$
			\end{enumerate}
			\item If $N(\varepsilon_{2p})=1$,   we have
			
			\begin{enumerate}[\rm $\bullet$]
				\item  The unit group of $\KK$ is :
				$$E_{\KK}=\langle -1,  \varepsilon_{2}, \varepsilon_{p},   \sqrt{\varepsilon_{q}}, \sqrt{\varepsilon_{2q}},  \sqrt{\varepsilon_{pq}} , \sqrt{\varepsilon_{2}^a\varepsilon_{p}^a \sqrt{\varepsilon_{q}} \sqrt{\varepsilon_{pq}}
					\sqrt{\varepsilon_{2p}}},
				\sqrt{\varepsilon_{2}^a\varepsilon_{p}^a  \sqrt{\varepsilon_{2q}} \sqrt{\varepsilon_{2pq}}
					\sqrt{\varepsilon_{2p}}}      \rangle,$$
				where $a\in\{0,1\}$ such that $a\equiv 1+u \pmod 2$.
				\item  The $2$-class group of  $\KK$ is cyclic of order  $h_2(2p)$.
			\end{enumerate}
		\end{enumerate}
	\end{theorem}
	\begin{proof}
		We shall use the method and  the preparations exposed in  Section \ref{sec2prep}. Therefore, we need the unit groups of the intermediate fields $k_1$, $k_2$ and $k_3$.
		\begin{enumerate}[\rm 1.]
			\item    Assume that   $N(\varepsilon_{2p})=-1$.  By Lemma \ref{units of k1},     $\{\varepsilon_{2}, \varepsilon_{p},	\sqrt{\varepsilon_{2}\varepsilon_{p}\varepsilon_{2p}}\}$  is a  fundamental system of units of $k_1$. One can easily deduce from Lemmas \ref{lm expressions of units under cond 1 eps_q, 2q} and \ref{lm expressions of units under cond 1} that $\{\varepsilon_{2}, \sqrt{\varepsilon_{q}}, \sqrt{\varepsilon_{2q}}\}$ and $\{ \varepsilon_{2}, 	\sqrt{\varepsilon_{pq}},\sqrt{\varepsilon_{2pq}}\}$ are respectively fundamental systems of units of $k_2$ and $k_3$.
			It follows that,  	$$E_{k_1}E_{k_2}E_{k_3}=\langle-1,  \varepsilon_{2}, \varepsilon_{p},   \sqrt{\varepsilon_{q}}, \sqrt{\varepsilon_{2q}},  \sqrt{\varepsilon_{pq}} ,\sqrt{ \varepsilon_{2pq}}, \sqrt{\varepsilon_{2}\varepsilon_{p}\varepsilon_{2p}}\rangle.$$	
			Let  $\xi$ be an element of $\KK$ which is the  square root of an element of $E_{k_1}E_{k_2}E_{k_3}$. Therefore, we can assume that
			$$\xi^2=\varepsilon_{2}^a\varepsilon_{p}^b \sqrt{\varepsilon_{q}}^c\sqrt{\varepsilon_{2q}}^d\sqrt{\varepsilon_{pq}}^e\sqrt{\varepsilon_{2pq}}^f
			\sqrt{\varepsilon_{2}\varepsilon_{p}\varepsilon_{2p}}^g,$$
			where $a, b, c, d, e, f$ and $g$ are in $\{0, 1\}$.\\ %Remark that the question now becomes whether the equation
			Remark that the question now becomes about the  solvability  in $\KK$   of the equation:
			$\xi^2-\varepsilon_{2}^a\varepsilon_{p}^b \sqrt{\varepsilon_{q}}^c\sqrt{\varepsilon_{2q}}^d\sqrt{\varepsilon_{pq}}^e\sqrt{\varepsilon_{2pq}}^f
			\sqrt{\varepsilon_{2}\varepsilon_{p}\varepsilon_{2p}}^g=0$. % has solutions in $\KK$ or not.
			Assuming that this equation has solutions in $\KK$, we shall firstly use norm maps 
			from $\KK$ to its  subextensions  to eliminate   the forms with do  not occur. 
			
			\noindent\ding{224}  Let us start	by applying   the norm map $N_{\KK/k_2}=1+\tau_2$. Using Lemma \ref{lm expressions of units under cond 1} we get 
			$\sqrt{\varepsilon_{2pq}}^{1+\tau_2}=(\frac{1}{\sqrt 2}(y_1+y_2\sqrt{2pq}))\times \tau_2(\frac{1}{\sqrt 2}(y_1+y_2\sqrt{2pq}))= (\frac{1}{\sqrt 2}(y_1+y_2\sqrt{2pq}))\times (\frac{1}{\sqrt 2}(y_1-y_2\sqrt{2pq}))=\frac 12(y_1^2-2pqy_2)=\frac 12(-2)=-1$. Similarly we have $\sqrt{\varepsilon_{pq}}^{1+\tau_2}=-1$. So by   Table \ref{norm q=3 p=1mod4q},   we have:
			%	\begin{eqnarray}\label{T_3_-1_tau2_N=-1}
			%	\begin{tabular}{ |c|c|c|c|c|c|c|c|c|}
			%			\hline
			%			$\varepsilon$&$\varepsilon_{2}$ &$ {\varepsilon_{p}}$&$\sqrt{\varepsilon_{q}}$&$\sqrt{\varepsilon_{2q}}$&$\sqrt{\varepsilon_{pq}}$&$\sqrt{\varepsilon_{2pq}}$&$\sqrt{\varepsilon_{2}\varepsilon_{p}\varepsilon_{2p}}$ \\
			%		\hline
			%			$\varepsilon^{1+\tau_2}$&$\varepsilon_{2}^2$ &$-1$&	$\varepsilon_{q}$&$\varepsilon_{2q}$&$ -1$& $-1$&$(-1)^v\varepsilon_{2}$ \\
			%		\hline
			%	\end{tabular} 
			%	\end{eqnarray}
			\begin{eqnarray*}
				N_{\KK/k_2}(\xi^2)&=&
				\varepsilon_{2}^{2a}(-1)^b \cdot \varepsilon_{q}^c\cdot \varepsilon_{2q}^d\cdot(-1)^e\cdot(-1)^f \cdot (-1)^{gs}\varepsilon_{2}^g\\
				&=&	\varepsilon_{2}^{2a}  \varepsilon_{q}^c\varepsilon_{2q}^d\cdot(-1)^{b+e+f+gs} \varepsilon_{2}^g.
			\end{eqnarray*}
			for some $s\in \{0,1\}$.	Thus, $b+e+f+gs\equiv 0\pmod2$.	Since $\varepsilon_{2}$ is not a square in $k_2$, then $g=0$. Therefore $b+e+f\equiv 0\pmod2$ and
			$$\xi^2=\varepsilon_{2}^a\varepsilon_{p}^b \sqrt{\varepsilon_{q}}^c\sqrt{\varepsilon_{2q}}^d\sqrt{\varepsilon_{pq}}^e\sqrt{\varepsilon_{2pq}}^f.$$

			\noindent\ding{224} Let us apply the norm $N_{\KK/k_5}=1+\tau_1\tau_2$, with $k_5=\QQ(\sqrt{q}, \sqrt{2p})$. We have 
			$\sqrt{\varepsilon_{pq}}^{1+\tau_1\tau_2}=1$ and
			$\sqrt{\varepsilon_{2pq}}^{1+\tau_1\tau_2}={-\varepsilon_{2pq}}$. Then, by Table \ref{norm q=3 p=1mod4q}, we get:
			%	\begin{eqnarray}\label{T_3_-1_tau1tau2_N=-1}
			%	\begin{tabular}{ |c|c|c|c|c|c|c|c|c|}
			%		\hline
			%		$\varepsilon$&$\varepsilon_{2}$ &$ {\varepsilon_{p}}$&$\sqrt{\varepsilon_{q}}$&$\sqrt{\varepsilon_{2q}}$&$\sqrt{\varepsilon_{pq}}$&$\sqrt{\varepsilon_{2pq}}$ \\
			%		\hline
			%		$\varepsilon^{1+\tau_1\tau_2}$&$-1$ &$-1$&	$-\varepsilon_{q}$&$1$&$ 1$& ${-\varepsilon_{2pq}}$   \\
			%		\hline
			%	\end{tabular} 
			%	\end{eqnarray}
			\begin{eqnarray*}
				N_{\KK/k_5}(\xi^2)&=&(-1)^a\cdot (-1)^b\cdot (-1)^c\cdot \varepsilon_{  q}^c\cdot1\cdot1\cdot (-1)^f\cdot \varepsilon_{  2pq}^f\\
				&=&	 (-1)^{a+b+c+f}\varepsilon_{  q}^c\cdot \varepsilon_{  2pq}^f.
			\end{eqnarray*}
			Thus $a+b+c+f=0\pmod 2$. By Lemmas \ref{lm expressions of units under cond 1 eps_q, 2q} and \ref{lm expressions of units under cond 1}, none of $\varepsilon_{  q}$ and $\varepsilon_{  2pq}$ is a square in $k_5$. Then    $f=c$. Thus, $a=b$. Therefore, 
			$$\xi^2=\varepsilon_{2}^a\varepsilon_{p}^a \sqrt{\varepsilon_{q}}^f\sqrt{\varepsilon_{2q}}^d\sqrt{\varepsilon_{pq}}^e\sqrt{\varepsilon_{2pq}}^f,$$

			\noindent\ding{224} Let us apply the norm $N_{\KK/k_6}=1+\tau_1\tau_3$, with $k_6=\QQ(\sqrt{p}, \sqrt{2q})$. We have 
			$\sqrt{\varepsilon_{pq}}^{1+\tau_1\tau_3}=1$ and $\sqrt{\varepsilon_{2pq}}^{1+\tau_1\tau_3}=-\varepsilon_{2pq}$. Then, by Table \ref{norm q=3 p=1mod4q}, we get:
			%	\begin{eqnarray}\label{T_3_-1_tau1tau3_N=-1}
			%		\begin{tabular}{ |c|c|c|c|c|c|c|c|c|}
			%		\hline
			%		$\varepsilon$&$\varepsilon_{2}$ &$ {\varepsilon_{p}}$&$\sqrt{\varepsilon_{q}}$&$\sqrt{\varepsilon_{2q}}$&$\sqrt{\varepsilon_{pq}}$&$\sqrt{\varepsilon_{2pq}}$ \\
			%		\hline
			%		$\varepsilon^{1+\tau_1\tau_3}$&$-1$ &$\varepsilon_{p}^2$&	$1$&$-\varepsilon_{2q}$&$ 1$& ${-\varepsilon_{2pq}}$   \\
			%		\hline
			%	\end{tabular} 
			%	\end{eqnarray}		
			\begin{eqnarray*}
				N_{\KK/k_6}(\xi^2)&=&(-1)^a\cdot \varepsilon_{  p}^{2a}\cdot 1\cdot (-1)^d\cdot \varepsilon_{2  q}^d\cdot 1\cdot(-1)^f\cdot \varepsilon_{  2pq}^f\\
				&=&	\varepsilon_{  p}^{2a} (-1)^{a+d+f}\varepsilon_{2  q}^d\varepsilon_{  2pq}^f.
			\end{eqnarray*}
			Thus $a+d+f=0\pmod 2$. Again by Lemmas \ref{lm expressions of units under cond 1 eps_q, 2q} and \ref{lm expressions of units under cond 1}, none of $\varepsilon_{  2q}$ and $\varepsilon_{  2pq}$ is a square in $k_6$. Then 
			$d=f$. Therefore $a=0$ and 
			$$\xi^2=  \sqrt{\varepsilon_{q}}^f\sqrt{\varepsilon_{2q}}^f\sqrt{\varepsilon_{pq}}^e\sqrt{\varepsilon_{2pq}}^f.$$

			\noindent\ding{224} Let us apply the norm $N_{\KK/k_3}=1+\tau_2\tau_3$, with $k_3=\QQ(\sqrt{2}, \sqrt{pq})$. Note that 
			$\sqrt{\varepsilon_{pq}}^{1+\tau_2\tau_3}=\varepsilon_{pq}$ and $\sqrt{\varepsilon_{2pq}}^{1+\tau_2\tau_3}=\varepsilon_{2pq}$. Then, by Table \ref{norm q=3 p=1mod4q}, we have:
			%	\begin{eqnarray}\label{T_3_-1_tau2tau3_N=-1}
			%	\begin{tabular}{ |c|c|c|c|c|c|c|c|c|}
			%		\hline
			%		$\varepsilon$&$\varepsilon_{2}$ &${\varepsilon_{p}}$&$\sqrt{\varepsilon_{q}}$&$\sqrt{\varepsilon_{2q}}$&$\sqrt{\varepsilon_{pq}}$&$\sqrt{\varepsilon_{2pq}}$ \\
			%		\hline
			%		$\varepsilon^{1+\tau_2\tau_3}$&$\varepsilon_{2}^2$ &$-1$&	$-1$&$-1$&$ {\varepsilon_{pq}}$& ${\varepsilon_{2pq}}$   \\
			%		\hline
			%	\end{tabular} 
			%	\end{eqnarray}
			\begin{eqnarray*}
				N_{\KK/k_3}(\xi^2)&=& (-1)^{f}\cdot   (-1)^f\cdot \varepsilon_{pq}^e \cdot \varepsilon_{  2pq}^f\\
				&=&  \varepsilon_{pq}^e \cdot \varepsilon_{  2pq}^f.
			\end{eqnarray*}
			By   Lemma  \ref{lm expressions of units under cond 1}, both $ \varepsilon_{pq}$ and $\varepsilon_{  2pq}$	are squares in $k_3$. So we deduce nothing.
			% Therefore,
			%	$$\xi^2=  \sqrt{\varepsilon_{q}}^f\sqrt{\varepsilon_{2q}}^f\sqrt{\varepsilon_{pq}}^f\sqrt{\varepsilon_{2pq}}^f.$$
			
			\noindent\ding{224} Let us apply the norm $N_{\KK/k_4}=1+\tau_1$, with $k_4=\QQ(\sqrt{p}, \sqrt{q})$. Note that 
			$\sqrt{\varepsilon_{pq}}^{1+\tau_1}=-\varepsilon_{pq}$ and $\sqrt{\varepsilon_{2pq}}^{1+\tau_1}=1$. Then, by Table \ref{norm q=3 p=1mod4q}, we have:
			\begin{eqnarray*}
				N_{\KK/k_4}(\xi^2)&=& (-1)^{f}\cdot \varepsilon_{q}^f\cdot1\cdot   (-1)^e\cdot \varepsilon_{pq}^e \cdot 1\\
				&=&  (-1)^{f+e} \varepsilon_{q}^f\varepsilon_{pq}^e.
			\end{eqnarray*}
			Thus, $f=e$ and 	
			$$\xi^2=  \sqrt{\varepsilon_{q}}^f\sqrt{\varepsilon_{2q}}^f\sqrt{\varepsilon_{pq}}^f\sqrt{\varepsilon_{2pq}}^f.$$

			Let us show that  the square root of $\sqrt{\varepsilon_{q}}\sqrt{\varepsilon_{2q}}\sqrt{\varepsilon_{pq}}\sqrt{\varepsilon_{2pq}}$ is an element of $\KK$.	
			Note that one can easily check that the  $2$-class group of $k_5=\mathbb Q(\sqrt{2p}, \sqrt{q})$ is cyclic and by   Lemmas  \ref{wada's f.} and \ref{class numbers of quadratic field}, we have   $h_2(k_5)=\frac{1}{4}q(k_5)h_2(2p)h_2(q)h_2(2pq)=\frac{1}{2}q(k_5)h_2(2p) $.
			Using Lemmas \ref{lm expressions of units under cond 1} and \ref{lm expressions of units under cond 1 eps_q, 2q}, we show that $q(k_5)=2$. Thus $h_2(k_5)=h_2(2p)$. Since $\KK/k_5$ is an unramified quadratic extension, then
			
			\begin{eqnarray}\label{class nbr of K+}
				h_2(\KK)=\frac{1}{2}\cdot h_2(k_5)=\frac{1}{2}\cdot h_2(2p) .
			\end{eqnarray}
			Assume by absurd that $\sqrt{\varepsilon_{q}}\sqrt{\varepsilon_{2q}}\sqrt{\varepsilon_{pq}}\sqrt{\varepsilon_{2pq}}$ is not a square in $\KK$. Then $q(\KK)=2^5$.	
			By Lemma  \ref{wada's f.}, we have:
			\begin{eqnarray}
				h_2(\KK)&=&\frac{1}{2^{9}}q(\KK)  h_2(2) h_2(p) h_2(q)h_2(2p) h_2(2q)h_2(pq)  h_2(2pq)  \\
				&=&\frac{1}{2^{9}}\cdot 2^5\cdot 1\cdot 1 \cdot 1  \cdot h_2(2p) \cdot 1 \cdot 2 \cdot 2=\frac{1}{4}\cdot h_2(2p).\nonumber
			\end{eqnarray}
			Which is a contradiction with	\eqref{class nbr of K+}. Therefore $f=1$ and $\sqrt{\varepsilon_{q}} \sqrt{\varepsilon_{2q}} \sqrt{\varepsilon_{pq}} \sqrt{\varepsilon_{2pq}}$\label{is a square} is a square in $\KK$.
			Hence, we have   
			$$E_{\KK}=\langle-1,  \varepsilon_{2}, \varepsilon_{p},   \sqrt{\varepsilon_{q}}, \sqrt{\varepsilon_{2q}},  \sqrt{\varepsilon_{pq}} , \sqrt{\varepsilon_{2}\varepsilon_{p}\varepsilon_{2p}}, 
			\sqrt{\sqrt{\varepsilon_{q}} \sqrt{\varepsilon_{2q}} \sqrt{\varepsilon_{pq}} \sqrt{\varepsilon_{2pq}}}\rangle.$$

			\item Let us now prove the second item. Assume that	$N(\varepsilon_{2p})=1$.  By Lemma \ref{units of k1},     $\{\varepsilon_{2}, \varepsilon_{p},	\sqrt{\varepsilon_{2p}}\}$  is a  fundamental system of units of $k_1$ and one can easily deduce from Lemmas \ref{lm expressions of units under cond 1 eps_q, 2q} and \ref{lm expressions of units under cond 1} that $\{\varepsilon_{2}, \sqrt{\varepsilon_{q}}, \sqrt{\varepsilon_{2q}}\}$ and $\{ \varepsilon_{2}, 	\sqrt{\varepsilon_{pq}},\sqrt{\varepsilon_{2pq}}\}$ are respectively fundamental systems of units of $k_2$ and $k_3$.
			So we have   	$$E_{k_1}E_{k_2}E_{k_3}=\langle-1,  \varepsilon_{2}, \varepsilon_{p},   \sqrt{\varepsilon_{q}}, \sqrt{\varepsilon_{2q}},  \sqrt{\varepsilon_{pq}} ,\sqrt{ \varepsilon_{2pq}}, \sqrt{\varepsilon_{2p}}\rangle.$$	
			Put
			$$\xi^2=\varepsilon_{2}^a\varepsilon_{p}^b \sqrt{\varepsilon_{q}}^c\sqrt{\varepsilon_{2q}}^d\sqrt{\varepsilon_{pq}}^e\sqrt{\varepsilon_{2pq}}^f
			\sqrt{\varepsilon_{2p}}^g,$$
			where $a, b, c, d, e, f$ and $g$ are in $\{0, 1\}$. Assume that $\xi$ belongs to $\KK$. We shall proceed as above, by using the norm maps from $\KK$ to its subextensions. Note that these norms are already computed in the proof of the first item, and we shall use \eqref{T_3_-1_eqi2p_N=1} for the norms of  $\sqrt{\varepsilon_{2p}}$.
			Let   $u$ be the integer defined in Lemma \ref{lm noms esp_2p}.	
			
			\noindent\ding{224}	Let us start	by applying   the norm map $N_{\KK/k_2}=1+\tau_2$. We have	
			\begin{eqnarray*}
				N_{\KK/k_2}(\xi^2)&=&
				\varepsilon_{2}^{2a}\cdot (-1)^b \cdot \varepsilon_{q}^c\cdot \varepsilon_{2q}^d\cdot(-1)^e\cdot(-1)^f \cdot (-1)^{gu}\\\
				&=&	\varepsilon_{2}^{2a}  \varepsilon_{q}^c\varepsilon_{2q}^d\cdot(-1)^{b+e+f+gu} .
			\end{eqnarray*}
			Thus,  $	b+e+f+gu\equiv 0\pmod2  $.
			
			\noindent\ding{224}	Let   us apply    the norm map $N_{\KK/k_5}=1+\tau_1\tau_2$, with $k_5=\QQ(\sqrt{q}, \sqrt{2p})$.	We have	
			\begin{eqnarray*}
				N_{\KK/k_5}(\xi^2)&=&
				(-1)^a  \cdot (-1)^b \cdot (-1)^c  \cdot\varepsilon_{q}^c\cdot1\cdot1 \cdot(-1)^f \cdot\varepsilon_{2pq}^f\cdot(-1)^g\cdot \varepsilon_{2p}^{g} \\\
				&=&  (-1)^{a+b+c+f+g } \cdot\varepsilon_{q}^c\varepsilon_{2pq}^f\varepsilon_{2p}^{g}.
			\end{eqnarray*}
			
			Thus, $a+b+c+f+g\equiv 0\pmod2$ and  $c+f+ g\equiv 0\pmod2$.	Therefore, $a=b$  and 
			$$\xi^2=\varepsilon_{2}^a\varepsilon_{p}^a \sqrt{\varepsilon_{q}}^c\sqrt{\varepsilon_{2q}}^d\sqrt{\varepsilon_{pq}}^e\sqrt{\varepsilon_{2pq}}^f
			\sqrt{\varepsilon_{2p}}^g.$$
			with $c+f+ g\equiv 0\pmod2$.

			\noindent\ding{224}	Let    us apply      the norm map $N_{\KK/k_6}=1+\tau_1\tau_3$, with $k_6=\QQ(\sqrt{p}, \sqrt{2q})$ .	We have	
			\begin{eqnarray*}
				N_{\KK/k_6}(\xi^2)&=&
				(-1)^a  \cdot \varepsilon_{p}^{2a} \cdot 1 \cdot(-1)^d \cdot\varepsilon_{2q}^d\cdot1  \cdot(-1)^f \cdot\varepsilon_{2pq}^f\cdot (-1)^{gu}\cdot (-1)^{g}\\\
				&=&\varepsilon_{p}^{2a}\cdot (-1)^{a + d+f+ug+g} \cdot\varepsilon_{2q}^d \varepsilon_{2pq}^f.
			\end{eqnarray*}
			Thus, $a + d+f+ug+g\equiv 0\pmod2$ and $d=f$. Therefore,   $a +ug+g\equiv 0\pmod2$ and
			
			$$\xi^2=\varepsilon_{2}^a\varepsilon_{p}^a \sqrt{\varepsilon_{q}}^c\sqrt{\varepsilon_{2q}}^f\sqrt{\varepsilon_{pq}}^e\sqrt{\varepsilon_{2pq}}^f
			\sqrt{\varepsilon_{2p}}^g.$$

			\noindent\ding{224}	Let   us apply      the norm map $N_{\KK/k_3}=1+\tau_2\tau_3$, with $k_3=\QQ(\sqrt{2}, \sqrt{pq})$.	We have		
			\begin{eqnarray*}
				N_{\KK/k_3}(\xi^2)&=&
				\varepsilon_{2}^{2a} \cdot (-1)^{a} \cdot (-1)^c \cdot(-1)^f \cdot\varepsilon_{pq}^e    \cdot\varepsilon_{2pq}^f\cdot (-1)^{gu}\\\
				&=&\varepsilon_{2}^{2a} \varepsilon_{pq}^e \varepsilon_{2pq}^f \cdot (-1)^{a+c+f+ug}.
			\end{eqnarray*}
			Thus, $a+c+f+ug\equiv 0\pmod2$. 
			Therefore, from these discussions, it follows that  we have:
			\begin{eqnarray}
				a+e+f+ug\equiv 0\pmod 2\label{eq1}\\
				c+f+g\equiv 0\pmod 2\label{eq2}\\
				a+ug+g\equiv 0\pmod 2\label{eq3}\\
				a+c+f+ug\equiv 0\pmod2\label{eq4}
			\end{eqnarray}
			
			From 	\eqref{eq1}, \eqref{eq3}   and \eqref{eq2}, we deduce that $e=c$. Thus 
			$$\xi^2=\varepsilon_{2}^a\varepsilon_{p}^a \sqrt{\varepsilon_{q}}^c\sqrt{\varepsilon_{2q}}^f\sqrt{\varepsilon_{pq}}^c\sqrt{\varepsilon_{2pq}}^f
			\sqrt{\varepsilon_{2p}}^g.$$

			On the other hand, 	as above, we show that the $2$-class group of $k_5$ is cyclic and that we have:	
			\begin{eqnarray}\label{class nbr of K+2}
				h_2(\KK)=\frac{1}{2}\cdot h_2(k_5)&=&\frac{1}{2}\cdot\frac{1}{4}q(k_5)h_2(2p)h_2(q)h_2(2pq)\nonumber\\
				&=&\frac{1}{2}\cdot\frac{1}{4}\cdot 4\cdot h_2(2p)\cdot 1  \cdot 2\nonumber\\
				&=& h_2(2p),\label{cln= h(2p)}
			\end{eqnarray}
			and	by class number formula (Lemma  \ref{wada's f.}), we have
			\begin{eqnarray}
				h_2(\KK)&=&\frac{1}{2^{9}}q(\KK)  h_2(2) h_2(p) h_2(q)h_2(2p) h_2(2q)h_2(pq)  h_2(2pq) \nonumber \\
				&=&\frac{1}{2^{9}}\cdot q(\KK)\cdot 1\cdot 1 \cdot 1  \cdot h_2(2p) \cdot 1 \cdot 2 \cdot 2=\frac{1}{2^7}\cdot q(\KK)\cdot h_2(2p).\nonumber
			\end{eqnarray}
			
			Therefore, $q(\KK)=2^7 $.

			Assume that each solution has $g=0$, then by \eqref{eq3} $a=0$. So  by \eqref{eq2} and \eqref{eq1}  $f=c=e$. 
			Therefore, $\xi^2=  \sqrt{\varepsilon_{q}}^c\sqrt{\varepsilon_{2q}}^c\sqrt{\varepsilon_{pq}}^c\sqrt{\varepsilon_{2pq}}^c.$ Thus, $q(\KK)=2^5$ or $2^6$. Which is absurd.
			This implies that there must be a solution having  $g=1$. So by \eqref{eq2}, $c\not= f$, and by \eqref{eq3} $a\equiv 1+u \pmod 2$. Finally, we have 
			$$\xi^2=\varepsilon_{2}^a\varepsilon_{p}^a \sqrt{\varepsilon_{q}}^c \sqrt{\varepsilon_{pq}}^c 
			\sqrt{\varepsilon_{2p}}  \text{ or } \varepsilon_{2}^a\varepsilon_{p}^a  \sqrt{\varepsilon_{2q}}^f \sqrt{\varepsilon_{2pq}}^f
			\sqrt{\varepsilon_{2p}} ,$$ 
			Since $q(\KK)=2^7 $, then  both of $\varepsilon_{2}^a\varepsilon_{p}^a \sqrt{\varepsilon_{q}} \sqrt{\varepsilon_{pq}}
			\sqrt{\varepsilon_{2p}}$ and $\varepsilon_{2}^a\varepsilon_{p}^a  \sqrt{\varepsilon_{2q}} \sqrt{\varepsilon_{2pq}}
			\sqrt{\varepsilon_{2p}}$ are squares in $\KK$, where $a\equiv 1+u \pmod 2$ and $u$ is defined in Lemma \ref{lm noms esp_2p}. Which completes the proof.
		\end{enumerate}
	\end{proof}
	\subsection{\bf The case: $p\equiv 1 \pmod8$,  $q\equiv 3 \pmod8$ and $\genfrac(){}{0}{p}{q} =1$} $\,$ 	\\
	For the sake that the reader could follow the proofs of this section, we suggest to start by reading carefully the proof of Theorem \ref{T_3_-1} which is exposed with some helpful  details.

	\newpage
	The following table summarizes very useful computations which we shall use frequently.
	{\begin{table}[H]
			\renewcommand{\arraystretch}{2.5}
			%\setlength{\tabcolsep}{1cm}
			%\footnotesize
			%\tiny		
			%\begin{center}%\rotatebox{-90}{%We construct the following table:\\
			\begin{tabular}{|c|c|c|c|c|c|c|c|c|c}
				\hline
				$\varepsilon$  & Conditions  & $\varepsilon^{1+\tau_2}$   & $\varepsilon^{1+\tau_1\tau_2}$& $\varepsilon^{1+\tau_1\tau_3}$& $\varepsilon^{1+\tau_2\tau_3}$& $\varepsilon^{1+\tau_1}$  \\ \hline

				\multirow{3}{*}{ $\sqrt{\varepsilon_{2pq}} $}	  &$(x-1)$ is a square in $\mathbb{N}$  & $-1$ & $-\varepsilon_{2pq}$& $-\varepsilon_{2pq}$ & $ \varepsilon_{2pq}$& $1$ \\ \cline{2-7}
				
				&$p(x-1)$ is a square in $\mathbb{N}$   & $1$ & $\varepsilon_{2pq}$&$-\varepsilon_{2pq}$ & $-\varepsilon_{2pq}$& $1$ \\ \cline{2-7}
				
				&$2p(x+1)$ is a square  in $\mathbb{N}$   & $-1$ & $-\varepsilon_{2pq}$& $\varepsilon_{2pq}$ & $-\varepsilon_{2pq}$& $1$ \\ \hline

				\multirow{3}{*}{  	$\sqrt{\varepsilon_{pq}}$} & $(v-1)$ is a square  in $\mathbb{N}$  & $-1$ & $1$& $1$ & $ \varepsilon_{pq}$&$ -\varepsilon_{pq}$ \\ \cline{2-7}
				
				& $p(v-1)$ is a square  in $\mathbb{N}$  & $1$ & $-1$& $1$ & $-\varepsilon_{pq}$& $-\varepsilon_{pq}$ \\ \cline{2-7}
				
				& $2p(v+1)$ is a square in $\mathbb{N}$ & $-1$ & $-1$& $1$ & $-\varepsilon_{pq}$&$\varepsilon_{pq}$ \\ \hline

			\end{tabular}
			\vspace*{0.2cm}
			\caption{Norms maps on units } \label{tab3} 
			%}
			%\caption{  Norms   when $q_1$ and $q_2$ verify conditions  $(\ref{cond 1})$ }\label{table1} 
			%	\end{center}
	\end{table} }

	Let $p\equiv 1\pmod{8}$ and $q\equiv3\pmod 8$ be two primes such that     $\genfrac(){}{0}{p}{q} =1$. Then, by Lemmas \ref{wada's f.} and  \ref{class numbers of quadratic field},  we have:
	\begin{eqnarray}\label{classnumberofKK}
		h_2(\KK)=	\frac{1}{2^{9}}q(\KK)\cdot h_2(2p)\cdot h_2(pq)\cdot h_2(2pq).
	\end{eqnarray}
	
	\begin{remark}
		Notice that by Lemma \ref{lm expressions of units under cond 3} there are nine possibilities which will be covered case by case by the following Theorems \ref{T_3_1_C1}-\ref{T_3_1_C9}.		\end{remark}
	\begin{theorem}\label{T_3_1_C1} Let $p\equiv 1\pmod{8}$ and $q\equiv3\pmod 8$ be two primes such that     $\genfrac(){}{0}{p}{q} =1$. 
		Put     $\KK=\QQ(\sqrt 2, \sqrt{p}, \sqrt{q} )$.  Assume furthermore that $x-1$ and $v-1$ are squares in $\mathbb{N}$, where $x$ and $v$ are defined in Lemma \ref{lm expressions of units under cond 3}.
		\begin{enumerate}[\rm 1.]
			\item If $N(\varepsilon_{2p})=-1$,   we have
			\begin{enumerate}[\rm $\bullet$]	
				\item The unit group of $\KK$ is :
				$$E_{\KK}=\langle -1,  \varepsilon_{2}, \varepsilon_{p},   \sqrt{\varepsilon_{q}}, \sqrt{\varepsilon_{2q}},  \sqrt{\varepsilon_{pq}} , \sqrt{\varepsilon_{2}\varepsilon_{p}\varepsilon_{2p}}, 
				\sqrt{\sqrt{\varepsilon_{q}}^a \sqrt{\varepsilon_{2q}}^a \sqrt{\varepsilon_{pq}}^a \sqrt{\varepsilon_{2pq}}^{1+b}}\rangle$$
				where $ a$, $b\in \{0,1\}$ are such that $a\not=b$ and $a =1$ if and only if $\sqrt{\varepsilon_{q}}\sqrt{\varepsilon_{2q}}\sqrt{\varepsilon_{pq}}\sqrt{\varepsilon_{2pq}}$ is a square in $\KK$.
				\item The $2$-class number  of  $\KK$ equals $ \frac{1}{2^{4-a}} h_2(2p)h_2(pq)h_2(2pq)$.
			\end{enumerate}
			\item Assume that $N(\varepsilon_{2p})=1$  and define  $a\in\{0,1\}$ to satisfy $a \equiv 1+u\pmod2$.   Then we have
			
			\begin{enumerate}[\rm $\bullet$]
				\item  The unit group of $\KK$ is :
				$$E_{\KK}=\langle -1,  \varepsilon_{2}, \varepsilon_{p},   \sqrt{\varepsilon_{q}}, \sqrt{\varepsilon_{2q}},  \sqrt{\varepsilon_{pq}} , \sqrt{\varepsilon_{2}^{ar'}\varepsilon_{p}^{ar'} \sqrt{\varepsilon_{q}}^{r'} \sqrt{\varepsilon_{pq}}^{r'}
					\sqrt{\varepsilon_{2p}}^{1+s'}},
				\sqrt{\varepsilon_{2}^{ar}\varepsilon_{p}^{ar}  \sqrt{\varepsilon_{2q}}^{r} \sqrt{\varepsilon_{2pq}}^{1+s}
					\sqrt{\varepsilon_{2p}}^r}      \rangle$$
				where $ r, r', s$, $s'\in \{0,1\}$ are such that $r\not=s$ $($resp. $r'\not=s'$$)$ and $r =1$ $($resp.  $r' =1$$)$ if and only if  $ \varepsilon_{2}^{a}\varepsilon_{p}^{a}  \sqrt{\varepsilon_{2q}} \sqrt{\varepsilon_{2pq}}
				\sqrt{\varepsilon_{2p}}$ (resp. $\varepsilon_{2}^{a}\varepsilon_{p}^{a}  \sqrt{\varepsilon_{ q}} \sqrt{\varepsilon_{pq}}
				\sqrt{\varepsilon_{2p}}$) is a square in $\KK$.
				\item  The $2$-class number  of  $\KK$ equals $ \frac{1}{2^{4-r-r'}} h_2(2p)h_2(pq)h_2(2pq)$.
			\end{enumerate}
		\end{enumerate}
	\end{theorem}
	\begin{proof}
		\begin{enumerate}[\rm 1.]
			\item  Assume that   $N(\varepsilon_{2p})=-1$.  By Lemma \ref{units of k1},     $\{\varepsilon_{2}, \varepsilon_{p},	\sqrt{\varepsilon_{2}\varepsilon_{p}\varepsilon_{2p}}\}$  is a  fundamental system of units of $k_1$. One can easily deduce from Lemmas \ref{lm expressions of units under cond 1 eps_q, 2q} and \ref{lm expressions of units under cond 1} that $\{\varepsilon_{2}, \sqrt{\varepsilon_{q}}, \sqrt{\varepsilon_{2q}}\}$ and $\{ \varepsilon_{2}, 	\sqrt{\varepsilon_{pq}},\sqrt{\varepsilon_{2pq}}\}$ are respectively fundamental systems of units of $k_2$ and $k_3$.
			It follows that,  	$$E_{k_1}E_{k_2}E_{k_3}=\langle-1,  \varepsilon_{2}, \varepsilon_{p},   \sqrt{\varepsilon_{q}}, \sqrt{\varepsilon_{2q}},  \sqrt{\varepsilon_{pq}} ,\sqrt{ \varepsilon_{2pq}}, \sqrt{\varepsilon_{2}\varepsilon_{p}\varepsilon_{2p}}\rangle.$$	
			Let  $\xi$ be an element of $\KK$ which is the  square root of an element of $E_{k_1}E_{k_2}E_{k_3}$. Therefore, we can assume that
			$$\xi^2=\varepsilon_{2}^a\varepsilon_{p}^b \sqrt{\varepsilon_{q}}^c\sqrt{\varepsilon_{2q}}^d\sqrt{\varepsilon_{pq}}^e\sqrt{\varepsilon_{2pq}}^f
			\sqrt{\varepsilon_{2}\varepsilon_{p}\varepsilon_{2p}}^g,$$
			where $a, b, c, d, e, f$ and $g$ are in $\{0, 1\}$.
			As  $x-1$ and $v-1$ are squares in $\mathbb{N}$, then clearly  with the same computations as in the proof of Theorem \ref{T_3_-1}, we get :
			$$\xi^2= \sqrt{\varepsilon_{q}}^f\sqrt{\varepsilon_{2q}}^f\sqrt{\varepsilon_{pq}}^f\sqrt{\varepsilon_{2pq}}^f.$$    
			So the first item.
			\item The same   computations   in the second part of the proof of Theorem \ref{T_3_-1} give the second item.
		\end{enumerate}
		The part concerning the $2$-class number follows from the above discussions  and \eqref{classnumberofKK}.
	\end{proof}

	\begin{theorem}\label{T_3_1_C2} Let $p\equiv 1\pmod{8}$ and $q\equiv3\pmod 8$ be two primes such that     $\genfrac(){}{0}{p}{q} =1$. 
		Put     $\KK=\QQ(\sqrt 2, \sqrt{p}, \sqrt{q} )$.  Assume furthermore that $x-1$ and $p(v-1)$ are squares in $\mathbb{N}$, where $x$ and $v$ are defined in Lemma \ref{lm expressions of units under cond 3}.
		Then
		\begin{enumerate}[\rm 1.]
			\item Assume that $N(\varepsilon_{2p})=-1$.  We have
			\begin{enumerate}[\rm $\bullet$]
				
				\item The unit group of $\KK$ is :
				$$E_{\KK}=\langle -1,   \varepsilon_{2}, \varepsilon_{p} ,    \sqrt{\varepsilon_{q}}, \sqrt{\varepsilon_{2q}},\sqrt{ \varepsilon_{pq}} ,\sqrt{ \varepsilon_{2pq}}, \sqrt{\varepsilon_{2}\varepsilon_{p}\varepsilon_{2p}} \rangle.$$
				\item  The $2$-class number  of  $\KK$ equals $ \frac{1}{2^{4}} h_2(2p)h_2(pq)h_2(2pq)$.  
			\end{enumerate}
			\item Assume that $N(\varepsilon_{2p})=1$ and let $a\in\{0,1\}$ be such that $a \equiv 1+u\pmod2$.   We have
			
			\begin{enumerate}[\rm $\bullet$]
				\item  The unit group of $\KK$ is :
				$$E_{\KK}=\langle -1,   \varepsilon_{2}, \varepsilon_{p},   \sqrt{\varepsilon_{q}}, \sqrt{\varepsilon_{2q}},  \sqrt{\varepsilon_{pq}} ,\sqrt{ \varepsilon_{2pq}},    \sqrt{\varepsilon_{2}^{a\alpha}\varepsilon_{p}^{a\alpha}     \sqrt{\varepsilon_{2q}}^\alpha\sqrt{\varepsilon_{2pq}}^\alpha
					\sqrt{\varepsilon_{2p}}^{1+\gamma} }  \rangle$$
				where $\alpha$, $\gamma\in \{0,1\}$ are such that $\alpha\not=\gamma$ and $\alpha =1$ if and only if  $\varepsilon_{2}^a \varepsilon_{p}^a      \sqrt{\varepsilon_{2q}} \sqrt{\varepsilon_{2pq}}\sqrt{\varepsilon_{2p}}$ is a square in $\KK$.
				\item  The $2$-class number  of  $\KK$ equals $ \frac{1}{2^{4-\alpha}} h_2(2p)h_2(pq)h_2(2pq)$.
			\end{enumerate}
		\end{enumerate}
	\end{theorem}
	\begin{proof}We shall make use of \eqref{T_3_-1_eqi2p_N=1}  and Tables \ref{norm q=3 p=1mod4q}  and \ref{tab3}.
		\begin{enumerate}[\rm 1.]
			\item Assume that   $N(\varepsilon_{2p})=-1$.
			Note that $\{\varepsilon_{2}, \varepsilon_{p},	\sqrt{\varepsilon_{2}\varepsilon_{p}\varepsilon_{2p}}\}$, is  a fundamental system  of units of $k_1$.
			Using 	Lemma \ref{lm expressions of units under cond 3}, we show that $\{\varepsilon_{2}, \sqrt{\varepsilon_{q}}, \sqrt{\varepsilon_{2q}}\}$ and $\{ \varepsilon_{2}, 	 \varepsilon_{pq},\sqrt{\varepsilon_{2pq}}\}$ are respectively fundamental systems of units of $k_2$ and $k_3$.
			So we have:
			$$E_{k_1}E_{k_2}E_{k_3}=\langle-1,  \varepsilon_{2}, \varepsilon_{p} ,   \varepsilon_{pq} ,   \sqrt{\varepsilon_{q}}, \sqrt{\varepsilon_{2q}} ,\sqrt{ \varepsilon_{2pq}}, \sqrt{\varepsilon_{2}\varepsilon_{p}\varepsilon_{2p}}\rangle.$$	
			Let  $\xi$ be an element of $\KK$ which is the  square root of an element of $E_{k_1}E_{k_2}E_{k_3}$. Therefore, we can assume that
			$$\xi^2=\varepsilon_{2}^a\varepsilon_{p}^b  \varepsilon_{pq}^c  \sqrt{\varepsilon_{q}}^d\sqrt{\varepsilon_{2q}}^e\sqrt{\varepsilon_{2pq}}^f
			\sqrt{\varepsilon_{2}\varepsilon_{p}\varepsilon_{2p}}^g,$$
			where $a, b, c, d, e, f$ and $g$ are in $\{0, 1\}$.
			
			\noindent\ding{224}  Let us start	by applying   the norm map $N_{\KK/k_2}=1+\tau_2$.
			We have 	%$\sqrt{\varepsilon_{pq}}^{1+\tau_2}=1$, 	$\sqrt{\varepsilon_{2pq}}^{1+\tau_2}=-1$ and
			$\sqrt{\varepsilon_{2}\varepsilon_{p}\varepsilon_{2p}}^{1+ \tau_2}=(-1)^v\varepsilon_{2}$, for some $v\in\{0,1\}$. Then,  by  Tables \ref{norm q=3 p=1mod4q}  and \ref{tab3}, we get:
			%	\begin{eqnarray}\label{T_3_1_tau2_N=1 C2}
			%	\begin{tabular}{ |c|c|c|c|c|c|c|c|c|}
			%		\hline
			%		$\varepsilon$&$\varepsilon_{2}$ &$ {\varepsilon_{p}}$&   $ \varepsilon_{pq}$& $\sqrt{\varepsilon_{q}}$&$\sqrt{\varepsilon_{2q}}$&$\sqrt{\varepsilon_{2pq}}$&$\sqrt{\varepsilon_{2}\varepsilon_{p}\varepsilon_{2p}}$ \\
			%	\hline
			%	$\varepsilon^{1+\tau_2}$&$\varepsilon_{2}^2$ &$-1$&	$1$& $\varepsilon_{q}$  &$\varepsilon_{2q}$&  $-1$&$(-1)^v\varepsilon_{2}$ \\
			%		\hline
			%	\end{tabular} 
			%	\end{eqnarray}
			%	Thus 
			\begin{eqnarray*}
				N_{\KK/k_2}(\xi^2)&=&
				\varepsilon_{2}^{2a}\cdot (-1)^b \cdot1\cdot \varepsilon_{q}^d\cdot \varepsilon_{2q}^e\cdot(-1)^f  \cdot (-1)^{gv}\varepsilon_{2}^g\\\
				&=&	\varepsilon_{2}^{2a}  \varepsilon_{q}^c\varepsilon_{2q}^d\cdot(-1)^{b +f+gv}\varepsilon_{2}^g .
			\end{eqnarray*} 
			Thus,  $b +f+gv\equiv 0 \pmod 2$ and $g=0$. Therefore, $b=f$ and 
			$$\xi^2=\varepsilon_{2}^a\varepsilon_{p}^f  \varepsilon_{pq}^c  \sqrt{\varepsilon_{q}}^d\sqrt{\varepsilon_{2q}}^e\sqrt{\varepsilon_{2pq}}^f.$$
			
			\noindent\ding{224} Let us apply the norm $N_{\KK/k_5}=1+\tau_1\tau_2$, with $k_5=\QQ(\sqrt{q}, \sqrt{2p})$.% We have 
			%	$\sqrt{\varepsilon_{pq}}^{1+\tau_1\tau_2}=-1$ and 
			%	$\sqrt{\varepsilon_{2pq}}^{1+\tau_1\tau_2}=-\varepsilon_{2pq}$. 
			Then, by  Tables \ref{norm q=3 p=1mod4q}  and \ref{tab3}, we get:
			%\begin{eqnarray}\label{T_3_1_tau1tau2_N=-1 C2}
			%	\begin{tabular}{ |c|c|c|c|c|c|c|c|c|}
			%		\hline
			%		$\varepsilon$&$\varepsilon_{2}$ &$ {\varepsilon_{p}}$ & $\varepsilon_{pq}$ &$\sqrt{\varepsilon_{q}}$&$\sqrt{\varepsilon_{2q}}$& $\sqrt{\varepsilon_{2pq}}$ \\
			%		\hline
			%		$\varepsilon^{1+\tau_1\tau_2}$&$-1$ &$-1$& $1$	& $-\varepsilon_{q}$&$1$&  ${-\varepsilon_{2pq}}$   \\
			%		\hline
			%	\end{tabular} 
			%	\end{eqnarray}
			\begin{eqnarray*}
				N_{\KK/k_5}(\xi^2)&=&(-1)^a\cdot (-1)^f\cdot 1\cdot (-1)^d\cdot \varepsilon_{  q}^d\cdot1 \cdot (-1)^f\cdot \varepsilon_{  2pq}^f\\
				&=&	 (-1)^{a +d }\varepsilon_{  q}^d\cdot \varepsilon_{  2pq}^f.
			\end{eqnarray*}
			Thus $a=d=f$ and 
			$$\xi^2=\varepsilon_{2}^f\varepsilon_{p}^f  \varepsilon_{pq}^c  \sqrt{\varepsilon_{q}}^f\sqrt{\varepsilon_{2q}}^e\sqrt{\varepsilon_{2pq}}^f.$$

			\noindent\ding{224} Let us apply the norm $N_{\KK/k_6}=1+\tau_1\tau_3$, with $k_6=\QQ(\sqrt{p}, \sqrt{2q})$. 
			%We have 
			%	$\sqrt{\varepsilon_{pq}}^{1+\tau_1\tau_3}=1$ and 
			%	$\sqrt{\varepsilon_{2pq}}^{1+\tau_1\tau_3}=-\varepsilon_{2pq}$. Then,
			By   Tables \ref{norm q=3 p=1mod4q}  and \ref{tab3}, we get: 
			%	\begin{eqnarray}\label{T_3_1_tau1tau3_N=-1 C2}
			%	\begin{tabular}{ |c|c|c|c|c|c|c|c|c|}
			%		\hline
			%		$\varepsilon$&$\varepsilon_{2}$ &$ {\varepsilon_{p}}$&$\varepsilon_{pq}$ &$\sqrt{\varepsilon_{q}}$&$\sqrt{\varepsilon_{2q}}$ &$\sqrt{\varepsilon_{2pq}}$ \\
			%		\hline
			%		$\varepsilon^{1+\tau_1\tau_3}$&$-1$ &$\varepsilon_{p}^2$ &$1$ &	$1$&$-\varepsilon_{2q}$&  ${-\varepsilon_{2pq}}$   \\
			%		\hline
			%	\end{tabular} 
			%	\end{eqnarray}
			\begin{eqnarray*}
				N_{\KK/k_6}(\xi^2)&=&(-1)^f\cdot \varepsilon_{  p}^{2f}\cdot 1\cdot 1\cdot (-1)^e\cdot \varepsilon_{2  q}^e\cdot  (-1)^f\cdot \varepsilon_{  2pq}^f\\
				&=&	\varepsilon_{  p}^{2a} (-1)^{e}\varepsilon_{2  q}^e\varepsilon_{  2pq}^f.
			\end{eqnarray*}
			Thus $e=f=0$. Hence 
			$$\xi^2=   \varepsilon_{pq}^c .$$
			
			By   Lemma \ref{lm expressions of units under cond 3} $\varepsilon_{pq}$ is a square in $\KK$, therefore
			$$E_{\KK} =\langle-1,  \varepsilon_{2}, \varepsilon_{p} ,    \sqrt{\varepsilon_{q}}, \sqrt{\varepsilon_{2q}},\sqrt{ \varepsilon_{pq}} ,\sqrt{ \varepsilon_{2pq}}, \sqrt{\varepsilon_{2}\varepsilon_{p}\varepsilon_{2p}}\rangle.$$

			\item Assume that 	$N(\varepsilon_{2p})=1$. 	Note that $\{\varepsilon_{2}, \varepsilon_{p},	\sqrt{\varepsilon_{2p}}\}$, is  a fundamental system  of units of $k_1$.
			Using 	Lemma \ref{lm expressions of units under cond 3}, we show that $\{\varepsilon_{2}, \sqrt{\varepsilon_{q}}, \sqrt{\varepsilon_{2q}}\}$ and $\{ \varepsilon_{2}, 	 \varepsilon_{pq},\sqrt{\varepsilon_{2pq}}\}$ are respectively fundamental systems of units of $k_2$ and $k_3$.
			It follows that,  	$$E_{k_1}E_{k_2}E_{k_3}=\langle-1,  \varepsilon_{2}, \varepsilon_{p},   {\varepsilon_{pq}} , \sqrt{\varepsilon_{q}}, \sqrt{\varepsilon_{2q}},  \sqrt{ \varepsilon_{2pq}}, \sqrt{\varepsilon_{2p}}\rangle.$$	
			Let  $\xi$ be an element of $\KK$ which is the  square root of an element of $E_{k_1}E_{k_2}E_{k_3}$. Therefore, we can assume that
			$$\xi^2=\varepsilon_{2}^a\varepsilon_{p}^b  \varepsilon_{pq}^c  \sqrt{\varepsilon_{q}}^d\sqrt{\varepsilon_{2q}}^e\sqrt{\varepsilon_{2pq}}^f
			\sqrt{\varepsilon_{2p}}^g,$$
			where $a, b, c, d, e, f$ and $g$ are in $\{0, 1\}$.

			\noindent\ding{224}	Let us start	by applying   the norm map $N_{\KK/k_2}=1+\tau_2$. By     Tables \ref{norm q=3 p=1mod4q}  and \ref{tab3}, we have	
			
			\begin{eqnarray*}
				N_{\KK/k_2}(\xi^2)&=&
				\varepsilon_{2}^{2a}\cdot (-1)^b \cdot1\cdot \varepsilon_{q}^d\cdot \varepsilon_{2q}^e\cdot(-1)^f  \cdot (-1)^{gu}\\\
				&=&	\varepsilon_{2}^{2a}  \varepsilon_{q}^c\varepsilon_{2q}^d\cdot(-1)^{b +f+gu} .
			\end{eqnarray*}
			Thus,  	  $b +f+gu\equiv 0\pmod2$.    
			
			\noindent\ding{224}	Let   us apply    the norm map $N_{\KK/k_5}=1+\tau_1\tau_2$, with $k_5=\QQ(\sqrt{q}, \sqrt{2p})$.	By    Tables \ref{norm q=3 p=1mod4q}  and \ref{tab3}, we have	
			\begin{eqnarray*}
				N_{\KK/k_5}(\xi^2)&=&(-1)^a\cdot(-1)^b\cdot   1\cdot(-1)^d\cdot\varepsilon_{2  q}^d\cdot  1 \cdot  (-1)^f\cdot \varepsilon_{  2pq}^f \cdot(-1)^g\cdot\varepsilon_{2  p}^g\\
				&=&  (-1)^{a+b+d+f+g}\varepsilon_{2  q}^d\varepsilon_{  2pq}^f\varepsilon_{2  p}^g.
			\end{eqnarray*}
			Thus,  	  $a+b+d+f+g\equiv 0\pmod2$ and $d+f+g\equiv 0\pmod2$.  Therefore $a=b$ and   
			$$\xi^2=\varepsilon_{2}^a\varepsilon_{p}^a  \varepsilon_{pq}^c  \sqrt{\varepsilon_{q}}^d\sqrt{\varepsilon_{2q}}^e\sqrt{\varepsilon_{2pq}}^f
			\sqrt{\varepsilon_{2p}}^g.$$
			
			\noindent\ding{224} Let us apply the norm $N_{\KK/k_6}=1+\tau_1\tau_3$, with $k_6=\QQ(\sqrt{p}, \sqrt{2q})$. By Tables \ref{norm q=3 p=1mod4q}  and \ref{tab3}, we have 
			\begin{eqnarray*}
				N_{\KK/k_6}(\xi^2)&=&(-1)^a\cdot \varepsilon_{  p}^{2a}\cdot 1\cdot 1\cdot (-1)^e\cdot \varepsilon_{2  q}^e\cdot  (-1)^f\cdot \varepsilon_{  2pq}^f\cdot  (-1)^{gu+g}\\
				&=&	\varepsilon_{  p}^{2a} (-1)^{a+e+f+g+ug}\varepsilon_{2  q}^e\varepsilon_{  2pq}^f.
			\end{eqnarray*}
			Thus,  	  $a+e+f+g+ug\equiv 0\pmod2$ and $e=f$.  Therefore,  $a+g+ug\equiv 0\pmod2$ and 
			$$\xi^2=\varepsilon_{2}^a\varepsilon_{p}^a  \varepsilon_{pq}^c  \sqrt{\varepsilon_{q}}^d\sqrt{\varepsilon_{2q}}^e\sqrt{\varepsilon_{2pq}}^e
			\sqrt{\varepsilon_{2p}}^g.$$

			\noindent\ding{224} Let us apply       the norm map $N_{\KK/k_3}=1+\tau_2\tau_3$, with $k_3=\QQ(\sqrt{2}, \sqrt{pq})$. By Tables \ref{norm q=3 p=1mod4q}  and \ref{tab3}, we get:
			%	\begin{eqnarray}\label{T_3_1_tau2tau3_N=-1 C2}
			%		\begin{tabular}{ |c|c|c|c|c|c|c|c|c|}
			%			\hline
			%			$\varepsilon$&$\varepsilon_{2}$ &$ {\varepsilon_{p}}$&$\varepsilon_{pq}$ &$\sqrt{\varepsilon_{q}}$&$\sqrt{\varepsilon_{2q}}$ &$\sqrt{\varepsilon_{2pq}}$ \\
			%			\hline
			%			$\varepsilon^{1+\tau_2\tau_3}$&$\varepsilon_{2}^2$ &$-1$ &$\varepsilon_{pq}^2$ &	$-1$&$-1$&  ${\varepsilon_{2pq}}$   \\
			%			\hline
			%		\end{tabular} 
			%	\end{eqnarray}
			\begin{eqnarray*}
				N_{\KK/k_3}(\xi^2)&=&
				\varepsilon_{2}^{2a}\cdot (-1)^a \cdot\varepsilon_{pq}^{2c}\cdot (-1)^d\cdot (-1)^e\cdot\varepsilon_{2pq}^e  \cdot (-1)^{gu}\\\
				&=&	\varepsilon_{2}^{2a} \varepsilon_{2pq}^e  \cdot(-1)^{a +d+e+gu} .
			\end{eqnarray*}
			Thus,  	  $a +d+e+gu\equiv 0\pmod2$.

			\noindent\ding{224} Let us apply the norm $N_{\KK/k_4}=1+\tau_1$, with $k_4=\QQ(\sqrt{p}, \sqrt{q})$. % We have 
			%	$\sqrt{\varepsilon_{pq}}^{1+\tau_1}=-\varepsilon_{pq}$ and 
			%	$\sqrt{\varepsilon_{2pq}}^{1+\tau_1}=1$. 
			So, by Tables \ref{norm q=3 p=1mod4q}  and \ref{tab3}, we get:
			%	\begin{eqnarray}\label{T_3_1_tau1_N=1 C2}{\tiny }
			%		\begin{tabular}{ |c|c|c|c|c|c|c|c|c|}
			%			\hline
			%			$\varepsilon$&$\varepsilon_{2}$ &${\varepsilon_{p}}$& $\varepsilon_{pq}$& $\sqrt{\varepsilon_{q}}$&$\sqrt{\varepsilon_{2q}}$& $\sqrt{\varepsilon_{2pq}}$ \\
			%			\hline
			%			$\varepsilon^{1+\tau_1}$&$-1$ &$\varepsilon_{p}^2$&$\varepsilon_{pq}^{2}$	&$-\varepsilon_{q}$&$1$&  $1$   \\
			%		\hline
			%	\end{tabular}
			%	\end{eqnarray}
			%	Then 
			\begin{eqnarray*}
				N_{\KK/k_4}(\xi^2)&=&(-1)^a\cdot
				\varepsilon_{p}^{2a}\cdot \varepsilon_{pq}^{2a} \cdot (-1)^d\cdot\varepsilon_{q}^d\cdot 1 \cdot1  \cdot (-1)^{gu+g}\\\
				&=&	\varepsilon_{2}^{2a}\varepsilon_{pq}^{2a}\cdot(-1)^{a +d+g+gu}\varepsilon_{q}^d .
			\end{eqnarray*}
			Thus $d=0$ and so   $a +e+gu\equiv 0\pmod2$. Since $a+g+ug\equiv 0\pmod2$, we have $e=g$. Since $\varepsilon_{pq}$ is a square in $\KK$, we can  disregard it in the form of $\xi^2$.
			Therefore, 
			$$\xi^2=\varepsilon_{2}^a\varepsilon_{p}^a     \sqrt{\varepsilon_{2q}}^e\sqrt{\varepsilon_{2pq}}^e
			\sqrt{\varepsilon_{2p}}^e,$$
			with $a+e+eu\equiv 0\pmod2$.  So the result (cf. \eqref{classnumberofKK}).
		\end{enumerate}
	\end{proof}

	\begin{theorem}\label{T_3_1-3} Let $p\equiv 1\pmod{8}$ and $q\equiv3\pmod 8$ be two primes such that     $\genfrac(){}{0}{p}{q} =1$. 
		Put     $\KK=\QQ(\sqrt 2, \sqrt{p}, \sqrt{q} )$.  Assume furthermore that $x-1$ and $2p(v+1)$ are squares in $\mathbb{N}$, where $x$ and $v$ are defined in Lemma \ref{lm expressions of units under cond 3}.
		\begin{enumerate}[\rm 1.]
			\item Assume that $N(\varepsilon_{2p})=-1$.  We have
			\begin{enumerate}[\rm $\bullet$]
				
				\item The unit group of $\KK$ is :
				$$E_{\KK}=\langle -1,   \varepsilon_{2}, \varepsilon_{p} ,    \sqrt{\varepsilon_{q}}, \sqrt{\varepsilon_{2q}},\sqrt{ \varepsilon_{pq}} ,\sqrt{ \varepsilon_{2pq}}, \sqrt{\varepsilon_{2}\varepsilon_{p}\varepsilon_{2p}} \rangle.$$
				\item  The $2$-class number  of  $\KK$ equals $ \frac{1}{2^{4}} h_2(2p)h_2(pq)h_2(2pq)$.  
			\end{enumerate}
			\item Assume that $N(\varepsilon_{2p})=1$ and let $a\in\{0,1\}$ be such that $a \equiv 1+u\pmod2$.   We have
			
			\begin{enumerate}[\rm $\bullet$]
				\item  The unit group of $\KK$ is :
				$$E_{\KK}=\langle -1,   \varepsilon_{2}, \varepsilon_{p},   \sqrt{\varepsilon_{q}}, \sqrt{\varepsilon_{2q}},  \sqrt{\varepsilon_{pq}} ,\sqrt{ \varepsilon_{2pq}},    \sqrt{\varepsilon_{2}^{a\alpha}\varepsilon_{p}^{a\alpha}     \sqrt{\varepsilon_{2q}}^\alpha\sqrt{\varepsilon_{2pq}}^\alpha
					\sqrt{\varepsilon_{2p}}^{1+\gamma} }  \rangle$$
				where $\alpha$, $\gamma\in \{0,1\}$ are such that $\alpha\not=\gamma$ and $\alpha =1$ if and only if $\varepsilon_{2}^a \varepsilon_{p}^a      \sqrt{\varepsilon_{2q}} \sqrt{\varepsilon_{2pq}}\sqrt{\varepsilon_{2p}}$ is a square in $\KK$.
				\item  The $2$-class number  of  $\KK$ equals $ \frac{1}{2^{4-\alpha}} h_2(2p)h_2(pq)h_2(2pq)$.
			\end{enumerate}
		\end{enumerate}
	\end{theorem}
	\begin{proof}
		\begin{enumerate}[\rm 1.]
			\item Assume that   $N(\varepsilon_{2p})=-1$.
			Note that $\{\varepsilon_{2}, \varepsilon_{p},	\sqrt{\varepsilon_{2}\varepsilon_{p}\varepsilon_{2p}}\}$, is  a fundamental system  of units of $k_1$.
			Using 	Lemma \ref{lm expressions of units under cond 3}, we show that $\{\varepsilon_{2}, \sqrt{\varepsilon_{q}}, \sqrt{\varepsilon_{2q}}\}$ and $\{ \varepsilon_{2}, 	 \varepsilon_{pq},\sqrt{\varepsilon_{2pq}}\}$ are respectively fundamental systems of units of $k_2$ and $k_3$.
			So we have:
			$$E_{k_1}E_{k_2}E_{k_3}=\langle-1,  \varepsilon_{2}, \varepsilon_{p} ,   \varepsilon_{pq} ,   \sqrt{\varepsilon_{q}}, \sqrt{\varepsilon_{2q}} ,\sqrt{ \varepsilon_{2pq}}, \sqrt{\varepsilon_{2}\varepsilon_{p}\varepsilon_{2p}}\rangle.$$	
			Let  $\xi$ be an element of $\KK$ which is the  square root of an element of $E_{k_1}E_{k_2}E_{k_3}$. Therefore, we can assume that
			$$\xi^2=\varepsilon_{2}^a\varepsilon_{p}^b  \varepsilon_{pq}^c  \sqrt{\varepsilon_{q}}^d\sqrt{\varepsilon_{2q}}^e\sqrt{\varepsilon_{2pq}}^f
			\sqrt{\varepsilon_{2}\varepsilon_{p}\varepsilon_{2p}}^g,$$
			where $a, b, c, d, e, f$ and $g$ are in $\{0, 1\}$.

			\noindent\ding{224}  Let us start	by applying   the norm map $N_{\KK/k_2}=1+\tau_2$.  We have 
			
			%	$\sqrt{\varepsilon_{pq}}^{1+\tau_2}=-1$,
			%	$\sqrt{\varepsilon_{2pq}}^{1+\tau_2}=-1$ and
			$\sqrt{\varepsilon_{2}\varepsilon_{p}\varepsilon_{2p}}^{1+ \tau_2}=(-1)^{v'}\varepsilon_{2}$, for some $v'\in\{0,1\}$. Then, by Tables \ref{norm q=3 p=1mod4q}  and \ref{tab3}, we get:
			%	\begin{eqnarray}\label{T_3_1_tau2_N=1 C2}
			%	\begin{tabular}{ |c|c|c|c|c|c|c|c|c|}
			%		\hline
			%		$\varepsilon$&$\varepsilon_{2}$ &$ {\varepsilon_{p}}$&   $ \varepsilon_{pq}$& $\sqrt{\varepsilon_{q}}$&$\sqrt{\varepsilon_{2q}}$&$\sqrt{\varepsilon_{2pq}}$&$\sqrt{\varepsilon_{2}\varepsilon_{p}\varepsilon_{2p}}$ \\
			%	\hline
			%	$\varepsilon^{1+\tau_2}$&$\varepsilon_{2}^2$ &$-1$&	$1$& $\varepsilon_{q}$  &$\varepsilon_{2q}$&  $-1$&$(-1)^v\varepsilon_{2}$ \\
			%		\hline
			%	\end{tabular} 
			%	\end{eqnarray}
			%	Thus 
			\begin{eqnarray*}
				N_{\KK/k_2}(\xi^2)&=&
				\varepsilon_{2}^{2a}\cdot (-1)^b \cdot1\cdot \varepsilon_{q}^d\cdot \varepsilon_{2q}^e\cdot(-1)^f  \cdot (-1)^{gv'}\varepsilon_{2}^g\\\
				&=&	\varepsilon_{2}^{2a}  \varepsilon_{q}^c\varepsilon_{2q}^d\cdot(-1)^{b +f+gv'}\varepsilon_{2}^g .
			\end{eqnarray*} 
			Thus,  $b +f+gv'\equiv 0 \pmod 2$ and $g=0$. Then $b=f$ and 
			$$\xi^2=\varepsilon_{2}^a\varepsilon_{p}^f  \varepsilon_{pq}^c  \sqrt{\varepsilon_{q}}^d\sqrt{\varepsilon_{2q}}^e\sqrt{\varepsilon_{2pq}}^f.$$
			
			\noindent\ding{224} Let us apply the norm $N_{\KK/k_5}=1+\tau_1\tau_2$, with $k_5=\QQ(\sqrt{q}, \sqrt{2p})$.%We have 
			%	$\sqrt{\varepsilon_{pq}}^{1+\tau_1\tau_2}=-1$ and 
			%	$\sqrt{\varepsilon_{2pq}}^{1+\tau_1\tau_2}=-\varepsilon_{2pq}$. 
			\; By Tables \ref{norm q=3 p=1mod4q}  and \ref{tab3}, we have:
			\begin{eqnarray*}
				N_{\KK/k_5}(\xi^2)&=&(-1)^a\cdot (-1)^f\cdot 1\cdot (-1)^d\cdot \varepsilon_{  q}^d\cdot1 \cdot (-1)^f\cdot \varepsilon_{  2pq}^f\\
				&=&	 (-1)^{a +d }\varepsilon_{  q}^d\cdot \varepsilon_{  2pq}^f.
			\end{eqnarray*}
			Thus $a=d=f$ and 
			$$\xi^2=\varepsilon_{2}^f\varepsilon_{p}^f  \varepsilon_{pq}^c  \sqrt{\varepsilon_{q}}^f\sqrt{\varepsilon_{2q}}^e\sqrt{\varepsilon_{2pq}}^f.$$

			\noindent\ding{224} Let us apply the norm $N_{\KK/k_6}=1+\tau_1\tau_3$, with $k_6=\QQ(\sqrt{p}, \sqrt{2q})$. %We have 
			%	$\sqrt{\varepsilon_{pq}}^{1+\tau_1\tau_3}=1$ and 
			%$\sqrt{\varepsilon_{2pq}}^{1+\tau_1\tau_3}=-\varepsilon_{2pq}$. 
			\; Then, by Tables \ref{norm q=3 p=1mod4q}  and \ref{tab3}, we get: 
			\begin{eqnarray*}
				N_{\KK/k_6}(\xi^2)&=&(-1)^f\cdot \varepsilon_{  p}^{2f}\cdot 1\cdot 1\cdot (-1)^e\cdot \varepsilon_{2  q}^e\cdot  (-1)^f\cdot \varepsilon_{  2pq}^f\\
				&=&	\varepsilon_{  p}^{2a} (-1)^{e}\varepsilon_{2  q}^e\varepsilon_{  2pq}^f.
			\end{eqnarray*}
			Thus $e=f=0$. Hence 
			$$\xi^2=   \varepsilon_{pq}^c .$$
			
			By   Lemma \ref{lm expressions of units under cond 3} $\varepsilon_{pq}$ is a square in $\KK$, therefore
			$$E_{\KK} =\langle-1,  \varepsilon_{2}, \varepsilon_{p} ,    \sqrt{\varepsilon_{q}}, \sqrt{\varepsilon_{2q}},\sqrt{ \varepsilon_{pq}} ,\sqrt{ \varepsilon_{2pq}}, \sqrt{\varepsilon_{2}\varepsilon_{p}\varepsilon_{2p}}\rangle.$$	
			The rest of the first item is direct from Lemma \ref{wada's f.}.

			\item Assume that 	$N(\varepsilon_{2p})=1$. 	Note that $\{\varepsilon_{2}, \varepsilon_{p},	\sqrt{\varepsilon_{2p}}\}$, is  a fundamental system  of units of $k_1$.
			Using 	Lemma \ref{lm expressions of units under cond 3}, we show that $\{\varepsilon_{2}, \sqrt{\varepsilon_{q}}, \sqrt{\varepsilon_{2q}}\}$ and $\{ \varepsilon_{2}, 	 \varepsilon_{pq},\sqrt{\varepsilon_{2pq}}\}$ are respectively fundamental systems of units of $k_2$ and $k_3$.
			It follows that,  	$$E_{k_1}E_{k_2}E_{k_3}=\langle-1,  \varepsilon_{2}, \varepsilon_{p},   {\varepsilon_{pq}} , \sqrt{\varepsilon_{q}}, \sqrt{\varepsilon_{2q}},  \sqrt{ \varepsilon_{2pq}}, \sqrt{\varepsilon_{2p}}\rangle.$$	
			Let  $\xi$ be an element of $\KK$ which is the  square root of an element of $E_{k_1}E_{k_2}E_{k_3}$. Therefore, we can assume that 
			$$\xi^2=\varepsilon_{2}^a\varepsilon_{p}^b  \varepsilon_{pq}^c  \sqrt{\varepsilon_{q}}^d\sqrt{\varepsilon_{2q}}^e\sqrt{\varepsilon_{2pq}}^f
			\sqrt{\varepsilon_{2p}}^g,$$
			where $a, b, c, d, e, f$ and $g$ are in $\{0, 1\}$.

			\noindent\ding{224}	Let us start	by applying   the norm map $N_{\KK/k_2}=1+\tau_2$. We have	
			\begin{eqnarray*}
				N_{\KK/k_2}(\xi^2)&=&
				\varepsilon_{2}^{2a}\cdot (-1)^b \cdot1\cdot \varepsilon_{q}^d\cdot \varepsilon_{2q}^e\cdot(-1)^f  \cdot (-1)^{gu}\\\
				&=&	\varepsilon_{2}^{2a}  \varepsilon_{q}^c\varepsilon_{2q}^d\cdot(-1)^{b +f+gu} .
			\end{eqnarray*}
			Thus,  	  $b +f+gu\equiv 0\pmod2$.    
			
			\noindent\ding{224}	Let   us apply    the norm map $N_{\KK/k_5}=1+\tau_1\tau_2$, with $k_5=\QQ(\sqrt{q}, \sqrt{2p})$.	We have	
			\begin{eqnarray*}
				N_{\KK/k_5}(\xi^2)&=&(-1)^a\cdot(-1)^b\cdot   1\cdot(-1)^d\cdot\varepsilon_{2  q}^d\cdot  1 \cdot  (-1)^f\cdot \varepsilon_{  2pq}^f \cdot(-1)^g\cdot\varepsilon_{2  p}^g\\
				&=&  (-1)^{a+b+d+f+g}\varepsilon_{2  q}^d\varepsilon_{  2pq}^f\varepsilon_{2  p}^g.
			\end{eqnarray*}
			Thus,  	  $a+b+d+f+g\equiv 0\pmod2$ and $d+f+g\equiv 0\pmod2$.  Therefore $a=b$ and   
			
			$$\xi^2=\varepsilon_{2}^a\varepsilon_{p}^a  \varepsilon_{pq}^c  \sqrt{\varepsilon_{q}}^d\sqrt{\varepsilon_{2q}}^e\sqrt{\varepsilon_{2pq}}^f
			\sqrt{\varepsilon_{2p}}^g.$$
			
			\noindent\ding{224} Let us apply the norm $N_{\KK/k_6}=1+\tau_1\tau_3$, with $k_6=\QQ(\sqrt{p}, \sqrt{2q})$. We have 
			\begin{eqnarray*}
				N_{\KK/k_6}(\xi^2)&=&(-1)^a\cdot \varepsilon_{  p}^{2a}\cdot 1\cdot 1\cdot (-1)^e\cdot \varepsilon_{2  q}^e\cdot  (-1)^f\cdot \varepsilon_{  2pq}^f\cdot  (-1)^{gu+g}\\
				&=&	\varepsilon_{  p}^{2a} (-1)^{a+e+f+g+ug}\varepsilon_{2  q}^e\varepsilon_{  2pq}^f.
			\end{eqnarray*}
			Thus,  	  $a+e+f+g+ug\equiv 0\pmod2$ and $e=f$.  Therefore,  $a+g+ug\equiv 0\pmod2$ and 
			$$\xi^2=\varepsilon_{2}^a\varepsilon_{p}^a  \varepsilon_{pq}^c  \sqrt{\varepsilon_{q}}^d\sqrt{\varepsilon_{2q}}^e\sqrt{\varepsilon_{2pq}}^e
			\sqrt{\varepsilon_{2p}}^g.$$

			\noindent\ding{224} Let us apply       the norm map $N_{\KK/k_3}=1+\tau_2\tau_3$, with $k_3=\QQ(\sqrt{2}, \sqrt{pq})$. %We have 
			%	$\sqrt{\varepsilon_{pq}}^{1+\tau_2\tau_3}=-\varepsilon_{pq}$ and 
			%	$\sqrt{\varepsilon_{2pq}}^{1+\tau_2\tau_3}=\varepsilon_{2pq}$.
			So, by Tables \ref{norm q=3 p=1mod4q}  and \ref{tab3}, we get:
			\begin{eqnarray*}
				N_{\KK/k_3}(\xi^2)&=&
				\varepsilon_{2}^{2a}\cdot (-1)^a \cdot\varepsilon_{pq}^{2c}\cdot (-1)^d\cdot (-1)^e\cdot\varepsilon_{2pq}^e  \cdot (-1)^{gu}\\\
				&=&	\varepsilon_{2}^{2a} \varepsilon_{2pq}^e  \cdot(-1)^{a +d+e+gu} .
			\end{eqnarray*}
			Thus,  	  $a +d+e+gu\equiv 0\pmod2$.

			\noindent\ding{224} Let us apply the norm $N_{\KK/k_4}=1+\tau_1$, with $k_4=\QQ(\sqrt{p}, \sqrt{q})$. % We have 
			%	$\sqrt{\varepsilon_{pq}}^{1+\tau_1}= \varepsilon_{pq}$ and 
			%$\sqrt{\varepsilon_{2pq}}^{1+\tau_1}=1$. 
			So, by Tables \ref{norm q=3 p=1mod4q}  and \ref{tab3}, we get:
			\begin{eqnarray*}
				N_{\KK/k_4}(\xi^2)&=&(-1)^a\cdot
				\varepsilon_{p}^{2a}\cdot \varepsilon_{pq}^{2a} \cdot (-1)^d\cdot\varepsilon_{q}^d\cdot 1 \cdot1  \cdot (-1)^{gu+g}\\\
				&=&	\varepsilon_{2}^{2a}\varepsilon_{pq}^{2a}\cdot(-1)^{a +d+g+gu}\varepsilon_{q}^d .
			\end{eqnarray*}
			Thus $d=0$ and so   $a +e+gu\equiv 0\pmod2$. As $a+g+ug\equiv 0\pmod2$, we have $e=g$. Since $\varepsilon_{pq}$ is a square in $\KK$ (Lemma \ref{lm expressions of units under cond 3}), we can  disregard it in the form of $\xi^2$.
			Therefore, 
			$$\xi^2=\varepsilon_{2}^a\varepsilon_{p}^a     \sqrt{\varepsilon_{2q}}^e\sqrt{\varepsilon_{2pq}}^e
			\sqrt{\varepsilon_{2p}}^e,$$
			with $a+e+eu\equiv 0\pmod2$.  So the result (cf. \eqref{classnumberofKK}).
		\end{enumerate}
	\end{proof}

	\begin{theorem}\label{T_3_1_C4} Let $p\equiv 1\pmod{8}$ and $q\equiv3\pmod 8$ be two primes such that     $\genfrac(){}{0}{p}{q} =1$. 
		Put     $\KK=\QQ(\sqrt 2, \sqrt{p}, \sqrt{q} )$.  Assume furthermore that $p(x-1)$ and $(v-1)$ are squares in $\mathbb{N}$, where $x$ and $v$ are defined in Lemma \ref{lm expressions of units under cond 3}.
		Then
		\begin{enumerate}[\rm 1.]
			\item Assume that $N(\varepsilon_{2p})=-1$.  We have
			\begin{enumerate}[\rm $\bullet$]
				
				\item The unit group of $\KK$ is :
				$$E_{\KK}=\langle -1,   \varepsilon_{2}, \varepsilon_{p} ,    \sqrt{\varepsilon_{q}}, \sqrt{\varepsilon_{2q}},\sqrt{ \varepsilon_{pq}} ,\sqrt{ \varepsilon_{2pq}}, \sqrt{\varepsilon_{2}\varepsilon_{p}\varepsilon_{2p}} \rangle.$$
				\item  The $2$-class number  of  $\KK$ equals $ \frac{1}{2^{4}} h_2(2p)h_2(pq)h_2(2pq)$.  
			\end{enumerate}
			\item Assume that $N(\varepsilon_{2p})=1$ and let be $a\in\{0,1\}$ such that $a \equiv 1+u\pmod2$.   We have
			
			\begin{enumerate}[\rm $\bullet$]
				\item  The unit group of $\KK$ is :
				$$E_{\KK}=\langle -1,   \varepsilon_{2}, \varepsilon_{p},   \sqrt{\varepsilon_{q}}, \sqrt{\varepsilon_{2q}},  \sqrt{\varepsilon_{pq}} ,\sqrt{ \varepsilon_{2pq}},    \sqrt{\varepsilon_{2}^{a\alpha}\varepsilon_{p}^{a\alpha}     \sqrt{\varepsilon_{q}}^\alpha\sqrt{\varepsilon_{pq}}^\alpha
					\sqrt{\varepsilon_{2p}}^{1+\gamma} }  \rangle$$
				where $\alpha$, $\gamma\in \{0,1\}$ are such that $\alpha\not=\gamma$ and $\alpha =1$ if and only if $\varepsilon_{2}^a \varepsilon_{p}^a      \sqrt{\varepsilon_{q}} \sqrt{\varepsilon_{pq}}\sqrt{\varepsilon_{2p}}$ is a square in $\KK$.
				\item  The $2$-class number  of  $\KK$ equals $ \frac{1}{2^{4-\alpha}} h_2(2p)h_2(pq)h_2(2pq)$.
			\end{enumerate}
		\end{enumerate}
	\end{theorem}
	\begin{proof} 
		\begin{enumerate}[\rm 1.]
			\item Assume that   $N(\varepsilon_{2p})=-1$.
			Note that $\{\varepsilon_{2}, \varepsilon_{p},	\sqrt{\varepsilon_{2}\varepsilon_{p}\varepsilon_{2p}}\}$, is  a fundamental system  of units of $k_1$.
			Using 	Lemma \ref{lm expressions of units under cond 3}, we show that $\{\varepsilon_{2}, \sqrt{\varepsilon_{q}}, \sqrt{\varepsilon_{2q}}\}$ and $\{ \varepsilon_{2}, 	 \varepsilon_{2pq},\sqrt{\varepsilon_{pq}}\}$ are respectively fundamental systems of units of $k_2$ and $k_3$.
			It follows that,  	$$E_{k_1}E_{k_2}E_{k_3}=\langle-1,  \varepsilon_{2}, \varepsilon_{p} ,   \varepsilon_{2pq} ,   \sqrt{\varepsilon_{q}}, \sqrt{\varepsilon_{2q}} ,\sqrt{ \varepsilon_{pq}}, \sqrt{\varepsilon_{2}\varepsilon_{p}\varepsilon_{2p}}\rangle.$$	
			
			Let  $\xi$ be an element of $\KK$ which is the  square root of an element of $E_{k_1}E_{k_2}E_{k_3}$. Therefore, we can assume that
			$$\xi^2=\varepsilon_{2}^a\varepsilon_{p}^b  \varepsilon_{2pq}^c  \sqrt{\varepsilon_{q}}^d\sqrt{\varepsilon_{2q}}^e\sqrt{\varepsilon_{pq}}^f
			\sqrt{\varepsilon_{2}\varepsilon_{p}\varepsilon_{2p}}^g,$$
			where $a, b, c, d, e, f$ and $g$ are in $\{0, 1\}$.
			
			\noindent\ding{224}  Let us start	by applying   the norm map $N_{\KK/k_2}=1+\tau_2$. % We have,  
			%$\sqrt{\varepsilon_{pq}}^{1+ \tau_2}=-1$ and 
			%$\sqrt{\varepsilon_{2pq}}^{1+ \tau_2}=1$. Then, 
			By Tables \ref{norm q=3 p=1mod4q}  and \ref{tab3}, we have:
			\begin{eqnarray*}
				N_{\KK/k_2}(\xi^2)&=&
				\varepsilon_{2}^{2a}\cdot (-1)^b \cdot1\cdot \varepsilon_{q}^d\cdot \varepsilon_{2q}^e\cdot(-1)^f  \cdot (-1)^{gu}\varepsilon_{2}^g\\\
				&=&	\varepsilon_{2}^{2a}  \varepsilon_{q}^c\varepsilon_{2q}^d\cdot(-1)^{b +f+gu}\varepsilon_{2}^g .
			\end{eqnarray*}
			Thus,  $b +f+gu\equiv 0 \pmod 2$ and $g=0$. Then $b=f$ and 
			$$\xi^2=\varepsilon_{2}^a\varepsilon_{p}^f  \varepsilon_{2pq}^c  \sqrt{\varepsilon_{q}}^d\sqrt{\varepsilon_{2q}}^e\sqrt{\varepsilon_{pq}}^f.$$
			
			\noindent\ding{224} Let us apply the norm $N_{\KK/k_5}=1+\tau_1\tau_2$, with $k_5=\QQ(\sqrt{q}, \sqrt{2p})$.  % We have 
			%	$\sqrt{\varepsilon_{pq}}^{1+ \tau_1\tau_2}=1$ and 
			%	$\sqrt{\varepsilon_{2pq}}^{1+ \tau_1\tau_2}=\varepsilon_{2pq}$. 
			Then, by Tables \ref{norm q=3 p=1mod4q}  and \ref{tab3}, we have: 
			\begin{eqnarray*}
				N_{\KK/k_5}(\xi^2)&=&(-1)^a\cdot (-1)^f\cdot \varepsilon_{2pq}^{2c}\cdot (-1)^d\cdot \varepsilon_{  q}^d\cdot1 \cdot 1\\
				&=&	 \varepsilon_{2pq}^{2c}(-1)^{a +d }\varepsilon_{  q}^d.
			\end{eqnarray*}
			
			Thus $a=d=0$ and 
			$$\xi^2=\varepsilon_{p}^f  \varepsilon_{2pq}^c  \sqrt{\varepsilon_{2q}}^e\sqrt{\varepsilon_{pq}}^f.$$

			\noindent\ding{224} Let us apply the norm $N_{\KK/k_6}=1+\tau_1\tau_3$, with $k_6=\QQ(\sqrt{p}, \sqrt{2q})$. % We have 
			%	$\sqrt{\varepsilon_{pq}}^{1+ \tau_1\tau_2}=1$ and 
			%	$\sqrt{\varepsilon_{2pq}}^{1+ \tau_1\tau_2}=-\varepsilon_{2pq}$.
			By Tables \ref{norm q=3 p=1mod4q}  and \ref{tab3}, we have: 
			\begin{eqnarray*}
				N_{\KK/k_6}(\xi^2)&=& \varepsilon_{  p}^{2f}\cdot \varepsilon_{2pq}^{2c}\cdot (-1)^e\cdot \varepsilon_{2  q}^e\cdot  1\\
				&=&	\varepsilon_{  p}^{2a}\varepsilon_{2pq}^{2c} (-1)^{e}\varepsilon_{2  q}^e .
			\end{eqnarray*}
			Therefore $e=0$ and 
			$$\xi^2=\varepsilon_{p}^f  \varepsilon_{2pq}^c  \sqrt{\varepsilon_{pq}}^f.$$
			
			\noindent\ding{224} Let us apply       the norm map $N_{\KK/k_3}=1+\tau_2\tau_3$, with $k_3=\QQ(\sqrt{2}, \sqrt{pq})$.   %We have 
			%	$\sqrt{\varepsilon_{pq}}^{1+ \tau_1\tau_2}=\varepsilon_{pq}$ and 
			%	$\sqrt{\varepsilon_{2pq}}^{1+ \tau_1\tau_2}=-\varepsilon_{2pq}$. 
			By Tables \ref{norm q=3 p=1mod4q}  and \ref{tab3}, we get:
			\begin{eqnarray*}
				N_{\KK/k_3}(\xi^2)&=& (-1)^f\cdot \varepsilon_{2pq}^{2c}\cdot \varepsilon_{pq}^{2f}=\varepsilon_{2pq}^{2c} \varepsilon_{pq}^{2f}(-1)^f.
			\end{eqnarray*}
			Thus, $f=0$ and 
			$$\xi^2=  \varepsilon_{2pq}^c .$$
			Since   $\varepsilon_{2pq}$ is a square in $\KK$, then we have the first item.

			\item	 Assume that   $N(\varepsilon_{2p})=1$. In this case we have: $\{\varepsilon_{2}, \varepsilon_{p},	\sqrt{ \varepsilon_{2p}}\}$, is  a fundamental system  of units of $k_1$.
			Using 	Lemma \ref{lm expressions of units under cond 3}, we show that $\{\varepsilon_{2}, \sqrt{\varepsilon_{q}}, \sqrt{\varepsilon_{2q}}\}$ and $\{ \varepsilon_{2}, 	 \varepsilon_{2pq},\sqrt{\varepsilon_{pq}}\}$ are respectively fundamental systems of units of $k_2$ and $k_3$.
			Thus, 	$$E_{k_1}E_{k_2}E_{k_3}=\langle-1,  \varepsilon_{2}, \varepsilon_{p},   {\varepsilon_{2pq}} , \sqrt{\varepsilon_{q}}, \sqrt{\varepsilon_{2q}},  \sqrt{ \varepsilon_{pq}}, \sqrt{\varepsilon_{2p}}\rangle.$$	
			Let  $\xi$ be an element of $\KK$ which is the  square root of an element of $E_{k_1}E_{k_2}E_{k_3}$. Therefore, we can assume that 
			$$\xi^2=\varepsilon_{2}^a\varepsilon_{p}^b  \varepsilon_{2pq}^c  \sqrt{\varepsilon_{q}}^d\sqrt{\varepsilon_{2q}}^e\sqrt{\varepsilon_{pq}}^f
			\sqrt{\varepsilon_{2p}}^g,$$
			where $a, b, c, d, e, f$ and $g$ are in $\{0, 1\}$.
			
			\noindent\ding{224}  Let us start	by applying   the norm map $N_{\KK/k_2}=1+\tau_2$. We have
			\begin{eqnarray*}
				N_{\KK/k_2}(\xi^2)&=&
				\varepsilon_{2}^{2a}\cdot (-1)^b \cdot1\cdot \varepsilon_{q}^d\cdot \varepsilon_{2q}^e\cdot(-1)^f  \cdot (-1)^{gu}\\
				&=&	\varepsilon_{2}^{2a}  \varepsilon_{q}^c\varepsilon_{2q}^d\cdot(-1)^{b +f+gu}  .
			\end{eqnarray*}
			Thus, $b +f+gu\equiv 0\pmod2$.
			
			\noindent\ding{224} Let us apply the norm $N_{\KK/k_5}=1+\tau_1\tau_2$, with $k_5=\QQ(\sqrt{q}, \sqrt{2p})$. We have 
			\begin{eqnarray*}
				N_{\KK/k_5}(\xi^2)&=&(-1)^a\cdot (-1)^b\cdot \varepsilon_{2pq}^{2c}\cdot (-1)^d\cdot \varepsilon_{  q}^d\cdot1 \cdot 1\cdot (-1)^{g}\cdot\varepsilon_{2p}^g \\
				&=&	 \varepsilon_{2pq}^{2c}(-1)^{a +b+d+g }\varepsilon_{  q}^d\varepsilon_{2p}^g.
			\end{eqnarray*}
			Thus, $a +b+d+g\equiv 0\pmod2$ and $d=g$. So $a=b$. Therefore, 
			$$\xi^2=\varepsilon_{2}^a\varepsilon_{p}^a  \varepsilon_{2pq}^c  \sqrt{\varepsilon_{q}}^g\sqrt{\varepsilon_{2q}}^e\sqrt{\varepsilon_{pq}}^f
			\sqrt{\varepsilon_{2p}}^g.$$
			
			\noindent\ding{224} Let us apply the norm $N_{\KK/k_6}=1+\tau_1\tau_3$, with $k_6=\QQ(\sqrt{p}, \sqrt{2q})$. We have 
			\begin{eqnarray*}
				N_{\KK/k_6}(\xi^2)&=& (-1)^{a}\cdot \varepsilon_{  p}^{2a}\cdot \varepsilon_{2pq}^{2c}\cdot1\cdot (-1)^e\cdot \varepsilon_{2  q}^e\cdot  1\cdot(-1)^{gu +g} \\
				&=&	\varepsilon_{  p}^{2a}\varepsilon_{2pq}^{2c} (-1)^{a+e+g+ug}\varepsilon_{2  q}^e .
			\end{eqnarray*}
			Thus, $a+e+g+ug\equiv 0\pmod2$ and $e=0$. Therefore, $a+g+ug\equiv 0\pmod2$. Since 
			$a +f+gu\equiv 0\pmod2$, we have $f=g$. As $\varepsilon_{2pq}$ is a square in $\KK$ then we may put:
			$$\xi^2=\varepsilon_{2}^a\varepsilon_{p}^a     \sqrt{\varepsilon_{q}}^g\sqrt{\varepsilon_{pq}}^g
			\sqrt{\varepsilon_{2p}}^g,$$
			where $a+g+ug\equiv 0\pmod2$.
			Which gives the result (cf. \eqref{classnumberofKK}). 		\end{enumerate}
	\end{proof}

	\begin{theorem}\label{T_3_1_C5} Let $p\equiv 1\pmod{8}$ and $q\equiv3\pmod 8$ be two primes such that     $\genfrac(){}{0}{p}{q} =1$. 
		Put     $\KK=\QQ(\sqrt 2, \sqrt{p}, \sqrt{q} )$.  Assume furthermore that $p(x-1)$ and $p(v-1)$ are squares in $\mathbb{N}$, where $x$ and $v$ are defined in Lemma \ref{lm expressions of units under cond 3}.
		Then
		\begin{enumerate}[\rm 1.]
			\item Assume that $N(\varepsilon_{2p})=-1$.  We have
			\begin{enumerate}[\rm $\bullet$]
				
				\item The unit group of $\KK$ is :
				$$E_{\KK}=\langle -1,   \varepsilon_{2}, \varepsilon_{p} ,    \sqrt{\varepsilon_{q}}, \sqrt{\varepsilon_{2q}},\sqrt{ \varepsilon_{pq}} ,\sqrt{ \varepsilon_{2pq}}, \sqrt{\varepsilon_{2}\varepsilon_{p}\varepsilon_{2p}} \rangle.$$
				\item  The $2$-class number  of  $\KK$ equals $ \frac{1}{2^{4}} h_2(2p)h_2(pq)h_2(2pq)$.  
			\end{enumerate}
			\item Assume that $N(\varepsilon_{2p})=1$ and let  $a\in\{0,1\}$ be such that $a \equiv 1+u\pmod2$.   We have
			
			\begin{enumerate}[\rm $\bullet$]
				\item  The unit group of $\KK$ is :
				$$E_{\KK}=\langle -1,   \varepsilon_{2}, \varepsilon_{p},   \sqrt{\varepsilon_{q}}, \sqrt{\varepsilon_{2q}},  \sqrt{\varepsilon_{pq}} ,\sqrt{ \varepsilon_{2pq}},    \sqrt{\varepsilon_{2}^{a\alpha}\varepsilon_{p}^{u\alpha}     \sqrt{\varepsilon_{q}}^\alpha \sqrt{\varepsilon_{2q}}^\alpha\sqrt{\varepsilon_{pq}}^\alpha\sqrt{\varepsilon_{2pq}}^\alpha
					\sqrt{\varepsilon_{2p}}^{1+\gamma} }  \rangle$$
				where $\alpha$, $\gamma\in \{0,1\}$ are such that $\alpha\not=\gamma$ and $\alpha =1$ if and only if $ \varepsilon_{2}^{a}\varepsilon_{p}^{u}     \sqrt{\varepsilon_{q}} \sqrt{\varepsilon_{2q}}\sqrt{\varepsilon_{pq}}\sqrt{\varepsilon_{2pq}}
				\sqrt{\varepsilon_{2p}}  $ is a square in $\KK$.
				\item  The $2$-class number  of  $\KK$ equals $ \frac{1}{2^{4-\alpha}} h_2(2p)h_2(pq)h_2(2pq)$.
			\end{enumerate}
		\end{enumerate}
	\end{theorem}
	\begin{proof} 		\begin{enumerate}[\rm 1.] 
			\item Assume that   $N(\varepsilon_{2p})=-1$.
			Note that $\{\varepsilon_{2}, \varepsilon_{p},	\sqrt{\varepsilon_{2}\varepsilon_{p}\varepsilon_{2p}}\}$, is  a fundamental system  of units of $k_1$.
			Using 	Lemma \ref{lm expressions of units under cond 3}, we show that $\{\varepsilon_{2}, \sqrt{\varepsilon_{q}}, \sqrt{\varepsilon_{2q}}\}$ and  $\{ \varepsilon_{2}, 	 \varepsilon_{pq},\sqrt{\varepsilon_{pq}\varepsilon_{2pq}}\}$ are respectively fundamental systems of units of $k_2$ and $k_3$.
			It follows that, 
			$$E_{k_1}E_{k_2}E_{k_3}=\langle-1,  \varepsilon_{2}, \varepsilon_{p} ,   \varepsilon_{pq} ,   \sqrt{\varepsilon_{q}}, \sqrt{\varepsilon_{2q}} ,\sqrt{ \varepsilon_{pq}\varepsilon_{2pq}}, \sqrt{\varepsilon_{2}\varepsilon_{p}\varepsilon_{2p}}\rangle.$$	
			
			Let  $\xi$ be an element of $\KK$ which is the  square root of an element of $E_{k_1}E_{k_2}E_{k_3}$. Therefore, we can assume that
			$$\xi^2=\varepsilon_{2}^a\varepsilon_{p}^b  \varepsilon_{pq}^c  \sqrt{\varepsilon_{q}}^d\sqrt{\varepsilon_{2q}}^e\sqrt{\varepsilon_{pq}\varepsilon_{2pq}}^f
			\sqrt{\varepsilon_{2}\varepsilon_{p}\varepsilon_{2p}}^g,$$
			where $a, b, c, d, e, f$ and $g$ are in $\{0, 1\}$.
			
			\noindent\ding{224}  Let us start	by applying   the norm map $N_{\KK/k_2}=1+\tau_2$.   We have 
			%$\sqrt{\varepsilon_{pq}}^{1+  \tau_2}=1$ and 
			%	$\sqrt{\varepsilon_{2pq}}^{1+  \tau_2}=1$. 
			Then, by Tables \ref{norm q=3 p=1mod4q}  and \ref{tab3}, we get: 
			\begin{eqnarray*}
				N_{\KK/k_2}(\xi^2)&=&
				\varepsilon_{2}^{2a}\cdot (-1)^b \cdot1\cdot \varepsilon_{q}^d\cdot \varepsilon_{2q}^e\cdot1  \cdot (-1)^{gu}\varepsilon_{2}^g\\\
				&=&	\varepsilon_{2}^{2a}  \varepsilon_{q}^c\varepsilon_{2q}^d\cdot(-1)^{b +gu}\varepsilon_{2}^g .
			\end{eqnarray*}
			Thus, $ b +gu\equiv 0 \pmod 2$ and $g=0$. Therefore, $b=0$ and 
			$$\xi^2=\varepsilon_{2}^a  \varepsilon_{pq}^c  \sqrt{\varepsilon_{q}}^d\sqrt{\varepsilon_{2q}}^e\sqrt{\varepsilon_{pq}\varepsilon_{2pq}}^f.$$

			\noindent\ding{224} Let us apply the norm $N_{\KK/k_5}=1+\tau_1\tau_2$, with $k_5=\QQ(\sqrt{q}, \sqrt{2p})$. 	%	We have 
			%	$\sqrt{\varepsilon_{pq}}^{1+ \tau_1\tau_2}=-1$ and 
			%	$\sqrt{\varepsilon_{2pq}}^{1+ \tau_1\tau_2}=\varepsilon_{2pq}$. 
			By Tables \ref{norm q=3 p=1mod4q}  and \ref{tab3}, we get: 
			\begin{eqnarray*}
				N_{\KK/k_5}(\xi^2)&=&(-1)^a\cdot 1\cdot   (-1)^d\cdot \varepsilon_{  q}^d\cdot1 \cdot (-1)^f\cdot \varepsilon_{2pq}^f\\
				&=&	\varepsilon_{2pq}^f (-1)^{a +d +f}\varepsilon_{  q}^d.
			\end{eqnarray*}
			Thus,   $a +d +f\equiv 0 \pmod 2$ and $d=0$. Therefore, $a=f$ and 
			$$\xi^2=\varepsilon_{2}^a  \varepsilon_{pq}^c   \sqrt{\varepsilon_{2q}}^e\sqrt{\varepsilon_{pq}\varepsilon_{2pq}}^a.$$

			\noindent\ding{224} Let us apply the norm $N_{\KK/k_6}=1+\tau_1\tau_3$, with $k_6=\QQ(\sqrt{p}, \sqrt{2q})$. % We have 
			%	$\sqrt{\varepsilon_{pq}}^{1+ \tau_1\tau_3}=1$ and 
			%	$\sqrt{\varepsilon_{2pq}}^{1+ \tau_1\tau_3}=-\varepsilon_{2pq}$. 
			Then, by Tables \ref{norm q=3 p=1mod4q}  and \ref{tab3}, we get: 
			\begin{eqnarray*}
				N_{\KK/k_6}(\xi^2)&=& (-1)^a \cdot 1  \cdot (-1)^e\cdot \varepsilon_{2  q}^e\cdot  (-1)^a \cdot \varepsilon_{2pq}^a\\
				&=&	 (-1)^{e}\varepsilon_{2  q}^e\varepsilon_{2pq}^a .
			\end{eqnarray*}
			Thus   $e=a=0$. Therefore $e=0$ and 
			$$\xi^2=   \varepsilon_{pq}^c .$$
			Since by the above Lemma $\varepsilon_{pq}$ is a square in $\KK$ so we have the first item.
			
			\item		 Assume that   $N(\varepsilon_{2p})=1$. 	We have $\{\varepsilon_{2}, \varepsilon_{p},	\sqrt{ \varepsilon_{2p}}\}$, is  a fundamental system  of units of $k_1$, and
			$\{\varepsilon_{2}, \sqrt{\varepsilon_{q}}, \sqrt{\varepsilon_{2q}}\}$ and  $\{ \varepsilon_{2}, 	 \varepsilon_{pq},\sqrt{\varepsilon_{pq}\varepsilon_{2pq}}\}$ are respectively fundamental systems of units of $k_2$ and $k_3$.
			So we have
			$$E_{k_1}E_{k_2}E_{k_3}=\langle-1,  \varepsilon_{2}, \varepsilon_{p} ,   \varepsilon_{pq} ,   \sqrt{\varepsilon_{q}}, \sqrt{\varepsilon_{2q}} ,\sqrt{ \varepsilon_{pq}\varepsilon_{2pq}},  \sqrt{ \varepsilon_{2p}}\rangle.$$	
			
			Let  $\xi$ be an element of $\KK$ which is the  square root of an element of $E_{k_1}E_{k_2}E_{k_3}$. Therefore, we can assume that
			$$\xi^2=\varepsilon_{2}^a\varepsilon_{p}^b  \varepsilon_{pq}^c  \sqrt{\varepsilon_{q}}^d\sqrt{\varepsilon_{2q}}^e\sqrt{\varepsilon_{pq}\varepsilon_{2pq}}^f
			\sqrt{ \varepsilon_{2p}}^g,$$
			where $a, b, c, d, e, f$ and $g$ are in $\{0, 1\}$.

			\noindent\ding{224}  Let us start	by applying   the norm map $N_{\KK/k_2}=1+\tau_2$. We have:
			\begin{eqnarray*}
				N_{\KK/k_2}(\xi^2)&=&
				\varepsilon_{2}^{2a}\cdot (-1)^b \cdot1\cdot \varepsilon_{q}^d\cdot \varepsilon_{2q}^e\cdot1  \cdot (-1)^{gu}\\\
				&=&	\varepsilon_{2}^{2a}  \varepsilon_{q}^c\varepsilon_{2q}^d\cdot(-1)^{b +gu}  .
			\end{eqnarray*}
			Thus, $ b +gu\equiv 0 \pmod 2$.

			\noindent\ding{224} Let us apply the norm $N_{\KK/k_5}=1+\tau_1\tau_2$, with $k_5=\QQ(\sqrt{q}, \sqrt{2p})$. We have 
			\begin{eqnarray*}
				N_{\KK/k_5}(\xi^2)&=&(-1)^a\cdot (-1)^b\cdot 1\cdot   (-1)^d\cdot \varepsilon_{  q}^d\cdot1 \cdot (-1)^f\cdot \varepsilon_{2pq}^f\cdot (-1)^g\cdot\varepsilon_{2p}^g. \\
				&=&	 \varepsilon_{2pq}^f (-1)^{a+b +d +f+g}\varepsilon_{  q}^d\varepsilon_{2p}^g.
			\end{eqnarray*}
			Thus,   $a+b +d +f+g\equiv 0 \pmod 2$ and  $d = g$. Therefore,  $a+b +f\equiv 0 \pmod 2$ and
			
			$$\xi^2=\varepsilon_{2}^a\varepsilon_{p}^b  \varepsilon_{pq}^c  \sqrt{\varepsilon_{q}}^d\sqrt{\varepsilon_{2q}}^e\sqrt{\varepsilon_{pq}\varepsilon_{2pq}}^f
			\sqrt{ \varepsilon_{2p}}^d.$$
			
			\noindent\ding{224} Let us apply the norm $N_{\KK/k_6}=1+\tau_1\tau_3$, with $k_6=\QQ(\sqrt{p}, \sqrt{2q})$. We have 
			\begin{eqnarray*}
				N_{\KK/k_6}(\xi^2)&=& (-1)^a\cdot \varepsilon_{p}^{2b} \cdot 1\cdot 1  \cdot (-1)^e\cdot \varepsilon_{2  q}^e\cdot  (-1)^f \cdot \varepsilon_{2pq}^f\cdot (-1)^{ud+d}\\
				&=&	 (-1)^{a+e+f+ug+g}\varepsilon_{2  q}^e\varepsilon_{2pq}^f .
			\end{eqnarray*}
			Thus, $ a+e+f+ud+d\equiv 0 \pmod 2$ and $e=f$. Then, $ a+ud+d\equiv 0 \pmod 2$. Since $ b +du\equiv 0 \pmod 2$ and  $a+b +f\equiv 0 \pmod 2$, we have $f=d$ and 
			
			$$\xi^2=\varepsilon_{2}^a\varepsilon_{p}^{ud}  \varepsilon_{pq}^c  \sqrt{\varepsilon_{q}}^d\sqrt{\varepsilon_{2q}}^d\sqrt{\varepsilon_{pq}\varepsilon_{2pq}}^d
			\sqrt{ \varepsilon_{2p}}^d.$$
			
			Since by the above Lemma $\varepsilon_{pq}$ is a square in $\KK$.  then we can put  
			$$\xi^2=\varepsilon_{2}^a\varepsilon_{p}^{ud}     \sqrt{\varepsilon_{q}}^d\sqrt{\varepsilon_{2q}}^d\sqrt{\varepsilon_{pq}\varepsilon_{2pq}}^d
			\sqrt{ \varepsilon_{2p}}^d.$$
			where $ a+ud+d\equiv 0 \pmod 2$. Which completes the proof. 	 		
			% 	\noindent\ding{224} Let us apply the norm $N_{\KK/k_4}=1+\tau_1$, with $k_4=\QQ(\sqrt{p}, \sqrt{q})$. We have 
			% \begin{eqnarray}\label{T_3_1_tau1_N=-1 C5}
			%	\begin{tabular}{ |c|c|c|c|c|c|c|c|c|}
			%		\hline
			%		$\varepsilon$&$\varepsilon_{2}$ &${\varepsilon_{p}}$&   $\sqrt{\varepsilon_{q}}$&$\sqrt{\varepsilon_{2q}}$&$\sqrt{\varepsilon_{pq}}$  & $\sqrt{\varepsilon_{2pq}}$ \\
			%		\hline
			%		$\varepsilon^{1+\tau_1}$&$-1$ &$\varepsilon_{p}^2$& $-\varepsilon_{q}$&$1$& $-\varepsilon_{pq}$ & $1$   \\
			%		\hline
			%	\end{tabular}
			%\end{eqnarray}
			
			%	\begin{eqnarray*}
			%	N_{\KK/k_5}(\xi^2)&=&(-1)^a\cdot \varepsilon_{  p}^{2ud}\cdot   (-1)^d\cdot \varepsilon_{  q}^d \cdot 1 \cdot (-1)^d\cdot \varepsilon_{pq}^d\cdot (-1)^{ud+d}. \\
			%	&=&	 \varepsilon_{2pq}^f (-1)^{a+b +d +f+g}\varepsilon_{  q}^d\varepsilon_{2p}^g.
			%\end{eqnarray*}
		\end{enumerate}
	\end{proof}

	\begin{theorem}\label{T_3_1_C6}  Let $p\equiv 1\pmod{8}$ and $q\equiv3\pmod 8$ be two primes such that     $\genfrac(){}{0}{p}{q} =1$. 
		Put     $\KK=\QQ(\sqrt 2, \sqrt{p}, \sqrt{q} )$.  Assume furthermore that $p(x-1)$ and $2p(v+1)$ are squares in $\mathbb{N}$, where $x$ and $v$ are defined in Lemma \ref{lm expressions of units under cond 3}.
		Then
		\begin{enumerate}[\rm 1.]
			\item Assume that $N(\varepsilon_{2p})=-1$.  We have
			\begin{enumerate}[\rm $\bullet$]
				
				\item The unit group of $\KK$ is :
				$$E_{\KK}=\langle -1,   \varepsilon_{2}, \varepsilon_{p} ,    \sqrt{\varepsilon_{q}},  \sqrt{ \varepsilon_{pq}} ,\sqrt{ \varepsilon_{2pq}}, \sqrt{\varepsilon_{2}\varepsilon_{p}\varepsilon_{2p}}, \sqrt{ \varepsilon_{p}^\alpha      \sqrt{\varepsilon_{2q}}^{1+\gamma}\sqrt{\varepsilon_{pq}\varepsilon_{2pq}}^\alpha} \rangle.$$
				where $\alpha$, $\gamma\in \{0,1\}$ are such that $\alpha\not=\gamma$ and $\alpha =1$ if and only if   $\varepsilon_{p}   \sqrt{\varepsilon_{2q}}  \sqrt{\varepsilon_{pq}\varepsilon_{2pq}} $ is a square in $\KK$
				\item  The $2$-class number  of  $\KK$ equals $ \frac{1}{2^{4}} h_2(2p)h_2(pq)h_2(2pq)$.  
			\end{enumerate}
			\item Assume that $N(\varepsilon_{2p})=1$ and let $a\in\{0,1\}$ be  such that $a \equiv 1+u\pmod2$.   We have
			
			\begin{enumerate}[\rm $\bullet$]
				\item  The unit group of $\KK$ is :
				$$E_{\KK}=\langle -1,   \varepsilon_{2}, \varepsilon_{p},   \sqrt{\varepsilon_{q}}, \sqrt{\varepsilon_{2q}},  \sqrt{\varepsilon_{pq}} ,\sqrt{ \varepsilon_{2pq}},    \sqrt{\varepsilon_{2}^{a\alpha}\varepsilon_{p}^{u\alpha}     \sqrt{\varepsilon_{q}}^\alpha \sqrt{\varepsilon_{2q}}^\alpha\sqrt{\varepsilon_{pq}}^\alpha\sqrt{\varepsilon_{2pq}}^\alpha
					\sqrt{\varepsilon_{2p}}^{1+\gamma} }  \rangle$$
				where $\alpha$, $\gamma\in \{0,1\}$ are such that $\alpha\not=\gamma$ and $\alpha =1$ if and only if $ \varepsilon_{2}^{a }\varepsilon_{p}^{u }     \sqrt{\varepsilon_{q}}  \sqrt{\varepsilon_{2q}} \sqrt{\varepsilon_{pq}} \sqrt{\varepsilon_{2pq}} 
				\sqrt{\varepsilon_{2p}}  $ is a square in $\KK$.
				\item  The $2$-class number  of  $\KK$ equals $ \frac{1}{2^{4-\alpha}} h_2(2p)h_2(pq)h_2(2pq)$.
			\end{enumerate}
		\end{enumerate}
	\end{theorem}
	\begin{proof} 		\begin{enumerate}[\rm 1.] 
			\item Assume that   $N(\varepsilon_{2p})=-1$.
			Note that $\{\varepsilon_{2}, \varepsilon_{p},	\sqrt{\varepsilon_{2}\varepsilon_{p}\varepsilon_{2p}}\}$, is  a fundamental system  of units of $k_1$.
			Using 	Lemma \ref{lm expressions of units under cond 3}, we show that $\{\varepsilon_{2}, \sqrt{\varepsilon_{q}}, \sqrt{\varepsilon_{2q}}\}$ and  $\{ \varepsilon_{2}, 	 \varepsilon_{pq},\sqrt{\varepsilon_{pq}\varepsilon_{2pq}}\}$ are respectively fundamental systems of units of $k_2$ and $k_3$.
			It follows that, 
			$$E_{k_1}E_{k_2}E_{k_3}=\langle-1,  \varepsilon_{2}, \varepsilon_{p} ,   \varepsilon_{pq} ,   \sqrt{\varepsilon_{q}}, \sqrt{\varepsilon_{2q}} ,\sqrt{ \varepsilon_{pq}\varepsilon_{2pq}}, \sqrt{\varepsilon_{2}\varepsilon_{p}\varepsilon_{2p}}\rangle.$$	
			
			Let  $\xi$ be an element of $\KK$ which is the  square root of an element of $E_{k_1}E_{k_2}E_{k_3}$. Therefore, we can assume that
			$$\xi^2=\varepsilon_{2}^a\varepsilon_{p}^b  \varepsilon_{pq}^c  \sqrt{\varepsilon_{q}}^d\sqrt{\varepsilon_{2q}}^e\sqrt{\varepsilon_{pq}\varepsilon_{2pq}}^f
			\sqrt{\varepsilon_{2}\varepsilon_{p}\varepsilon_{2p}}^g,$$
			where $a, b, c, d, e, f$ and $g$ are in $\{0, 1\}$.
			
			\noindent\ding{224}  Let us start	by applying   the norm map $N_{\KK/k_2}=1+\tau_2$. By Tables \ref{norm q=3 p=1mod4q}  and \ref{tab3}  we have 
			\begin{eqnarray*}
				N_{\KK/k_2}(\xi^2)&=&
				\varepsilon_{2}^{2a}\cdot (-1)^b \cdot1\cdot \varepsilon_{q}^d\cdot \varepsilon_{2q}^e\cdot(-1)^f  \cdot (-1)^{gu}\varepsilon_{2}^g\\\
				&=&	\varepsilon_{2}^{2a}  \varepsilon_{q}^c\varepsilon_{2q}^d\cdot(-1)^{b +f+gu}\varepsilon_{2}^g .
			\end{eqnarray*}
			Thus, $ b +f+gu\equiv 0 \pmod 2$ and $g=0$. Therefore, $b=f$ and 
			$$\xi^2=\varepsilon_{2}^a\varepsilon_{p}^b  \varepsilon_{pq}^c  \sqrt{\varepsilon_{q}}^d\sqrt{\varepsilon_{2q}}^e\sqrt{\varepsilon_{pq}\varepsilon_{2pq}}^b.$$

			\noindent\ding{224} Let us apply the norm $N_{\KK/k_5}=1+\tau_1\tau_2$, with $k_5=\QQ(\sqrt{q}, \sqrt{2p})$.  Tables \ref{norm q=3 p=1mod4q}  and \ref{tab3} give:
			%	We have 
			%	$\sqrt{\varepsilon_{pq}}^{1+ \tau_1\tau_2}=-1$ and 
			%$\sqrt{\varepsilon_{2pq}}^{1+ \tau_1\tau_2}=\varepsilon_{2pq}$. Then, by \eqref{norm q=3 p=1mod4q}, we get: 
			\begin{eqnarray*}
				N_{\KK/k_5}(\xi^2)&=&(-1)^a\cdot(-1)^b\cdot 1\cdot   (-1)^d\cdot \varepsilon_{  q}^d\cdot1 \cdot (-1)^b\cdot \varepsilon_{2pq}^b\\
				&=&	\varepsilon_{2pq}^b (-1)^{a +d }\varepsilon_{  q}^d.
			\end{eqnarray*}
			Thus,   $a=d=0$. Therefore,  
			$$\xi^2= \varepsilon_{p}^b  \varepsilon_{pq}^c   \sqrt{\varepsilon_{2q}}^e\sqrt{\varepsilon_{pq}\varepsilon_{2pq}}^b.$$

			\noindent\ding{224} Let us apply the norm $N_{\KK/k_6}=1+\tau_1\tau_3$, with $k_6=\QQ(\sqrt{p}, \sqrt{2q})$. So by Tables \ref{norm q=3 p=1mod4q}  and \ref{tab3}  we have %We have 
			%	$\sqrt{\varepsilon_{pq}}^{1+ \tau_1\tau_3}=1$ and 
			%	$\sqrt{\varepsilon_{2pq}}^{1+ \tau_1\tau_3}=-\varepsilon_{2pq}$. Then, by \eqref{norm q=3 p=1mod4q}, we get: 
			\begin{eqnarray*}
				N_{\KK/k_6}(\xi^2)&=& \varepsilon_{2  p}^{2b} \cdot 1  \cdot (-1)^e\cdot \varepsilon_{2  q}^e\cdot  (-1)^b \cdot \varepsilon_{2pq}^b\\
				&=&	 \varepsilon_{2  p}^{2b}(-1)^{e+b}\varepsilon_{2  q}^e\varepsilon_{2pq}^b .
			\end{eqnarray*}
			Thus   $b=e$. Therefore   
			$$\xi^2= \varepsilon_{p}^b  \varepsilon_{pq}^c   \sqrt{\varepsilon_{2q}}^b\sqrt{\varepsilon_{pq}\varepsilon_{2pq}}^b.$$
			
			By applying the other norms we deduce nothing new.	Since by the above Lemma $\varepsilon_{pq}$ is a square in $\KK$, we have the first item.
			
			\item		 Assume that   $N(\varepsilon_{2p})=1$. 	We have $\{\varepsilon_{2}, \varepsilon_{p},	\sqrt{ \varepsilon_{2p}}\}$  is  a fundamental system  of units of $k_1$, and
			$\{\varepsilon_{2}, \sqrt{\varepsilon_{q}}, \sqrt{\varepsilon_{2q}}\}$ and  $\{ \varepsilon_{2}, 	 \varepsilon_{pq},\sqrt{\varepsilon_{pq}\varepsilon_{2pq}}\}$ are respectively fundamental systems of units of $k_2$ and $k_3$.
			So we have
			$$E_{k_1}E_{k_2}E_{k_3}=\langle-1,  \varepsilon_{2}, \varepsilon_{p} ,   \varepsilon_{pq} ,   \sqrt{\varepsilon_{q}}, \sqrt{\varepsilon_{2q}} ,\sqrt{ \varepsilon_{pq}\varepsilon_{2pq}},  \sqrt{ \varepsilon_{2p}}\rangle.$$	
			
			Let  $\xi$ be an element of $\KK$ which is the  square root of an element of $E_{k_1}E_{k_2}E_{k_3}$. Therefore, we can assume that
			$$\xi^2=\varepsilon_{2}^a\varepsilon_{p}^b  \varepsilon_{pq}^c  \sqrt{\varepsilon_{q}}^d\sqrt{\varepsilon_{2q}}^e\sqrt{\varepsilon_{pq}\varepsilon_{2pq}}^f
			\sqrt{ \varepsilon_{2p}}^g,$$
			where $a, b, c, d, e, f$ and $g$ are in $\{0, 1\}$.

			\noindent\ding{224}  Let us start	by applying   the norm map $N_{\KK/k_2}=1+\tau_2$. We have:
			\begin{eqnarray*}
				N_{\KK/k_2}(\xi^2)&=&
				\varepsilon_{2}^{2a}\cdot (-1)^b \cdot1\cdot \varepsilon_{q}^d\cdot \varepsilon_{2q}^e\cdot1  \cdot (-1)^{gu}\\\
				&=&	\varepsilon_{2}^{2a}  \varepsilon_{q}^c\varepsilon_{2q}^d\cdot(-1)^{b +gu}  .
			\end{eqnarray*}
			Thus, $ b +gu\equiv 0 \pmod 2$.

			\noindent\ding{224} Let us apply the norm $N_{\KK/k_5}=1+\tau_1\tau_2$, with $k_5=\QQ(\sqrt{q}, \sqrt{2p})$. We have 
			\begin{eqnarray*}
				N_{\KK/k_5}(\xi^2)&=&(-1)^a\cdot (-1)^b\cdot 1\cdot   (-1)^d\cdot \varepsilon_{  q}^d\cdot1 \cdot (-1)^f\cdot \varepsilon_{2pq}^f\cdot (-1)^g\cdot\varepsilon_{2p}^g. \\
				&=&	 \varepsilon_{2pq}^f (-1)^{a+b +d +f+g}\varepsilon_{  q}^d\varepsilon_{2p}^g.
			\end{eqnarray*}
			Thus,   $a+b +d +f+g\equiv 0 \pmod 2$ and  $d = g$. Therefore,  $a+b +f\equiv 0 \pmod 2$ and
			
			$$\xi^2=\varepsilon_{2}^a\varepsilon_{p}^b  \varepsilon_{pq}^c  \sqrt{\varepsilon_{q}}^d\sqrt{\varepsilon_{2q}}^e\sqrt{\varepsilon_{pq}\varepsilon_{2pq}}^f
			\sqrt{ \varepsilon_{2p}}^d.$$
			
			\noindent\ding{224} Let us apply the norm $N_{\KK/k_6}=1+\tau_1\tau_3$, with $k_6=\QQ(\sqrt{p}, \sqrt{2q})$. We have 
			\begin{eqnarray*}
				N_{\KK/k_6}(\xi^2)&=& (-1)^a\cdot \varepsilon_{p}^{2b} \cdot 1\cdot 1  \cdot (-1)^e\cdot \varepsilon_{2  q}^e\cdot  (-1)^f \cdot \varepsilon_{2pq}^f\cdot (-1)^{ud+d}\\
				&=&	\varepsilon_{p}^{2b} (-1)^{a+e+f+ud+d}\varepsilon_{2  q}^e\varepsilon_{2pq}^f .
			\end{eqnarray*}
			Thus, $ a+e+f+ud+d\equiv 0 \pmod 2$ and $e=f$. Then, $ a+ud+d\equiv 0 \pmod 2$. Since $ b +du\equiv 0 \pmod 2$ and  $a+b +f\equiv 0 \pmod 2$, we have $f=d$ and 
			
			$$\xi^2=\varepsilon_{2}^a\varepsilon_{p}^{ud}  \varepsilon_{pq}^c  \sqrt{\varepsilon_{q}}^d\sqrt{\varepsilon_{2q}}^d\sqrt{\varepsilon_{pq}\varepsilon_{2pq}}^d
			\sqrt{ \varepsilon_{2p}}^d.$$
			
			Since by the above Lemma $\varepsilon_{pq}$ is a square in $\KK$.  then we can put  
			$$\xi^2=\varepsilon_{2}^a\varepsilon_{p}^{ud}     \sqrt{\varepsilon_{q}}^d\sqrt{\varepsilon_{2q}}^d\sqrt{\varepsilon_{pq}\varepsilon_{2pq}}^d
			\sqrt{ \varepsilon_{2p}}^d.$$
			where $ a+ud+d\equiv 0 \pmod 2$. Which completes the proof. 	 		
			
		\end{enumerate}
	\end{proof}

	\begin{theorem}\label{T_3_1_C7} Let $p\equiv 1\pmod{8}$ and $q\equiv3\pmod 8$ be two primes such that     $\genfrac(){}{0}{p}{q} =1$. 
		Put     $\KK=\QQ(\sqrt 2, \sqrt{p}, \sqrt{q} )$.  Assume furthermore that $2p(x+1)$ and $ (v-1)$ are squares in $\mathbb{N}$, where $x$ and $v$ are defined in Lemma \ref{lm expressions of units under cond 3}.
		Then
		\begin{enumerate}[\rm 1.]
			\item Assume that $N(\varepsilon_{2p})=-1$.  We have
			\begin{enumerate}[\rm $\bullet$]
				
				\item The unit group of $\KK$ is :
				$$E_{\KK}=\langle -1,   \varepsilon_{2}, \varepsilon_{p} ,    \sqrt{\varepsilon_{q}}, \sqrt{\varepsilon_{2q}},\sqrt{ \varepsilon_{pq}} ,\sqrt{ \varepsilon_{2pq}}, \sqrt{\varepsilon_{2}\varepsilon_{p}\varepsilon_{2p}} \rangle.$$
				\item  The $2$-class number  of  $\KK$ equals $ \frac{1}{2^{4}} h_2(2p)h_2(pq)h_2(2pq)$.  
			\end{enumerate}
			\item Assume that $N(\varepsilon_{2p})=1$ and let $a\in\{0,1\}$ be such that $a \equiv 1+u\pmod2$.   We have
			
			\begin{enumerate}[\rm $\bullet$]
				\item  The unit group of $\KK$ is :
				$$E_{\KK}=\langle -1,   \varepsilon_{2}, \varepsilon_{p},   \sqrt{\varepsilon_{q}}, \sqrt{\varepsilon_{2q}},  \sqrt{\varepsilon_{pq}} ,\sqrt{ \varepsilon_{2pq}},    \sqrt{\varepsilon_{2}^{a\alpha}\varepsilon_{p}^{a\alpha}     \sqrt{\varepsilon_{q}}^{\alpha} \sqrt{\varepsilon_{pq}}^{\alpha}
					\sqrt{ \varepsilon_{2p}}^{1+\gamma}}  
				\rangle$$
				where $\alpha$, $\gamma\in \{0,1\}$ are such that $\alpha\not=\gamma$ and $\alpha =1$ if and only if $ \varepsilon_{2}^{a }\varepsilon_{p}^{a }     \sqrt{\varepsilon_{q}}  \sqrt{\varepsilon_{pq}} 
				\sqrt{ \varepsilon_{2p}} $ is a square in $\KK$.
				\item  The $2$-class number  of  $\KK$ equals $ \frac{1}{2^{4-\alpha}} h_2(2p)h_2(pq)h_2(2pq)$.
			\end{enumerate}
		\end{enumerate}
	\end{theorem}  	
	\begin{proof} 		\begin{enumerate}[\rm 1.] 
			\item Assume that   $N(\varepsilon_{2p})=-1$.
			Note that $\{\varepsilon_{2}, \varepsilon_{p},	\sqrt{\varepsilon_{2}\varepsilon_{p}\varepsilon_{2p}}\}$, is  a fundamental system  of units of $k_1$.
			Using 	Lemma \ref{lm expressions of units under cond 3}, we show that $\{\varepsilon_{2}, \sqrt{\varepsilon_{q}}, \sqrt{\varepsilon_{2q}}\}$ and  $\{ \varepsilon_{2}, 	 \varepsilon_{2pq},\sqrt{\varepsilon_{pq}}\}$ are respectively fundamental systems of units of $k_2$ and $k_3$.
			It follows that, 
			$$E_{k_1}E_{k_2}E_{k_3}=\langle-1,  \varepsilon_{2}, \varepsilon_{p} ,   \varepsilon_{2pq} ,   \sqrt{\varepsilon_{q}}, \sqrt{\varepsilon_{2q}} ,\sqrt{ \varepsilon_{pq}}, \sqrt{\varepsilon_{2}\varepsilon_{p}\varepsilon_{2p}}\rangle.$$	
			
			%	Notice that, by 	Lemma \ref{lm expressions of units under cond 3}, $\varepsilon_{2pq}$ is a square in in $\KK$.
			Let  $\xi$ be an element of $\KK$ which is the  square root of an element of $E_{k_1}E_{k_2}E_{k_3}$. So, we can assume that
			$$\xi^2=\varepsilon_{2}^a\varepsilon_{p}^b  \varepsilon_{2pq}^c  \sqrt{\varepsilon_{q}}^d\sqrt{\varepsilon_{2q}}^e\sqrt{\varepsilon_{pq}}^f
			\sqrt{\varepsilon_{2}\varepsilon_{p}\varepsilon_{2p}}^g,$$
			where $a, b, c, d, e, f$ and $g$ are in $\{0, 1\}$.
			
			\noindent\ding{224}  Let us start	by applying   the norm map $N_{\KK/k_2}=1+\tau_2$. By Tables \ref{norm q=3 p=1mod4q}  and \ref{tab3}  we have 
			\begin{eqnarray*}
				N_{\KK/k_2}(\xi^2)&=&
				\varepsilon_{2}^{2a}\cdot (-1)^b \cdot1\cdot \varepsilon_{q}^d\cdot \varepsilon_{2q}^e\cdot(-1)^{f}  \cdot (-1)^{gu}\varepsilon_{2}^g\\\
				&=&	\varepsilon_{2}^{2a}  \varepsilon_{q}^c\varepsilon_{2q}^d\cdot(-1)^{b +f+gu}\varepsilon_{2}^g .
			\end{eqnarray*}
			Thus, $ b + f+gu\equiv 0 \pmod 2$ and $g=0$. Therefore, $b=f$ and 
			$$\xi^2=\varepsilon_{2}^a\varepsilon_{p}^b  \varepsilon_{2pq}^c  \sqrt{\varepsilon_{q}}^d\sqrt{\varepsilon_{2q}}^e\sqrt{\varepsilon_{pq}}^b
			.$$

			\noindent\ding{224} Let us apply the norm $N_{\KK/k_5}=1+\tau_1\tau_2$, with $k_5=\QQ(\sqrt{q}, \sqrt{2p})$.  Tables \ref{norm q=3 p=1mod4q}  and \ref{tab3} give:
			%	We have 
			%	$\sqrt{\varepsilon_{pq}}^{1+ \tau_1\tau_2}=-1$ and 
			%$\sqrt{\varepsilon_{2pq}}^{1+ \tau_1\tau_2}=\varepsilon_{2pq}$. Then, by \eqref{norm q=3 p=1mod4q}, we get: 
			\begin{eqnarray*}
				N_{\KK/k_5}(\xi^2)&=&(-1)^a\cdot \varepsilon_{p}^{2b}\cdot \varepsilon_{2pq}^{2c}\cdot   (-1)^d\cdot \varepsilon_{  q}^d\cdot1 \cdot   1\\
				&=&	  \varepsilon_{p}^{2b}  \varepsilon_{2pq}^{2c} (-1)^{a +d  }\varepsilon_{  q}^d.
			\end{eqnarray*}
			Thus,   $a  =  d=0$. Therefore,   
			$$\xi^2= \varepsilon_{p}^b  \varepsilon_{2pq}^c  \sqrt{\varepsilon_{2q}}^e\sqrt{\varepsilon_{pq}}^b
			.$$

			\noindent\ding{224} Let us apply the norm $N_{\KK/k_6}=1+\tau_1\tau_3$, with $k_6=\QQ(\sqrt{p}, \sqrt{2q})$. So by Tables \ref{norm q=3 p=1mod4q}  and \ref{tab3}  we have %We 
			\begin{eqnarray*}
				N_{\KK/k_6}(\xi^2)&=& \varepsilon_{p}^{2b} \cdot  \varepsilon_{2pq}^{2c}\cdot(-1)^{ e}  \cdot \varepsilon_{2q}^e\cdot 1\\
				&=&	 \varepsilon_{p}^{2b}\varepsilon_{2pq}^{2c} (-1)^{e}\varepsilon_{2  q}^e.
			\end{eqnarray*}
			So $e=0$ and therefore, 	 
			$$\xi^2= \varepsilon_{p}^b  \varepsilon_{2pq}^c  \sqrt{\varepsilon_{pq}}^b
			.$$
			\noindent\ding{224} Let us apply       the norm map $N_{\KK/k_3}=1+\tau_2\tau_3$, with $k_3=\QQ(\sqrt{2}, \sqrt{pq})$.  
			So, by Tables \ref{norm q=3 p=1mod4q}  and \ref{tab3}, we get:
			\begin{eqnarray*}
				N_{\KK/k_3}(\xi^2)&=&  (-1)^b \cdot \varepsilon_{2pq}^{2c} \cdot \varepsilon_{ pq}^e \\\
				&=&	\varepsilon_{2pq}^{2c} \cdot(-1)^{b} \varepsilon_{ pq}^b .
			\end{eqnarray*}
			Thus,  	  $b=0$ and so  $\xi^2=    \varepsilon_{2pq}^c $. Since  by Lemma   \ref{lm expressions of units under cond 3}, $ \varepsilon_{2pq}$ is a square in $\KK$, then we have the result.

			\item		 Assume that   $N(\varepsilon_{2p})=1$. In this case we have 
			$$E_{k_1}E_{k_2}E_{k_3}=\langle-1,  \varepsilon_{2}, \varepsilon_{p} ,   \varepsilon_{2pq} ,   \sqrt{\varepsilon_{q}}, \sqrt{\varepsilon_{2q}} ,\sqrt{ \varepsilon_{pq}}, \sqrt{ \varepsilon_{2p}}\rangle.$$	
			
			Let  $\xi$ be an element of $\KK$ which is the  square root of an element of $E_{k_1}E_{k_2}E_{k_3}$. So, we can assume that
			$$\xi^2=\varepsilon_{2}^a\varepsilon_{p}^b  \varepsilon_{2pq}^c  \sqrt{\varepsilon_{q}}^d\sqrt{\varepsilon_{2q}}^e\sqrt{\varepsilon_{pq}}^f
			\sqrt{ \varepsilon_{2p}}^g,$$
			where $a, b, c, d, e, f$ and $g$ are in $\{0, 1\}$.

			\noindent\ding{224}  Let us start	by applying   the norm map $N_{\KK/k_2}=1+\tau_2$. We have:
			\begin{eqnarray*}
				N_{\KK/k_2}(\xi^2)&=&
				\varepsilon_{2}^{2a}\cdot (-1)^b \cdot1\cdot \varepsilon_{q}^d\cdot \varepsilon_{2q}^e\cdot(-1)^f  \cdot (-1)^{gu}\\\
				&=&	\varepsilon_{2}^{2a}  \varepsilon_{q}^c\varepsilon_{2q}^d\cdot(-1)^{b+f +gu}  .
			\end{eqnarray*}
			So we have   $ b+f +gu\equiv 0 \pmod 2$.

			\noindent\ding{224} Let us apply the norm $N_{\KK/k_5}=1+\tau_1\tau_2$, with $k_5=\QQ(\sqrt{q}, \sqrt{2p})$. We have 
			\begin{eqnarray*}
				N_{\KK/k_5}(\xi^2)&=&(-1)^a\cdot (-1)^b\cdot \varepsilon_{2pq}\cdot   (-1)^d\cdot \varepsilon_{  q}^d\cdot1 \cdot 1\cdot (-1)^g\cdot\varepsilon_{2p}^g. \\
				&=&	 \varepsilon_{2pq}^f (-1)^{a+b +d +g}\varepsilon_{  q}^d\varepsilon_{2p}^g.
			\end{eqnarray*}
			Thus,   $a+b +d  +g\equiv 0 \pmod 2$ and  $d = g$. Therefore,  $a=b$ and
			$$\xi^2=\varepsilon_{2}^a\varepsilon_{p}^a  \varepsilon_{2pq}^c  \sqrt{\varepsilon_{q}}^d\sqrt{\varepsilon_{2q}}^e\sqrt{\varepsilon_{pq}}^f
			\sqrt{ \varepsilon_{2p}}^d,$$
			
			\noindent\ding{224} Let us apply the norm $N_{\KK/k_6}=1+\tau_1\tau_3$, with $k_6=\QQ(\sqrt{p}, \sqrt{2q})$. We have 
			\begin{eqnarray*}
				N_{\KK/k_6}(\xi^2)&=& (-1)^a\cdot \varepsilon_{p}^{2a} \cdot \varepsilon_{2pq}^{2c}\cdot 1  \cdot (-1)^e\cdot \varepsilon_{2  q}^e  \cdot (-1)^{ud+d}\\
				&=&	\varepsilon_{p}^{2a} \varepsilon_{2pq}^{2c} (-1)^{a+e +ud+d}\varepsilon_{2  q}^e  .
			\end{eqnarray*}
			Thus, $ a+e +ud+d\equiv 0 \pmod 2$ and $e=0$. Then, $ a+ud+d\equiv 0 \pmod 2$. As the above discussions imply  $ a+f +du\equiv 0 \pmod 2$, then $f=d$. Therefore;
			$$\xi^2=\varepsilon_{2}^a\varepsilon_{p}^a  \varepsilon_{2pq}^c  \sqrt{\varepsilon_{q}}^d \sqrt{\varepsilon_{pq}}^d
			\sqrt{ \varepsilon_{2p}}^d,$$

			Since  by Lemma   \ref{lm expressions of units under cond 3}, $ \varepsilon_{2pq}$ is a square in $\KK$,   then we can put  
			$$\xi^2=\varepsilon_{2}^a\varepsilon_{p}^a     \sqrt{\varepsilon_{q}}^d \sqrt{\varepsilon_{pq}}^d
			\sqrt{ \varepsilon_{2p}}^d,$$
			where $ a+ud+d\equiv 0 \pmod 2$. Which completes the proof. 		  	 		
			
		\end{enumerate}
	\end{proof}

	\begin{theorem}\label{T_3_1_C8} Let $p\equiv 1\pmod{8}$ and $q\equiv3\pmod 8$ be two primes such that     $\genfrac(){}{0}{p}{q} =1$. 
		Put     $\KK=\QQ(\sqrt 2, \sqrt{p}, \sqrt{q} )$.  Assume furthermore that $2p(x+1)$ and $ p(v-1)$ are squares in $\mathbb{N}$, where $x$ and $v$ are defined in Lemma \ref{lm expressions of units under cond 3}.
		Then
		\begin{enumerate}[\rm 1.]
			\item Assume that $N(\varepsilon_{2p})=-1$.  We have
			\begin{enumerate}[\rm $\bullet$]
				
				\item The unit group of $\KK$ is :
				$$E_{\KK}=\langle -1,   \varepsilon_{2}, \varepsilon_{p} ,     \sqrt{\varepsilon_{2q}},\sqrt{ \varepsilon_{pq}} ,\sqrt{ \varepsilon_{2pq}}, \sqrt{\varepsilon_{2}\varepsilon_{p}\varepsilon_{2p}},\sqrt{\varepsilon_{p}^{\alpha}  \sqrt{\varepsilon_{q}}^{1+\gamma} \sqrt{\varepsilon_{pq}\varepsilon_{2pq}}^{\alpha}} \rangle,$$
				where $\alpha$, $\gamma\in \{0,1\}$ are such that $\alpha\not=\gamma$ and $\alpha =1$ if and only if   $\varepsilon_{p}   \sqrt{\varepsilon_{q}}  \sqrt{\varepsilon_{pq}\varepsilon_{2pq}} $ is a square in $\KK$
				\item  The $2$-class number  of  $\KK$ equals $ \frac{1}{2^{4-\alpha}} h_2(2p)h_2(pq)h_2(2pq)$.  
			\end{enumerate}
			\item  Assume that $N(\varepsilon_{2p})=1$ and let $a\in\{0,1\}$ be such that $a \equiv 1+u\pmod2$.   We have
			
			\begin{enumerate}[\rm $\bullet$]
				\item  The unit group of $\KK$ is :
				$$E_{\KK}=\langle -1,   \varepsilon_{2}, \varepsilon_{p},    \sqrt{\varepsilon_{2q}},  \sqrt{\varepsilon_{pq}} ,\sqrt{ \varepsilon_{2pq}}, \sqrt{\varepsilon_{2p}},   \sqrt{\varepsilon_{p}^\alpha     \sqrt{\varepsilon_{q}}^{1+\gamma}\sqrt{\varepsilon_{pq}\varepsilon_{2pq}}^\alpha }  
				\rangle$$
				where $\alpha$, $\gamma\in \{0,1\}$ are such that $\alpha\not=\gamma$ and $\alpha =1$ if and only if  $\varepsilon_{p}      \sqrt{\varepsilon_{q}} \sqrt{\varepsilon_{pq}\varepsilon_{2pq}} $ is a square in $\KK$.
				\item  The $2$-class number  of  $\KK$ equals $ \frac{1}{2^{4-\alpha}} h_2(2p)h_2(pq)h_2(2pq)$.
			\end{enumerate}
		\end{enumerate}
	\end{theorem}  	
	\begin{proof} 	\begin{enumerate}[\rm 1.] 
			\item  Assume that   $N(\varepsilon_{2p})=-1$.
			Note that $\{\varepsilon_{2}, \varepsilon_{p},	\sqrt{\varepsilon_{2}\varepsilon_{p}\varepsilon_{2p}}\}$, is  a fundamental system  of units of $k_1$.
			Using 	Lemma \ref{lm expressions of units under cond 3}, we show that $\{\varepsilon_{2}, \sqrt{\varepsilon_{q}}, \sqrt{\varepsilon_{2q}}\}$ and  $\{ \varepsilon_{2}, 	 \varepsilon_{pq},\sqrt{\varepsilon_{pq}\varepsilon_{2pq}}\}$ are respectively fundamental systems of units of $k_2$ and $k_3$.
			It follows that, 
			$$E_{k_1}E_{k_2}E_{k_3}=\langle-1,  \varepsilon_{2}, \varepsilon_{p} ,   \varepsilon_{pq} ,   \sqrt{\varepsilon_{q}}, \sqrt{\varepsilon_{2q}} ,\sqrt{ \varepsilon_{pq}\varepsilon_{2pq}}, \sqrt{\varepsilon_{2}\varepsilon_{p}\varepsilon_{2p}}\rangle.$$	
			
			Let  $\xi$ be an element of $\KK$ which is the  square root of an element of $E_{k_1}E_{k_2}E_{k_3}$. Therefore, we can assume that
			$$\xi^2=\varepsilon_{2}^a\varepsilon_{p}^b  \varepsilon_{pq}^c  \sqrt{\varepsilon_{q}}^d\sqrt{\varepsilon_{2q}}^e\sqrt{\varepsilon_{pq}\varepsilon_{2pq}}^f
			\sqrt{\varepsilon_{2}\varepsilon_{p}\varepsilon_{2p}}^g,$$
			where $a, b, c, d, e, f$ and $g$ are in $\{0, 1\}$.	
			
			\noindent\ding{224}  Let us start	by applying   the norm map $N_{\KK/k_2}=1+\tau_2$. By Tables \ref{norm q=3 p=1mod4q}  and \ref{tab3}  we have 
			\begin{eqnarray*}
				N_{\KK/k_2}(\xi^2)&=&
				\varepsilon_{2}^{2a}\cdot (-1)^b \cdot1\cdot \varepsilon_{q}^d\cdot \varepsilon_{2q}^e\cdot(-1)^{f}  \cdot (-1)^{gu}\varepsilon_{2}^g\\\
				&=&	\varepsilon_{2}^{2a}  \varepsilon_{q}^c\varepsilon_{2q}^d\cdot(-1)^{b +f+gu}\varepsilon_{2}^g .
			\end{eqnarray*}
			Thus, $ b + f+gu\equiv 0 \pmod 2$ and $g=0$. Therefore, $b=f$ and 
			$$\xi^2=\varepsilon_{2}^a\varepsilon_{p}^b  \varepsilon_{pq}^c  \sqrt{\varepsilon_{q}}^d\sqrt{\varepsilon_{2q}}^e\sqrt{\varepsilon_{pq}\varepsilon_{2pq}}^b,$$

			\noindent\ding{224} Let us apply the norm $N_{\KK/k_5}=1+\tau_1\tau_2$, with $k_5=\QQ(\sqrt{q}, \sqrt{2p})$.  Tables \ref{norm q=3 p=1mod4q}  and \ref{tab3} give:
			\begin{eqnarray*}
				N_{\KK/k_5}(\xi^2)&=&(-1)^a\cdot(-1)^b\cdot 1\cdot   (-1)^d\cdot \varepsilon_{  q}^d\cdot1 \cdot   \varepsilon_{2pq}^b\\
				&=&	      (-1)^{a+b +d  }\varepsilon_{  q}^d\varepsilon_{2pq}^b.
			\end{eqnarray*}
			Thus, $a+b +d=0\pmod8$ et  $   d=b$. Therefore,   $a=0$ and 
			$$\xi^2= \varepsilon_{p}^b  \varepsilon_{pq}^c  \sqrt{\varepsilon_{q}}^b\sqrt{\varepsilon_{2q}}^e\sqrt{\varepsilon_{pq}\varepsilon_{2pq}}^b.$$

			\noindent\ding{224} Let us apply the norm $N_{\KK/k_6}=1+\tau_1\tau_3$, with $k_6=\QQ(\sqrt{p}, \sqrt{2q})$. So by Tables \ref{norm q=3 p=1mod4q}  and \ref{tab3}  we have %We 
			\begin{eqnarray*}
				N_{\KK/k_6}(\xi^2)&=&   \varepsilon_{p}^{2b} \cdot 1\cdot1\cdot  
				(-1)^{ e}  \cdot \varepsilon_{2q}^e\cdot  \varepsilon_{2pq}^{b}\\
				&=&	 \varepsilon_{p}^{2b} \varepsilon_{2pq}^{b} (-1)^{e }\varepsilon_{2  q}^e .
			\end{eqnarray*}
			So $e =0$ and therefore, 	 
			$$\xi^2= \varepsilon_{p}^b  \varepsilon_{pq}^c  \sqrt{\varepsilon_{q}}^b \sqrt{\varepsilon_{pq}\varepsilon_{2pq}}^b.$$
			As $\varepsilon_{pq}$ is a square in $\KK$,  then we can put $\xi^2= \varepsilon_{p}^b    \sqrt{\varepsilon_{q}}^b \sqrt{\varepsilon_{pq}\varepsilon_{2pq}}^b$.  By applying $1+\tau_2\tau_3$,  $1+ \tau_3$ and $1+\tau_1 $, we deduce thing new.	So the first item.
			\item In this case we have: $$E_{k_1}E_{k_2}E_{k_3}=\langle-1,  \varepsilon_{2}, \varepsilon_{p} ,   \varepsilon_{pq} ,   \sqrt{\varepsilon_{q}}, \sqrt{\varepsilon_{2q}} ,\sqrt{ \varepsilon_{pq}\varepsilon_{2pq}}, \sqrt{ \varepsilon_{2p}}\rangle.$$	
			
			Let  $\xi$ be an element of $\KK$ which is the  square root of an element of $E_{k_1}E_{k_2}E_{k_3}$. Therefore, we can assume that
			$$\xi^2=\varepsilon_{2}^a\varepsilon_{p}^b  \varepsilon_{pq}^c  \sqrt{\varepsilon_{q}}^d\sqrt{\varepsilon_{2q}}^e\sqrt{\varepsilon_{pq}\varepsilon_{2pq}}^f
			\sqrt{\varepsilon_{2p}}^g,$$
			where $a, b, c, d, e, f$ and $g$ are in $\{0, 1\}$.	
			
			\noindent\ding{224}  Let us start	by applying   the norm map $N_{\KK/k_2}=1+\tau_2$. We have:
			\begin{eqnarray*}
				N_{\KK/k_2}(\xi^2)&=&
				\varepsilon_{2}^{2a}\cdot (-1)^b \cdot1\cdot \varepsilon_{q}^d\cdot \varepsilon_{2q}^e\cdot(-1)^f  \cdot (-1)^{gu}\\\
				&=&	\varepsilon_{2}^{2a}  \varepsilon_{q}^c\varepsilon_{2q}^d\cdot(-1)^{b+f +gu}  .
			\end{eqnarray*}
			Thus, $ b +f+gu\equiv 0 \pmod 2$.

			\noindent\ding{224} Let us apply the norm $N_{\KK/k_5}=1+\tau_1\tau_2$, with $k_5=\QQ(\sqrt{q}, \sqrt{2p})$. We have 
			%	\begin{eqnarray*}
			%	N_{\KK/k_5}(\xi^2)&=&(-1)^a\cdot (-1)^b\cdot 1\cdot   (-1)^d\cdot \varepsilon_{  %q}^d\cdot1 \cdot (-1)^f\cdot \varepsilon_{2pq}^f\cdot (-1)^g\cdot\varepsilon_{2p}^g. %\\
			%		&=&	 (-1)^{a+b +d +f+g}\varepsilon_{  q}^d\varepsilon_{2pq}^f \varepsilon_{2p}^g.
			%	\end{eqnarray*}
			%	Thus,   $a+b +d +f+g\equiv 0 \pmod 2$ and   $ d +f+g\equiv 0 \pmod 2$. Therefore,  $a=b$ %and
			%	$$\xi^2=\varepsilon_{2}^a\varepsilon_{p}^a  \varepsilon_{pq}^c  %\sqrt{\varepsilon_{q}}^d\sqrt{\varepsilon_{2q}}^e\sqrt{\varepsilon_{pq}\varepsilon_{2pq}}^%f  \sqrt{\varepsilon_{2p}}^g.$$
			\begin{eqnarray*}
				N_{\KK/k_5}(\xi^2)&=&(-1)^a\cdot (-1)^b\cdot 1\cdot   (-1)^d\cdot \varepsilon_{  q}^d\cdot1 \cdot \varepsilon_{2pq}^f\cdot (-1)^g\cdot\varepsilon_{2p}^g. \\
				&=&	 (-1)^{a+b +d  +g}\varepsilon_{  q}^d\varepsilon_{2pq}^f \varepsilon_{2p}^g.
			\end{eqnarray*}
			Thus,   $a+b +d  +g\equiv 0 \pmod 2$ and   $ d +f+g\equiv 0 \pmod 2$.  
			% $$\xi^2=\varepsilon_{2}^a\varepsilon_{p}^b  \varepsilon_{pq}^c  \sqrt{\varepsilon_{q}}^d\sqrt{\varepsilon_{2q}}^e\sqrt{\varepsilon_{pq}\varepsilon_{2pq}}^f  \sqrt{\varepsilon_{2p}}^g.$$

			\noindent\ding{224} Let us apply the norm $N_{\KK/k_6}=1+\tau_1\tau_3$, with $k_6=\QQ(\sqrt{p}, \sqrt{2q})$. We have 
			\begin{eqnarray*}
				N_{\KK/k_6}(\xi^2)&=& (-1)^a\cdot \varepsilon_{p}^{2b} \cdot 1\cdot 1  \cdot (-1)^e\cdot \varepsilon_{2  q}^e\cdot    \varepsilon_{2pq}^f\cdot (-1)^{ug+g}\\
				&=&	\varepsilon_{p}^{2b}\varepsilon_{2pq}^f (-1)^{a+e+ug+g}\varepsilon_{2  q}^e .
			\end{eqnarray*}
			Thus, $ a+e +ug+g\equiv 0 \pmod 2$ and $e=0$. Then, $ a+ug+g\equiv 0 \pmod 2$. Since $ b +f+gu\equiv 0 \pmod 2$ and $a+b +d  +g\equiv b +d  +ug     \equiv 0 \pmod 2$, this implies that $f=d$.  The equality $ d +f+g\equiv 0 \pmod 2$ gives $g=0$. Thus, $a=0$ and $b=f$. Therefore,
			$$\xi^2= \varepsilon_{p}^f  \varepsilon_{pq}^c  \sqrt{\varepsilon_{q}}^f \sqrt{\varepsilon_{pq}\varepsilon_{2pq}}^f   .$$
			
			As $\varepsilon_{pq}$ is a square in $\KK$,  then we can put  
			$\xi^2=\varepsilon_{p}^f     \sqrt{\varepsilon_{q}}^f \sqrt{\varepsilon_{pq}\varepsilon_{2pq}}^f $. 
			%	where $ a+ug+g\equiv 0 \pmod 2$. Which completes the proof. 	
		\end{enumerate}
	\end{proof}

	\begin{theorem}\label{T_3_1_C9}Let $p\equiv 1\pmod{8}$ and $q\equiv3\pmod 8$ be two primes such that     $\genfrac(){}{0}{p}{q} =1$. 
		Put     $\KK=\QQ(\sqrt 2, \sqrt{p}, \sqrt{q} )$.  Assume furthermore that $2p(x+1)$ and $ 2p(v+1)$ are squares in $\mathbb{N}$, where $x$ and $v$ are defined in Lemma \ref{lm expressions of units under cond 3}.
		Then
		\begin{enumerate}[\rm 1.]
			\item Assume that $N(\varepsilon_{2p})=-1$.  We have
			\begin{enumerate}[\rm $\bullet$]
				
				\item The unit group of $\KK$ is :
				$$E_{\KK}=\langle -1,   \varepsilon_{2}, \varepsilon_{p} ,    \sqrt{\varepsilon_{q}}, \sqrt{\varepsilon_{2q}},\sqrt{ \varepsilon_{pq}} ,\sqrt{ \varepsilon_{2pq}}, \sqrt{\varepsilon_{2}\varepsilon_{p}\varepsilon_{2p}} \rangle.$$
				\item  The $2$-class number  of  $\KK$ equals $ \frac{1}{2^{4}} h_2(2p)h_2(pq)h_2(2pq)$.  
			\end{enumerate}
			\item  Assume that $N(\varepsilon_{2p})=1$ and let $a\in\{0,1\}$ be such that $a \equiv 1+u\pmod2$.   We have
			
			\begin{enumerate}[\rm $\bullet$]
				\item  The unit group of $\KK$ is :
				$$E_{\KK}=\langle -1,   \varepsilon_{2}, \varepsilon_{p},   \sqrt{\varepsilon_{q}}, \sqrt{\varepsilon_{2q}},  \sqrt{\varepsilon_{pq}} ,\sqrt{ \varepsilon_{2pq}},  \sqrt{\varepsilon_{2}^{a\alpha}\varepsilon_{p}^{  u\alpha}     \sqrt{\varepsilon_{pq}\varepsilon_{2pq}}^{ \alpha}
					\sqrt{\varepsilon_{2p}}^{1+\gamma}}  
				\rangle$$
				where $\alpha$, $\gamma\in \{0,1\}$ are such that $\alpha\not=\gamma$ and $\alpha =1$ if and only if  $\varepsilon_{2}^{a }\varepsilon_{p}^{  u }     \sqrt{\varepsilon_{pq}\varepsilon_{2pq}} 
				\sqrt{\varepsilon_{2p}} $ is a square in $\KK$.
				\item  The $2$-class number  of  $\KK$ equals $ \frac{1}{2^{4-\alpha}} h_2(2p)h_2(pq)h_2(2pq)$.
			\end{enumerate}
		\end{enumerate}
	\end{theorem}  	
	\begin{proof} 	\begin{enumerate}[\rm 1.] 
			\item  Assume that   $N(\varepsilon_{2p})=-1$.
			Note that $\{\varepsilon_{2}, \varepsilon_{p},	\sqrt{\varepsilon_{2}\varepsilon_{p}\varepsilon_{2p}}\}$, is  a fundamental system  of units of $k_1$.
			Using 	Lemma \ref{lm expressions of units under cond 3}, we show that $\{\varepsilon_{2}, \sqrt{\varepsilon_{q}}, \sqrt{\varepsilon_{2q}}\}$ and  $\{ \varepsilon_{2}, 	 \varepsilon_{pq},\sqrt{\varepsilon_{pq}\varepsilon_{2pq}}\}$ are respectively fundamental systems of units of $k_2$ and $k_3$.
			It follows that, 
			$$E_{k_1}E_{k_2}E_{k_3}=\langle-1,  \varepsilon_{2}, \varepsilon_{p} ,   \varepsilon_{pq} ,   \sqrt{\varepsilon_{q}}, \sqrt{\varepsilon_{2q}} ,\sqrt{ \varepsilon_{pq}\varepsilon_{2pq}}, \sqrt{\varepsilon_{2}\varepsilon_{p}\varepsilon_{2p}}\rangle.$$	
			
			Let  $\xi$ be an element of $\KK$ which is the  square root of an element of $E_{k_1}E_{k_2}E_{k_3}$. Therefore, we can assume that
			$$\xi^2=\varepsilon_{2}^a\varepsilon_{p}^b  \varepsilon_{pq}^c  \sqrt{\varepsilon_{q}}^d\sqrt{\varepsilon_{2q}}^e\sqrt{\varepsilon_{pq}\varepsilon_{2pq}}^f
			\sqrt{\varepsilon_{2}\varepsilon_{p}\varepsilon_{2p}}^g,$$
			where $a, b, c, d, e, f$ and $g$ are in $\{0, 1\}$.	
			
			\noindent\ding{224}  Let us start	by applying   the norm map $N_{\KK/k_2}=1+\tau_2$. By Tables \ref{norm q=3 p=1mod4q}  and \ref{tab3}  we have 
			\begin{eqnarray*}
				N_{\KK/k_2}(\xi^2)&=&
				\varepsilon_{2}^{2a}\cdot (-1)^b \cdot1\cdot \varepsilon_{q}^d\cdot \varepsilon_{2q}^e\cdot1  \cdot (-1)^{gu}\varepsilon_{2}^g\\\
				&=&	\varepsilon_{2}^{2a}  \varepsilon_{q}^c\varepsilon_{2q}^d\cdot(-1)^{b +f+gu}\varepsilon_{2}^g .
			\end{eqnarray*}
			Thus, $ b +  gu\equiv 0 \pmod 2$ and $g=0$. Therefore, $b=0$ and 
			$$\xi^2=\varepsilon_{2}^a  \varepsilon_{pq}^c  \sqrt{\varepsilon_{q}}^d\sqrt{\varepsilon_{2q}}^e\sqrt{\varepsilon_{pq}\varepsilon_{2pq}}^f ,$$

			\noindent\ding{224} Let us apply the norm $N_{\KK/k_5}=1+\tau_1\tau_2$, with $k_5=\QQ(\sqrt{q}, \sqrt{2p})$.  Tables \ref{norm q=3 p=1mod4q}  and \ref{tab3} give:
			\begin{eqnarray*}
				N_{\KK/k_5}(\xi^2)&=&(-1)^a\cdot \varepsilon_{p}^{2b}\cdot 1\cdot   (-1)^d\cdot \varepsilon_{  q}^d\cdot1 \cdot   \varepsilon_{2pq}^f\\
				&=&	  \varepsilon_{p}^{2b}    (-1)^{a +d  }\varepsilon_{  q}^d\varepsilon_{2pq}^f.
			\end{eqnarray*}
			Thus,   $a  =  d=f$. Therefore,   
			$$\xi^2=\varepsilon_{2}^a  \varepsilon_{pq}^c  \sqrt{\varepsilon_{q}}^a\sqrt{\varepsilon_{2q}}^e\sqrt{\varepsilon_{pq}\varepsilon_{2pq}}^a .$$

			\noindent\ding{224} Let us apply the norm $N_{\KK/k_6}=1+\tau_1\tau_3$, with $k_6=\QQ(\sqrt{p}, \sqrt{2q})$. So by Tables \ref{norm q=3 p=1mod4q}  and \ref{tab3}  we have %We 
			\begin{eqnarray*}
				N_{\KK/k_5}(\xi^2)&=& (-1)^a\cdot \varepsilon_{p}^{2b} \cdot 1\cdot1\cdot  
				(-1)^{ e}  \cdot \varepsilon_{2q}^e\cdot  \varepsilon_{2pq}^{a}\\
				&=&	 \varepsilon_{p}^{2b} \varepsilon_{2pq}^{a} (-1)^{e+b}\varepsilon_{2  q}^e .
			\end{eqnarray*}
			So $a=e=0$. Thus, 	 
			$$\xi^2=    \varepsilon_{ pq}^c    	.$$
			So the first item.
			\item In this case we have: $$E_{k_1}E_{k_2}E_{k_3}=\langle-1,  \varepsilon_{2}, \varepsilon_{p} ,   \varepsilon_{pq} ,   \sqrt{\varepsilon_{q}}, \sqrt{\varepsilon_{2q}} ,\sqrt{ \varepsilon_{pq}\varepsilon_{2pq}}, \sqrt{ \varepsilon_{2p}}\rangle.$$	
			
			Let  $\xi$ be an element of $\KK$ which is the  square root of an element of $E_{k_1}E_{k_2}E_{k_3}$. Thus, we can assume that
			$$\xi^2=\varepsilon_{2}^a\varepsilon_{p}^b  \varepsilon_{pq}^c  \sqrt{\varepsilon_{q}}^d\sqrt{\varepsilon_{2q}}^e\sqrt{\varepsilon_{pq}\varepsilon_{2pq}}^f
			\sqrt{\varepsilon_{2p}}^g,$$
			where $a, b, c, d, e, f$ and $g$ are in $\{0, 1\}$.	
			
			\noindent\ding{224}  Let us start	by applying   the norm map $N_{\KK/k_2}=1+\tau_2$. We have:
			\begin{eqnarray*}
				N_{\KK/k_2}(\xi^2)&=&
				\varepsilon_{2}^{2a}\cdot (-1)^b \cdot1\cdot \varepsilon_{q}^d\cdot \varepsilon_{2q}^e\cdot1  \cdot (-1)^{gu}\\\
				&=&	\varepsilon_{2}^{2a}  \varepsilon_{q}^c\varepsilon_{2q}^d\cdot(-1)^{b +gu}  .
			\end{eqnarray*}
			Thus, $ b +gu\equiv 0 \pmod 2$.

			\noindent\ding{224} Let us apply the norm $N_{\KK/k_5}=1+\tau_1\tau_2$, with $k_5=\QQ(\sqrt{q}, \sqrt{2p})$. We have 
			\begin{eqnarray*}
				N_{\KK/k_5}(\xi^2)&=&(-1)^a\cdot (-1)^b\cdot 1\cdot   (-1)^d\cdot \varepsilon_{  q}^d\cdot1 \cdot \varepsilon_{2pq}^f\cdot (-1)^g\cdot\varepsilon_{2p}^g. \\
				&=&	 (-1)^{a+b +d+g}\varepsilon_{  q}^d\varepsilon_{2pq}^f \varepsilon_{2p}^g.
			\end{eqnarray*}
			Thus,   $a+b +d +g\equiv 0 \pmod 2$ and   $ d +f+g\equiv 0 \pmod 2$.

			\noindent\ding{224} Let us apply the norm $N_{\KK/k_6}=1+\tau_1\tau_3$, with $k_6=\QQ(\sqrt{p}, \sqrt{2q})$. We have 
			\begin{eqnarray*}
				N_{\KK/k_6}(\xi^2)&=& (-1)^a\cdot \varepsilon_{p}^{2b} \cdot 1\cdot 1  \cdot (-1)^e\cdot \varepsilon_{2  q}^e\cdot    \varepsilon_{2pq}^f\cdot (-1)^{ug+g}\\
				&=&	\varepsilon_{p}^{2b}\varepsilon_{2pq}^f (-1)^{a+e+ug+g}\varepsilon_{2  q}^e .
			\end{eqnarray*}
			Thus, $ a+e +ug+g\equiv 0 \pmod 2$ and $e=0$. Then, $ a+ug+g\equiv 0 \pmod 2$. 
			Since $ b +gu\equiv 0 \pmod 2$ and  $ a+ug+g\equiv 0 \pmod 2$, this implies 
			that $a+b+g\equiv 0 \pmod 2$.  As $a+b +d +g\equiv 0 \pmod 2$
			(resp. $ d +f+g\equiv 0 \pmod 2$), then $d=0$ (resp. $f=g$). Therefore,
			$$\xi^2=\varepsilon_{2}^a\varepsilon_{p}^{gu}  \varepsilon_{pq}^c  \sqrt{\varepsilon_{pq}\varepsilon_{2pq}}^g
			\sqrt{\varepsilon_{2p}}^g,$$
			with  $ a+ug+g\equiv 0 \pmod 2$. By applying the other norms we deduce nothing new. So we have the second item.
		\end{enumerate}
	\end{proof}

	\subsection{\bf The cases: $p\equiv 5$ or $3 \pmod8$  and $q\equiv 3 \pmod4$}$\,$\\
	The following theorem provide some families with odd class number and explicit unit groups.
	\begin{theorem}\label{easy cases} Let $p$ and $q$ be two primes.    
		Put     $\KK=\QQ(\sqrt 2, \sqrt{p}, \sqrt{q} )$. Then
		\begin{enumerate}[\rm 1.]
			\item If $p\equiv 3\pmod 8$,  $q\equiv7\pmod 8 $                and    $\left(\dfrac{p}{q}\right)=1 $,
			$$E_{\KK}=\langle-1,\varepsilon_{  2},    \sqrt{\varepsilon_{   q}}, \sqrt{\varepsilon_{   2q}},  \sqrt{\varepsilon_{   p}}, 
			%\sqrt{\varepsilon_{   2p}},   \sqrt{\varepsilon_{pq} } ,
			\sqrt{ \varepsilon_{2pq}},
			\sqrt[4]{ {\varepsilon_{   2q}}    {\varepsilon_{pq}\varepsilon_{2pq}} },
			\sqrt[4]{ \varepsilon_{  2}^2   {\varepsilon_{   q}} {\varepsilon_{   2q}}  {\varepsilon_{   p}} {\varepsilon_{   2p}} } 
			\;\rangle.$$
			\item If $p\equiv 3\pmod 8$,  $q\equiv7\pmod 8 $                and    $\left(\dfrac{p}{q}\right)=-1 $,
			$$E_{\KK}=\langle-1, \varepsilon_{  2}  ,   \sqrt{\varepsilon_{   q}}, \sqrt{\varepsilon_{   2p}},   \sqrt{\varepsilon_{pq} }, \sqrt{ \varepsilon_{2pq}}, 
			\sqrt[4]{ {\varepsilon_{q}}    {\varepsilon_{   p}}   {\varepsilon_{   2p}}    {\varepsilon_{pq}\varepsilon_{2pq}}},\sqrt[4]{\varepsilon_{  2}^2     {\varepsilon_{   2q}}    {\varepsilon_{pq}\varepsilon_{2pq}} }  \; \rangle.$$
			
			\item If  $p\equiv 5\pmod 8$, $q\equiv7\pmod 8$  and    $\left(\dfrac{p}{q}\right)=1 $,
			$$E_{\KK}=\langle-1,   \varepsilon_{2}, \varepsilon_{p},  \sqrt{\varepsilon_{q}}, \sqrt{\varepsilon_{2q}},
			\sqrt{ \varepsilon_{pq}}, \sqrt{\varepsilon_{2}\varepsilon_{p}\varepsilon_{2p}}, 
			\sqrt[4]{    {\varepsilon_{2q}}  {\varepsilon_{pq}\varepsilon_{2pq}} }   \rangle.$$
			\item  If $p\equiv 5\pmod 8$, $q\equiv7\pmod 8$     and     $\left(\dfrac{p}{q}\right)=-1$,
			$$E_{\mathbb{K}}=\langle-1,  \varepsilon_{2},  \varepsilon_{p},  \sqrt{\varepsilon_{q}},  \sqrt{\varepsilon_{2q}},  \sqrt{\varepsilon_{pq}},   \sqrt{\varepsilon_{2}\varepsilon_{p}\varepsilon_{2p}},
			\sqrt[4]{\varepsilon_{2}^2\varepsilon_{q}\varepsilon_{pq}\varepsilon_{2pq}}\rangle.$$
			\item If $p\equiv   5\pmod 8$,  $q\equiv  3\pmod 8$	    and    $\left(\dfrac{p}{q}\right)=1$, then 
			$$E_{\KK }=\langle-1,  \varepsilon_{2}, \varepsilon_{p},  \sqrt{\varepsilon_{q}}, \sqrt{\varepsilon_{2q}},
			\sqrt{ \varepsilon_{pq}}, \sqrt{\varepsilon_{2}\varepsilon_{p}\varepsilon_{2p}},
			\sqrt[4]{   \varepsilon_{p}^2 {\varepsilon_{2q}}{\varepsilon_{pq}\varepsilon_{2pq}}   } \rangle.$$
			\item If $p\equiv   5\pmod 8$,  $q\equiv  3\pmod 8$	    and    $\left(\dfrac{p}{q}\right)=-1$,
			$$E_{\KK }=\langle-1,  \varepsilon_{2}, \varepsilon_{p},  \sqrt{\varepsilon_{q}}, \sqrt{\varepsilon_{2q}},
			\sqrt{ \varepsilon_{pq}}, \sqrt{\varepsilon_{2}\varepsilon_{p}\varepsilon_{2p}},
			\sqrt[4]{\varepsilon_{2}^2\varepsilon_{p}^2\varepsilon_{q}\varepsilon_{pq}\varepsilon_{2pq}     }   \rangle.$$ 
			\item If $p\equiv    q\equiv  3\pmod 8$, then 
			$$E_{\KK}= \langle-1, \varepsilon_{  2} ,    \sqrt{\varepsilon_{   p}},  \sqrt{\varepsilon_{   2p}} ,     \sqrt{\varepsilon_{   q}},
			\sqrt{\varepsilon_{   pq}}  ,
			\sqrt[4]{ {\varepsilon_{   p}}      {\varepsilon_{   q}}      {\varepsilon_{   2pq}} },
			\sqrt[4]{ {\varepsilon_{   2p}}       {\varepsilon_{   2q}}    {\varepsilon_{   2pq}}}  \rangle.$$
		\end{enumerate}
		Furthermore, in all the above cases the class number of $\KK$ is odd.
	\end{theorem}
	\begin{proof}
		The points 4, 5 and 6 are proved in \cite{chemszekhniniaziziUnits1} and \cite{ChemsUnits9} respectively. 
		The proof of the rest demands very long computations as above, however we suggest to the   reader  proceed as in the proof of Theorem \ref{T_3_-1} or \cite[Theorem 2.5]{ChemsUnits9} to construct a detailed prove.
	\end{proof}

	\section{\bf Some   families of Fr\"ohlich triquadratic fields whose $2$-class groups  are  of type $(2,2)$ }$\,$ 	 	
	
	Now we can give some    families   of real triquadratic number fields whose $2$-class groups are of type $(2,2)$.
	
	\begin{theorem}\label{app1}    
		Let $p\equiv 1\pmod{8}$ and $q\equiv3\pmod 8$ be two primes such that     $\genfrac(){}{0}{p}{q} =1$. 
		Put     $\KK=\QQ(\sqrt 2, \sqrt{p}, \sqrt{q} )$.    Then the $2$-class group of $\KK$ is of type $(2,2)$ in the following cases:
		\begin{enumerate}[\rm 1.]
			\item $ h_2(pq)=h_2(2pq)=2\cdot h_2(2p)=4 $ and  $(x-1)$ or ($v-1)$  is not square in $\NN$.
			\item $N(\varepsilon_{2p})=-1$, $ h_2(pq)=h_2(2pq)=   h_2(2p)=4 $ and one of the  following conditions is satisfied: 
			\begin{enumerate}[\rm a.]
				\item $(x-1)$ and $2p(v+1)$ are squares in $\NN$,
				\item $(x-1)$ and $p(v-1)$ are squares in $\NN$,
				\item	$p(x-1)$ and $(v-1)$ are squares in $\NN$,
				\item $p(x-1)$ and $p(v-1)$  are squares in $\NN$,
				\item $2p(x+1)$ and $(v-1)$  are squares in $\NN$,
				\item  $2p(x+1)$ and $2p(v+1)$  are squares in $\NN$,
			\end{enumerate}
		\end{enumerate}
		where $x$ and $v$ are defined in Lemma \ref{lm expressions of units under cond 3}.
	\end{theorem}
	\begin{proof} We shall prove the first item and the reader can similarly prove the second one. So assume that we are in conditions of the first item. Note that by  \cite[Theorem 2]{kuvcera1995parity}, we have $N(\varepsilon_{2p})=1$. Using   Lemmas  \ref{wada's f.} and \ref{lm expressions of units under cond 3}, we get $h_2(k_3)=\frac{1}{4}q(k_3)\cdot4\cdot4=8$, where $k_3=\QQ(\sqrt 2, \sqrt{pq})$.Therefore, as by \cite{AM} the $2$-rank of the class group of $k_3$ equals $2$, the $2$-class group of $k_3$ is of type $(2,4)$.
		By the second items of Theorems  \ref{T_3_1_C2}, \ref{T_3_1-3}, \ref{T_3_1_C4}, \ref{T_3_1_C5}, \ref{T_3_1_C6}, \ref{T_3_1_C7}, \ref{T_3_1_C8} and \ref{T_3_1_C9}, we have 
		$h_2(\KK)= \frac{1}{2^{4-\alpha}} h_2(2p)h_2(pq)h_2(2pq)$, for some $\alpha \in\{0,1\}$. If we assume that $\alpha=0$, then  
		$h_2(\KK)= \frac{1}{2^{4 }}\cdot2\cdot4\cdot 4=2$. This implies that the $2$-class group of $\KK$ is cyclic, but this is impossible by class field theory and the fact that    $\KK/k_5$ is an unramified quadratic extension. Therefore, $h_2(\KK)=4$. Now, let us   show that the $2$-class group of $\KK$ is not cyclic.
		By \cite[Theorem 4.1, (iii)]{BaeYue}, $k_3$ admits three quadratic extensions $\KK_1=\KK$ and the conjugate extensions  $\KK_2:=k_3(\sqrt{\alpha_1^*})$ and $\KK_3$, all contained in 
		$\mathbb{Q}(\sqrt{2},\sqrt{p}, \sqrt{q},\sqrt{\alpha_1^*})$ which is an unramified extension of $k_3$ of degree $4$, where the $\alpha_1^*$ is (the element attached to $p$) defined in \cite[Theorem 4.1, (iii)]{BaeYue}.  So  by the group theoretic properties  given in  \cite[p. 110]{Benja} the $2$-class group of $\KK$ is not cyclic. Hence the $2$-class group of $\KK$ is of type $(2,2)$.
	\end{proof} 	
	
	%\begin{corollary} Keep the hypothesis of the previous theorem. Then,
	%	the length $2$-class field tower of the real biquadratic fields  $\mathbb Q(\sqrt{2p}, \sqrt{q})$ and  $\mathbb Q(\sqrt{2}, \sqrt{pq})$ equals at most $2$.
	% \end{corollary}  	
	%	\begin{proof}
	%	This follows from the fact that $\KK$ is an unramified quadratic extension of these  two fields.
	%	\end{proof}
	%\begin{question}
	%  Keep the hypothesis of the previous theorem.	Is the $2$-class groups $\mathbb Q(\sqrt{2p}, \sqrt{q})$ and  $\mathbb Q(\sqrt{2}, \sqrt{pq})$ of type $(2,4)$?
	% \end{question}	

	%	\begin{remark}
	%	To avoid making the article longer,  a continuation and further applications   will be published in   forthcoming papers.
	%	\end{remark} 	

	%	Let us close this paper by suggesting the following problem.
	%	\begin{problem}
	%		Let $d\not=1$ be a positive squarefree  integer (for example $d=pq$ or $ 2pq$). Many authors characterized  the small values (i.e. $2$, $4$ or $8$) of the $2$-class %number of the quadratic field $\mathbb{Q}(\sqrt{d})$ in terms of 
	%		the signs of Legendre symbols and the rational biquadratic residue symbols (e.g.  \cite{kaplan76}).   
	%		Is it possible to characterize   the small values of the $2$-class number of the quadratic field $\mathbb{Q}(\sqrt{d})$ in terms of $x$ (and $y$), where $x$, $y$ are the integers  such that
	%		$\varepsilon_m=x+y\sqrt{m}$   and $m$   iterate over the  divisors of $d$$\mathord{?} $ (or in terms of certain conditions analogous to those appearing in Theorem \ref{app1}$\mathord{?} $).
	
	%	\end{problem}  	

	\begin{remark}
		For a continuation of this work, we refer the reader to the paper \cite{CEH2024}  which is dedicated to  computation of   the unit group of the fields $ \QQ(\sqrt 2, \sqrt{p}, \sqrt{q} )$, where $p\equiv 1\pmod{8}$ and $q\equiv7\pmod 8$ are two prime numbers. 
	\end{remark}


\begin{thebibliography}{11}
		
		\bibitem{Az-00} A. Azizi,  {Sur la capitulation des $2$-classes d'id\'eaux de	$\kk =\QQ(\sqrt{2pq}, i)$,  o\`u $p\equiv-q\equiv1\pmod4$, }{  Acta. Arith.  { 94} (2000),  383-399.}%{ Zbl 0953.11033,  MR 1779950.}
		
		\bibitem{AM} A. Azizi  et A. Mouhib, {Sur le rang du $2$-groupe de classes de $\mathbb{Q}( \sqrt{m},\sqrt{d} )$ où $m=2$ ou un premier $p \equiv 1 \pmod{4}$ .} {Trans. Amer. Math. Soc., 353 (2001),  2741–2752.}
		
		
		\bibitem{azizitalbi} A. Azizi et M. Talbi, {Capitulation des $2$-classes d’idéaux de certains corps biquadratiques cycliques,}{ Acta Arith., { 127}  (2007), 231-248.}
		\bibitem{BaeYue}	S. Bae and Q. Yue,   Hilbert genus fields of real biquadratic fields. Ramanujan J 24, 161–181 (2011).
		\bibitem{Benja} E. Benjamin and C. Snyder, Classification of metabelian 2-groups $G$ with  $ G{ab}=(2,2n)$,  $ n\geq2$, and rank  $d(G')=2$;  Applications to real quadratic number fields, J. Pure Appl. Algebra, 223   (2019),   108-130.
		
		\bibitem{Bulant} M. Bulant, On the parity of the class number of the field  $\mathbb{Q}(\sqrt{p},\sqrt{q},\sqrt{r})$, J. Number Theory 68 (1998) 72-86.
		
		\bibitem{ChemsUnits9}	M. M. Chems-Eddin,  Unit groups of some multiquadratic number fields and 2-class groups. Period. Math. Hung. 84 (2022), 235–249. https://doi.org/10.1007/s10998-021-00402-0
		
		
		\bibitem{CEH2024} M. M. Chems-Eddin, M. B. T. El Hamam and  M. A.  Hajjami, On the unit group and the $2$-class number of $\mathbb{Q}(\sqrt{2},\sqrt{p},\sqrt{q})$, Ramanujan J., 2024 (To Appear).
		
		\bibitem{chemszekhniniaziziUnits1} M. M. Chems-Eddin,   A. Zekhnini and A. Azizi, \textit{Units and  $2$-class field towers  of some multiquadratic number fields},  Turk. J. Math., 44  (2020), 1466-1483.
		
		
		
		
		\bibitem{connor88}	P. E. Conner and J. Hurrelbrink,  {Class number parity, }{  Ser. Pure. Math. $8$,  World Scientific,  1988. }
		
		\bibitem{FrohlichCentral} A. Fr\"ohlich, Central Extensions, Galois Groups and Ideal Class Groups of Number Fields, Contemp. Math.,
		Vol. 24, Amer. Math. Soc., 1983.
		
		
		
		
		\bibitem{kaplan76}	P.~Kaplan,  {Sur le $2$-groupe des classes d'id\'eaux des corps quadratiques, }{ J. Reine. Angew. Math.,  283/284 (1976),  313-363.}
		

		
		\bibitem{kuvcera1995parity}	{R.~Ku{\v{c}}era}, On the parity of the class number of a biquadratic field,\textit{   J. Number Theory} \textbf{52} (1995),  43–52.
		
				\bibitem{Ku-50} S. Kuroda,  {\"Uber die Klassenzahlen algebraischer Zahlk\"orper,}{ Nagoya Math. J., 1 (1950),  1--10.}
		
		
		\bibitem{MouhibParity} A. Mouhib, On the parity of the class number of multiquadratic number fields, J. Number Theory 129 (2009),
		1205-1211.
		\bibitem{MouhibMouvahhedi} A. Mouhib and  A. Movahhedi, Sur le $2$-groupe de classes des corps mutiquadratiques réels, J. Théor. Nr. Bordx 17 (2005), 619-641.
		
		\bibitem{wada} H. Wada, {On the class number and the unit group of certain algebraic number fields.}{ J. Fac. Sci. Univ. Tokyo,  13 (1966),   201--209.}
		
		%	\bibitem{ZAT-15}	A. Zekhnini,  A. Azizi and M. Taous,  {On the generators of the $2$-class group of the field  $\QQ(\sqrt{q_1q_2p},  i)$ Correction to Theorem $3$ of [5]}{,  IJPAM,  Volume {\bf 103},   No. 1 (2015),  99-107.}
	\end{thebibliography}
\end{document}